\documentclass[10pt,reqno]{amsart}

\usepackage[margin=0.98in]{geometry}

\usepackage{amsmath,amssymb,amsthm,mathtools}
\numberwithin{equation}{section}

\usepackage[T1]{fontenc}
\usepackage[utf8]{inputenc}
\usepackage{lmodern}

\usepackage{enumitem}
\setlist[enumerate]{label=(\roman*),leftmargin=2.2em}
\setlist[itemize]{leftmargin=2.2em}

\usepackage{xcolor}
\usepackage[
  colorlinks=true,
  linkcolor=blue!50!black,
  citecolor=blue!50!black,
  urlcolor=blue!50!black
]{hyperref}
\usepackage[capitalise,nameinlink]{cleveref}

\theoremstyle{plain}
\newtheorem{theorem}{Theorem}[section]
\newtheorem{proposition}[theorem]{Proposition}
\newtheorem{lemma}[theorem]{Lemma}
\newtheorem{corollary}[theorem]{Corollary}

\theoremstyle{definition}
\newtheorem{definition}[theorem]{Definition}

\theoremstyle{remark}

\crefname{assumption}{Assumption}{Assumptions}
\Crefname{assumption}{Assumption}{Assumptions}

\newcommand{\N}{\mathbb N}
\newcommand{\Z}{\mathbb Z}
\newcommand{\Q}{\mathbb Q}
\newcommand{\R}{\mathbb R}
\newcommand{\C}{\mathbb C}

\newcommand{\sequence}[1]{\left\{ {#1} \right\}}

\newcommand{\innerprod}[2]{\left\langle{#1, #2}\right\rangle}

\newcommand{\set}[1]{\left\{#1\right\}}

\newcommand{\norm}[1]{\left\lVert#1\right\rVert}

\DeclareMathOperator{\spec}{Sp}

\DeclareMathOperator{\dist}{dist}

\newcommand\restr[2]{{
  \left.\kern-\nulldelimiterspace 
  #1 
  \vphantom{\big|} 
  \right|_{#2} 
  }}

\newcommand{\specac}{\spec_{\mathrm{ac}}}
\newcommand{\specsc}{\spec_{\mathrm{sc}}}
\newcommand{\specpp}{\spec_{\mathrm{pp}}}

\newcommand{\Cof}{\mathbf{Cof}}
\newcommand{\Tot}{\mathbf{Tot}}

\newcommand{\1}{\mathbf 1}

\title[Computational Bounds for Schr\"odinger Spectral Types]{Sharp Computational Bounds for\\Spectral Types of Schr\"odinger Operators}

\author{Matthew J. Colbrook}
\address{Centre for Mathematical Sciences, University of Cambridge, United Kingdom}
\email{m.colbrook@damtp.cam.ac.uk}

\author{George Coote}
\address{Centre for Mathematical Sciences, University of Cambridge, United Kingdom}
\email{gc602@cam.ac.uk}


\date{\today}

\begin{document}

\begin{abstract}
We prove sharp bounds for determining the spectral-type decomposition of Schr\"odinger operators in the spirit of Smale's program on the foundations of computation. For explicit one-dimensional self-adjoint Schr\"odinger operators $H=-{\mathrm d^2}/{\mathrm dx^2}+V$ on $L^2(\mathbb R)$, where $V\in C^\infty(\mathbb R;\mathbb R)$ is given by a finite description of all derivatives and derivative bounds, the pure point and absolutely continuous spectral sets cannot, in general, be recovered by any single limiting procedure. The singular continuous spectral set is strictly harder: it cannot, in general, be recovered by two nested limiting procedures. Analytic constructions of dichotomies realize the lower bounds: Gordon-type repetitions for pure point spectrum, high barriers for absolutely continuous spectrum, and an inverse spectral construction for singular continuous spectrum based on Riesz products, moment-killing perturbations, and a computational Gelfand--Levitan scheme. The finite-description framework also implies corresponding limitations on what can be certified in fixed formal systems (e.g., when used in computer-assisted proofs). Conversely, using wavelet-based certified computation, we prove matching upper bounds for broad classes of self-adjoint differential operators on $\mathbb R^d$ with coefficients of locally bounded variation and quantitative local variation control: two limits suffice for the pure point and absolutely continuous parts, and three for the singular continuous part. This provides a sharp hierarchy for spectral types. 
\end{abstract}

\maketitle

\newcommand{\marke}[1]{\textcolor[rgb]{0,0,1}{#1}}

\hypersetup{
	linkcolor={black!30!blue},
	citecolor={black!30!green},
	urlcolor={black!30!blue}
}

\section{Introduction}

Spectral type is a principal qualitative invariant of a quantum Hamiltonian. For a self-adjoint operator on a separable Hilbert space, the associated scalar spectral measures decompose into pure point, absolutely continuous, and singular continuous parts. These components are closely tied to quantum dynamics: pure point spectrum is typically associated with localization, absolutely continuous spectrum with transport and scattering, and singular continuous spectrum with intermediate phenomena such as fractal spectral measures and anomalous transport. For example, for the Fibonacci Hamiltonian, bounds on the fractal dimension of the spectrum have consequences for wavepacket propagation \cite{damanik2008fractal}. This point of view goes back at least to Ruelle's work and the RAGE theorem \cite{ruelle1969remark}, and remains central in the spectral theory of Schr\"odinger operators, ergodic models, random media, and quantum dynamics \cite{combes1993connections,last1996quantum,simon1990absence,damanik1999uniform,damanik2022one,damanik2025one, aizenman2015random,cycon2009schrodinger}.

Thus, deciding spectral type is a basic task in the mathematical classification of quantum phases and propagation. In practice, this task is approached through analytic, numerical, or computer-assisted methods applied to finite descriptions of Hamiltonians. The question considered in this paper is whether the spectral-type decomposition can be determined from an explicit finite description of the Hamiltonian. We prove that the answer is governed by a sharp hierarchy of limiting complexity.

Our lower bounds hold for one-dimensional self-adjoint Schr\"odinger operators
$$
    H=-\frac{\mathrm d^2}{\mathrm dx^2}+V
    \qquad\text{on }L^2(\mathbb R),
$$
with $V\in C^\infty(\mathbb R;\mathbb R)$. The input is not an arbitrary oracle containing infinitely much hidden information. Rather, it is a finite source code that computes all derivatives of $V$ and provides certified derivative bounds. Such finite descriptions include Hamiltonians given by formulas, specified by finite programs, or used in computer-assisted proofs.

Even in this setting, the pure point and absolutely continuous spectral sets cannot, in general, be recovered by a single limiting procedure. The singular continuous spectral set is strictly harder: it cannot, in general, be recovered by two nested limiting procedures. Thus the obstructions proved here are not artifacts of rough coefficients, non-effective input data, or pathological oracles. They are intrinsic to the spectral-type decomposition itself, and occur for smooth potentials given by finite descriptions.

The three lower bounds are realized by different spectral mechanisms. For pure point spectrum, the construction uses Gordon-type repetitions: in one case, arbitrarily long exact repetitions exclude eigenvalues; in the other, a compactly supported well remains and forces an eigenvalue in a fixed negative interval. For absolutely continuous spectrum, the construction uses high barriers: unbounded barriers on both half-lines destroy absolutely continuous spectrum, while the alternative leaves only finitely many barriers and hence a compactly supported perturbation of the free operator.

The singular continuous lower bound requires a more elaborate inverse spectral construction. We first build Riesz-product measures whose singular continuous component is controlled by a column condition of a binary matrix. We then add a computable absolutely continuous signed perturbation, supported away from the encoding interval, which kills all moments of the perturbation of the free spectral measure without changing this dichotomy. The moment cancellation makes the Gelfand--Levitan reconstruction flat at the origin. Reflecting the resulting half-line potential then gives a smooth whole-line Schr\"odinger operator whose singular continuous spectrum detects the encoded condition. A Weyl $m$-function estimate controls the reflected Neumann channel and prevents it from introducing additional singular continuous spectrum.

The lower bounds are sharp. We prove matching upper bounds for broad classes of self-adjoint differential operators on $\mathbb R^d$ whose coefficients have locally bounded variation and quantitative local variation control. The upper algorithms reduce spectral type to the certified computation of resolvent matrix elements in compactly supported wavelet bases. For these classes, pure point and absolutely continuous spectrum can be computed with two successive limiting procedures, while singular continuous spectrum can be computed with three. Hence spectral type is computable at a finite level of the hierarchy, but the number of limits cannot in general be reduced.

\subsection{Multiple Limits and Finite Descriptions}

Let $\Omega_{\mathrm{DS}}$ denote the class of bounded self-adjoint discrete Schr\"odinger operators on $\ell^2(\mathbb Z)$,
$$
    [A_Vx]_n=x_{n-1}+x_{n+1}+V(n)x_n.
$$
For this class (and more general bounded discrete operators), the algorithms in \cite{colbrook2019computing} use two limits:
$$
    \lim_{n_2\to\infty}\lim_{n_1\to\infty}
    \Gamma_{\diamond,n_2,n_1}(V)
    =
    \spec_\diamond(A_V),
    \qquad A_V\in\Omega_{\mathrm{DS}},\quad
    \diamond\in\{\mathrm{ac},\mathrm{pp}\},
$$
where $\spec_\diamond$ denotes the $\diamond$-spectrum. This result raises an optimality question: can the same spectral sets be computed with a single limit by a different algorithm? The answer is no. The two-limit height is necessary, even when allowing exact computation over $\mathbb R$.

The work \cite{colbrook2019computing} also gives a three-limit procedure for singular continuous spectrum of bounded self-adjoint discrete Schr\"odinger operators. However, the corresponding lower bound was obtained from a countable direct sum of discrete Schr\"odinger operators. One of the contributions of the present paper is to obtain the analogous lower bound inside a smooth one-dimensional continuum class, using inverse spectral embeddings rather than direct sums. In general, we obtain lower bounds in a substantially more rigid operator class, and hence obtain stronger results.

There is an important change in the input model. In \cite{colbrook2019computing}, the potential $V\in\ell^\infty(\mathbb Z)$ is given through the sequence $\{V(n)\}_{n\in\mathbb Z}$. Such an input may contain infinitely much noncomputable information. Here the input is a finite source code, as is typically the case in applications. The code describes the operator and supplies the quantitative information required by the algorithms. This is the formal setting for operators given by explicit formulas, finite programs, or constructions used in computer-assisted proofs. The lower bounds therefore show that, even for explicit smooth Schr\"odinger operators, resolving spectral type may require irreducibly many limiting stages. The technical task is to implement the lower-bound mechanisms uniformly within smooth spectral constructions.

We measure computational difficulty using the Solvability Complexity Index (SCI) hierarchy. Informally, a problem lies at a given level of this hierarchy if it can be solved by a prescribed number of successive limiting procedures, with the relevant form of error control. This framework was developed to classify problems in analysis and spectral theory \cite{colbrook2019compute,colbrook4,colbrook3,Jonathan_res, benartzi2020computing,ben2022complexity}; see \cref{sec:prev_work}.

\subsection{Summary of Theorems}

We now state the main results in a form sufficient for the introduction. Full definitions of the computational models and convergence notions are given in \cref{sec:prereq}.

\subsubsection{Lower Bounds on the Difficulty of Computing Spectral Type}

For a real-valued potential $V\in C^\infty(\mathbb R)$, let
$$
    H_{V,0}f=-f''+Vf,
    \qquad
    \operatorname{Dom}(H_{V,0})=C_c^\infty(\mathbb R).
$$
We define
$$
    \mathcal O_M^\infty
    =
    \left\{
        H_V:=\overline{H_{V,0}}:
        V\in C^\infty(\mathbb R;\mathbb R)
        \text{ and }H_{V,0}\text{ is essentially self-adjoint}
    \right\}.
$$
Thus every member of $\mathcal O_M^\infty$ is the self-adjoint closure of the Schr\"odinger expression initially defined on $C_c^\infty(\mathbb R)$; no additional boundary condition at infinity is part of the input.

The input describing $H_V$ is a smooth source code for $V$: for each derivative order $k$, computable point $x$, and accuracy $2^{-n}$, the code computes $V^{(k)}(x)$ to within $2^{-n}$, and for each $k$ and $R$ it supplies a certified bound for $ \sup_{|x|\le R}|V^{(k)}(x)|. $ We denote by $\mathsf{Code}_M^\infty$ the class of such codes whose decoded potential satisfies the essential-self-adjointness condition above. A valid code for $V$ is interpreted as a code for the uniquely determined operator $H_V$.

The following theorem combines \cref{thm:pp,thm:ac,thm:sc}.

\begin{theorem}[Lower bounds for smooth one-dimensional Schr\"odinger operators]
\label{thm:intro_smooth_lower_bounds}
For the Markov class $(\mathcal O_M^\infty,\mathsf{Code}_M^\infty)$, one has
$$
(\mathcal O_M^\infty,\mathsf{Code}_M^\infty,\specpp)\notin\Delta_2^M,
\qquad
(\mathcal O_M^\infty,\mathsf{Code}_M^\infty,\specac)\notin\Delta_2^M,
$$
and
$$
(\mathcal O_M^\infty,\mathsf{Code}_M^\infty,\specsc)\notin\Delta_3^M.
$$
\end{theorem}

Here $\Delta_k^M$ denotes the Markov class of problems solvable with $k-1$ successive limiting procedures. Thus membership in $\Delta_2^M$ corresponds to a single convergent sequence of finite outputs, while membership in $\Delta_3^M$ allows one additional nested limit. Concretely, the first two conclusions in \cref{thm:intro_smooth_lower_bounds} say that there is no finite procedure which, given a smooth source code $s$ for $H$ and an index $n\in\mathbb N$, outputs a finite set $\Gamma_n(s)\subset\mathbb Q+i\mathbb Q$ such that, for every $H\in\mathcal O_M^\infty$ and every valid source code $s$ for $H$, $\Gamma_n(s)\to\specpp(H)$ in the Attouch--Wets metric; the same statement holds with $\specpp$ replaced by $\specac$. The singular continuous conclusion says that even one inner limiting process is not enough: there is no finite procedure that, given $(n_2,n_1)\in\mathbb N^2$ and a source code $s$, produces a finite output $\Gamma_{n_2,n_1}(s)\subset\mathbb Q+i\mathbb Q$ such that, for every $H\in\mathcal O_M^\infty$ and every valid source code $s$ for $H$, the successive limits
$$
    \Gamma_{n_2}(s)=\lim_{n_1\to\infty}\Gamma_{n_2,n_1}(s),
    \qquad
    \lim_{n_2\to\infty}\Gamma_{n_2}(s)=\specsc(H)
$$
exist in the Attouch--Wets metric. The obstructions therefore persist even when the algorithm is given complete certified information about all derivatives of the potential.

\subsubsection{Upper Bounds Showing the Lower Bounds Are Sharp}

The lower bounds are complemented by constructive upper bounds. These show that the hierarchy obtained above is optimal.

We prove these upper bounds for a broad class of self-adjoint differential operators with coefficients of locally bounded variation. Put $Q_r=[-r,r]^d$. For $r>0$, let $\mathcal A_r$ be the algebra of pointwise-defined bounded Borel functions $f:Q_r\to\mathbb C$ having finite Hardy--Krause variation anchored at the upper corner $(r,\ldots,r)$. Thus, elements of $\mathcal A_r$ are actual functions, not equivalence classes modulo equality almost everywhere. For a complex-valued function, $\mathsf{TV}_{Q_r}(f)$ denotes the sum of the Hardy--Krause variations of its real and imaginary parts. We equip $\mathcal A_r$ with the norm
$$
    \|f\|_r
    =
    \|f\|_{\infty,Q_r}
    +(3^d+1)\mathsf{TV}_{Q_r}(f),
    \qquad
    \|f\|_{\infty,Q_r}:=\sup_{x\in Q_r}|f(x)|.
$$
A coefficient $a:\mathbb R^d\to\mathbb C$ has locally bounded variation if $a|_{Q_r}\in\mathcal A_r$ for every $r>0$. If $f\in C^d(Q_r;\mathbb C)$, then
$$
    \mathsf{TV}_{Q_r}(f)
    \le
    \sum_{\emptyset\ne J\subseteq\{1,\ldots,d\}}
    (2r)^{|J|}
    \bigl(
        \|\partial_J\operatorname{Re}f\|_{\infty,Q_r}
        +
        \|\partial_J\operatorname{Im}f\|_{\infty,Q_r}
    \bigr),
    \qquad
    \partial_J:=\prod_{j\in J}\partial_j.
$$
Consequently, every $C^d$ function on $\mathbb R^d$ has locally bounded Hardy--Krause variation. The class also contains many discontinuous coefficients.

Let $\mathcal D_{N,d}^{\mathrm{BV}}$ be the class of self-adjoint operators on $L^2(\mathbb R^d)$ of the form
$$
    T=\sum_{|\alpha|\le N}a_\alpha\partial^\alpha
$$
which have coefficients $a_\alpha$ of locally bounded variation, and for which $C_c^\infty(\mathbb R^d)$ is a core. When an operator is presented to the algorithm, each coefficient $a_\alpha$ is supplied as a designated pointwise representative $a_\alpha:\mathbb R^d\to\mathbb C$. The differential operator depends only on the almost-everywhere equivalence class of this representative, whereas every point evaluation and every local $\mathcal A_r$-bound below refers to the same designated representative.

The algorithm is given point evaluations of the coefficients at rational points and, in addition, local variation bounds: a sequence $\{c_j(T)\}_{j\in\mathbb N}$ such that
$$
    \max_{|\alpha|\le N}\|a_\alpha\|_r
    \le c_{\lceil r\rceil}(T),
    \qquad r>0.
$$
We denote this information by $\mathsf{Eval}_{N,d}^{\mathrm{BV}}$.

\begin{theorem}[Upper bounds for spectral type]
\label{thm:main_pos}
For the class $(\mathcal D_{N,d}^{\mathrm{BV}},\mathsf{Eval}_{N,d}^{\mathrm{BV}})$, one has
$$
(\mathcal D_{N,d}^{\mathrm{BV}},
 \mathsf{Eval}_{N,d}^{\mathrm{BV}},
 \specpp)
\in\Sigma_2^A,
\qquad
(\mathcal D_{N,d}^{\mathrm{BV}},
 \mathsf{Eval}_{N,d}^{\mathrm{BV}},
 \specac)
\in\Sigma_2^A,
$$
and
$$
(\mathcal D_{N,d}^{\mathrm{BV}},
 \mathsf{Eval}_{N,d}^{\mathrm{BV}},
 \specsc)
\in\Sigma_3^A.
$$
The same classifications hold in the corresponding Markov formulation, where the coefficients and local variation bounds are supplied by source code.
\end{theorem}

Here $\Sigma_k^A$ denotes the arithmetic-model class of problems solvable by a tower of $k$ successive limiting procedures with one-sided error control. Thus $\Sigma_2^A$ corresponds to two successive limits, while $\Sigma_3^A$ allows three. The superscript $A$ refers to the arithmetic model of the SCI hierarchy, in which algorithms are Blum--Shub--Smale machines \cite{BCSS} with access to the specified evaluations and local bounds. In the Markov formulation, the same data are supplied by finite source codes, and the same construction is implemented by ordinary Turing machines.

The comparison with the lower-bound class is useful. Let $H_V\in\mathcal O_M^\infty$, and set
$$
    H_V=(-1)\partial^2+0\partial+V.
$$
Then $H_V\in\mathcal D_{2,1}^{\mathrm{BV}}$: by definition, $C_c^\infty(\mathbb R)$ is a core, and in one dimension
$$
    \mathsf{TV}_{[-r,r]}(V)
    =\int_{-r}^{r}|V'(x)|\,\mathrm dx
    \le 2r\sup_{|x|\le r}|V'(x)|.
$$
If the smooth source code supplies certified bounds $C_{k,j}$ for $\sup_{|x|\le j}|V^{(k)}(x)|$, then
$$
    c_j(H_V):=\max\{1,C_{0,j}+8jC_{1,j}\}
$$
is a certified local $\mathcal A_r$-bound for every $r\le j$, because $3^1+1=4$. Point values of the coefficients are obtained from the same smooth source code. Hence there is a uniform computable translation from $\mathsf{Code}_M^\infty$ to $\mathsf{Eval}_{2,1}^{\mathrm{BV}}$. The lower-bound problems therefore reduce to the Markov versions of the upper problems. Together with \cref{thm:intro_smooth_lower_bounds,thm:main_pos}, this proves that two limits for pure point and absolutely continuous spectrum, and three limits for singular continuous spectrum, are both sufficient and in general necessary.

\subsubsection{Consequences for Formal Certification}

As a further consequence of the finite-description lower bounds, we obtain limitations on formal certification of spectral type. Since the lower-bound operators are finitely described, spectral inclusion and exclusion statements for them can be represented by arithmetical sentences.

Let $\mathfrak T$ be an effectively axiomatized, arithmetically sound theory, and let $I\subset\mathbb R$ be a non-empty open interval with rational endpoints. In \cref{thm:main_non_prov} we prove that, for each $\diamond\in\{\mathrm{pp},\mathrm{ac},\mathrm{sc}\}$, there exist self-adjoint one-dimensional Schr\"odinger operators $ H_\diamond^+,\ H_\diamond^-\in\mathcal O_M^\infty, $ with smooth real-valued potentials given by source codes in $\mathsf{Code}_M^\infty$, such that $ \spec_\diamond(H_\diamond^+)\cap I\ne\emptyset $ is true but not provable in $\mathfrak T$, while $ \spec_\diamond(H_\diamond^-)\cap\overline I=\emptyset $ is true but not provable in $\mathfrak T$.

Taking $\mathfrak T=\mathsf{ZFC}$, this gives true spectral-type statements for finitely described smooth Schr\"odinger operators which are not provable in $\mathsf{ZFC}$, conditional on the arithmetical soundness of $\mathsf{ZFC}$.

We also obtain a relative hierarchy of non-provability. For pure point and absolutely continuous spectrum, the obstruction persists after adjoining the relevant $\Pi_1^0$ truth predicates. For singular continuous spectrum, the corresponding obstruction persists one level higher, after adjoining the relevant $\Pi_2^0$ truth predicates. This mirrors the SCI hierarchy: singular continuous spectrum is strictly harder.

\subsection{Connections with Previous Work}
\label{sec:prev_work}

The paper connects with several strands of spectral theory, computational spectral theory, and mathematical physics.

Multiple-limit phenomena occur throughout mathematics. They can be traced back at least to the dynamical study of iterative rational maps for polynomial root-finding. Smale asked whether there exists a universally convergent iterative rational map for finding polynomial zeros \cite{smale_question}. McMullen showed that such an algorithm exists for cubic polynomials but not for higher-degree polynomials \cite{McMullen1,mcmullen1988braiding}. Doyle and McMullen later showed that degrees four and five can be handled using multiple successive limits, while degree six cannot \cite{Doyle_McMullen}. Another central example is the classical spectral problem, with roots in Szeg{\H o}'s work on finite-section approximations \cite{Szego} and Schwinger's finite-dimensional approximations to quantum systems \cite{Schwinger}. This problem asks for an algorithm that computes $\spec(A)$ for every $A\in\mathcal B(\ell^2(\mathbb N))$, given its matrix entries. It was shown in \cite{Hansen_JAMS} that this problem admits a three-limit algorithm, and in \cite{ben2015can} that this is optimal. The Solvability Complexity Index (SCI) hierarchy (\cref{def:scihierarchy}) was developed precisely to measure this kind of intrinsic limiting complexity and to distinguish algorithms with different forms of error control \cite{colbrook2019compute,colbrook4,colbrook3,Jonathan_res, benartzi2020computing,ben2022complexity}. Multiple-limit phenomena are discussed systematically in the book \cite{colbrook2026infinite}.

The finite-section method has long been a central tool in computational spectral theory. Even when it converges for the computation of spectra, however, it typically does not by itself provide error control or verification. Examples include work of B{\"o}ttcher, Brunner, Iserles, and N{\o}rsett \cite{Arieh2}; B{\"o}ttcher, Grudsky, and Iserles \cite{bottcher_grudsky_iserles_2011}; Marletta \cite{Marletta_pollution}; and Marletta and Scheichl \cite{marletta2012eigenvalues}. These works also discuss mechanisms by which finite-section methods can fail. For the more delicate spectral features studied in the present paper, such difficulties are amplified by phenomena such as spectral pollution. For the Maxwell curl--curl operator, Warburton and Embree show that the stabilization parameter in a local discontinuous Galerkin discretization separates the discrete spectrum into a part approximating the curl-conforming finite element spectrum and a spurious part whose eigenvalues can be made arbitrarily large \cite{warburton2006penalty}. See also work of B{\"o}ttcher \cite{Albrecht_Fields,Boettcher1994}, B{\"o}ttcher and Silbermann \cite{Bottcher,bottcher2006analysis}, Laptev and Safarov \cite{Laptev}, and Brown \cite{brown2007quasi,Brown_2006,Brown_Memoars}. B{\"o}ttcher and Silbermann \cite{MR721280} were among the pioneers in combining spectral computation with $C^*$-algebraic methods.

Olver, Townsend, and Webb have developed influential frameworks for infinite-dimensional numerical linear algebra and for computational methods with infinite data structures, contributing both theoretical foundations and practical algorithms \cite{Olver_Townsend_Proceedings,olver2013fast,webb_thesis}. For example, Webb and Olver use connection coefficients between families of orthogonal polynomials to compute spectral measures of Jacobi operators arising as compact perturbations of Toeplitz operators; in this setting, the computation of the eigenvalues and absolutely continuous spectrum lies in $\Delta_0^G$ \cite{webb2017spectra}.

The SCI perspective is also closely connected with limiting procedures in spectral approximation. For example, Zworski \cite{Zworski1} showed that scattering resonances of Schr\"odinger operators $-\Delta+V$, with $V$ bounded and compactly supported, arise as limits of eigenvalues of $-\Delta+V-i\varepsilon x^2$ as $\varepsilon\to0^+$. A complementary instability of spectral geometry and type occurs for continuum limit-periodic Schr\"odinger operators. Damanik, Fillman, and Lukic showed that uniformly small perturbations introduced on successively longer periodic scales can lead from periodic approximants with band spectrum and purely absolutely continuous spectral type to a uniform limit with zero-measure Cantor spectrum and purely singular continuous spectral type. The latter behavior is residual among continuous limit-periodic potentials in the uniform topology \cite{damanik2017limitperiodic}. Periodic approximation also provides a computational route for structured aperiodic models. Puelz, Embree, and Fillman develop an $O(K^2)$ method for computing the spectrum of a period-$K$ Jacobi operator and apply it to periodic approximations of the Fibonacci, period-doubling, and Thue--Morse models \cite{puelz2015spectral}. More generally, spectral pollution and spectral invisibility require precision about the convergence topology and about the number of limits used. The present paper shows that such issues are not restricted to the spectral set itself: they also occur, in a sharper form, for the spectral-type decomposition.

SCI methods are also closely connected to computer-assisted proof. Results in the hierarchy can be interpreted as statements about the type of certified limiting procedure required to establish a spectral claim \cite{AIM}. Fefferman and Seco's proof of the Dirac--Schwinger conjecture \cite{fefferman1990,fefferman1992,fefferman1996interval}, concerning the asymptotic behavior of the ground state of a family of Schr\"odinger operators, may be viewed in this light as a $\Sigma_1^A$ result. Rigorous computer-assisted methods have also appeared in spectral-theoretic work such as \cite{brown2010eigenvalue,bogli2014guaranteed}. The finite-description aspect of our constructions gives complementary consequences for spectral-type certification.

There is also a growing literature on spectral enclosures and certified spectral computation. Chandler--Wilde, Chonchaiya, and Lindner \cite{chandler2024spectral} proved that banded operators on $\ell^2(\mathbb Z)$ admit spectral covers under suitable hypotheses, yielding algorithms for computing spectra and rigorous spectral exclusions. We also refer to work of Ben-Artzi, Marletta, and R\"osler on computing scattering resonances via spectral covers \cite{Jonathan_res,benartzi2020computing}. The comparison with our results is instructive: spectral covers concern inclusion and exclusion for the spectral set, whereas the present paper shows that the measure-theoretic type carried by the spectrum can be strictly harder to resolve. Analogous certified enclosures for spectral type are subject to the sharper limiting-complexity obstructions proved here.

\subsection{Outline of the Paper}

In \cref{sec:prereq} we recall the spectral-type decomposition, the SCI hierarchy, the Markov model, and the reduction principles used in the lower-bound arguments. The three lower bounds are then proved separately. In \cref{sec:pp} we prove the pure point lower bound using Gordon-type repetitions. In \cref{sec:ac} we prove the absolutely continuous lower bound using high barriers. In \cref{sec:sc} we prove the singular continuous lower bound by constructing Riesz-product spectral measures, killing moments, and applying a computational version of inverse spectral theory through the Gelfand--Levitan equation. In \cref{sec:upper} we prove the matching upper bounds for broad classes of differential operators by computing resolvent matrix elements in compactly supported wavelet bases. Finally, \cref{sec:prov} provides the formal certification consequences of the Markov lower bounds.

\section{Preliminaries}
\label{sec:prereq}
\subsection{Spectral Type}
We recall the decomposition of a self-adjoint operator into spectral types. Let $\mu$ be a finite positive Borel measure on $\R$. We call $\mu$ \emph{continuous} if it has no atoms: $\mu(\{x\}) = 0$ for each $x \in \R$. We say that $\mu$ is \emph{concentrated} on a Borel set $C$ if $\mu(C^c) = 0$. A finite Borel measure is classified as follows:
\begin{itemize}
	\item \emph{pure point} or \emph{purely atomic} if it is concentrated on a countable set;
	\item \emph{singular continuous} if it is continuous and concentrated on a set of Lebesgue measure zero;
	\item \emph{absolutely continuous} if $\mu(C) = 0$ for every Borel set $C$ of Lebesgue measure zero.
\end{itemize}
By the Radon--Nikodym theorem, an absolutely continuous measure is determined by an $L^1$-density, whereas a pure point measure is a countable sum of weighted Dirac masses. The remaining component is singular continuous, as expressed by the following decomposition.

\begin{proposition}[e.g., {\cite[Theorem 4.3.2]{cohn2013measure}}]
\label{prop:lebesgue_decomp}
Let $\mu$ be a finite positive Borel measure on $\R$. Then there exist unique finite positive Borel measures $\mu_{\mathrm{pp}}$, $\mu_{\mathrm{sc}}$ and $\mu_{\mathrm{ac}}$ such that $\mu_{\mathrm{pp}}$ is pure point, $\mu_{\mathrm{sc}}$ is singular continuous, $\mu_{\mathrm{ac}}$ is absolutely continuous and $\mu = \mu_{\mathrm{pp}} + \mu_{\mathrm{sc}} + \mu_{\mathrm{ac}}$.
\end{proposition}

We define the spectrum, or closed support, of a positive Borel measure $\mu$ on $\R$ by
$$
\spec(\mu) = \set {\lambda \in \R : \mu((\lambda - \epsilon, \lambda + \epsilon)) > 0 \text { for all } \epsilon > 0}.
$$
Let $\mathcal H$ be a separable Hilbert space and let $A$ be a self-adjoint operator with projection-valued spectral measure $E_A$. For $\psi \in \mathcal H$, define the finite measure
$$
\mu_\psi(S) = \innerprod {E_A(S) \psi} \psi, \quad S \in \mathcal B(\R).
$$
This is the unique finite Borel measure satisfying the resolvent identity
$$
\innerprod {(A - z I)^{-1} \psi} \psi = \int_\R \frac {\mathrm d \mu_\psi(\lambda)} {\lambda - z} \quad \forall z \in \C \setminus \R.
$$
For a more detailed discussion, see \cite[Section 3.1]{terschl}. Throughout, the Hilbert-space inner product is linear in the first argument. Applying \cref{prop:lebesgue_decomp} to the scalar spectral measures associated with $A$, we obtain the decomposition
$$
\mathcal H = \mathcal H_{\mathrm{ac}} \oplus \mathcal H_{\mathrm{sc}} \oplus \mathcal H_{\mathrm{pp}},
$$
where
$$
\mathcal H_{\mathrm{ac}} = \set {\psi \in \mathcal H : \mu_\psi \text { is ac}},\qquad
\mathcal H_{\mathrm{sc}} = \set {\psi \in \mathcal H : \mu_\psi \text { is sc}},\qquad
\mathcal H_{\mathrm{pp}} = \set {\psi \in \mathcal H : \mu_\psi \text { is pp}}.
$$
These sets are closed reducing subspaces for $A$, and the corresponding spectral projections give the orthogonal decomposition. For $\diamond \in \set {\mathrm{ac}, \mathrm{sc}, \mathrm{pp}}$, let $A_\diamond$ denote the restriction of $A$ to $\operatorname{Dom}(A)\cap \mathcal H_\diamond$, and define $\spec_\diamond(A) = \spec(A_\diamond)$. We use the convention that the spectrum of the operator on the zero Hilbert space is empty. So $\spec_\diamond(A) = \emptyset$ precisely when $\mathcal H_\diamond = \set 0$. An operator has \emph{purely absolutely continuous spectrum} if $\mathcal H = \mathcal H_{\mathrm{ac}}$. Purely singular continuous spectrum and pure point spectrum are defined analogously.

\subsection{SCI Hierarchy}\label{sec:SCI_hierarchy}

The Solvability Complexity Index (SCI) classifies a computational problem by the number of successive limits needed to solve it. It also provides the language for proving that this number is optimal. For a spectral-set problem, the target is a closed subset of $\C$, possibly unbounded and with non-trivial accumulation. We approximate it by finite subsets of $\mathbb Q+i\mathbb Q$, using the Attouch--Wets metric to measure convergence.

\begin{definition}[Attouch--Wets metric {\cite[Definition 3.1.2]{beer1993topologies}}]
Let $\mathrm{CL}(\C)$ be the set of non-empty closed subsets of $\C$. We define the Attouch--Wets metric by
$$
d_{\mathrm{AW}_0}(A, B) = \sum_{n = 1}^\infty 2^{-n} \min \set {1, \max_{|z| \le n} |\dist(z, A) - \dist(z, B)|} \qquad \forall A, B \in \mathrm{CL}(\C).
$$
\end{definition}

The following standard characterization is the local analogue of Hausdorff convergence.

\begin{proposition}
\label{prop:awconverge}
For each non-empty compact set $K \subseteq \C$, define
$$
d_K(C_1,C_2)=\sup_{z\in K}|\dist(z,C_1)-\dist(z,C_2)|.
$$
Let $\sequence {C_n}_{n \in \N}\subset\mathrm{CL}(\C)$ and $C\in \mathrm{CL}(\C)$. We have $d_{\mathrm{AW}_0}(C_n, C) \to 0$ if and only if for each non-empty compact set $K \subseteq \C$ we have $d_K(C_n, C) \to 0$ as $n \to \infty$.
\end{proposition}

Thus $d_{\mathrm{AW}_0}(C_n,C)\to0$ precisely when the distance functions converge uniformly on compact subsets of $\C$. This excludes spectral invisibility and persistent spectral pollution on bounded sets, while allowing points of $C_n$ to escape to infinity.

To define the SCI hierarchy, we first specify its underlying notion of a computational problem.

\begin{definition}
\label{def:compproblem}
A computational problem is a quadruple $(\Omega,\Lambda,\mathcal M,\Xi)$, where:
\begin{itemize}
    \item $\Omega$ is the \emph{primary set}, namely the class of inputs; in this paper, it is typically a family of Schr\"odinger operators;
    \item $\Lambda$ is the \emph{evaluation set}, consisting of the maps $\Omega\to\C$ available to an algorithm; for example, if $-\Delta+V\in\Omega$, these may include the maps $-\Delta+V\mapsto V(q)$ for $q\in\mathbb Q$;
    \item $\mathcal M$ is the metric space in which the approximations converge;
    \item $\Xi:\Omega\to\mathcal M$ is the \emph{problem function}; for a spectral problem, one may have $\Xi(A)=\spec(A)$.
\end{itemize}
We require the evaluation set to separate any two inputs separated by $\Xi$: if $A,B\in\Omega$ and $\Xi(A)\ne\Xi(B)$, then there exists $f\in\Lambda$ such that $f(A)\ne f(B)$.
\end{definition}

Unlike the whole spectrum of a self-adjoint operator, the spectral sets $\spec_{\diamond}$ may be empty. For example, the free Laplacian has purely absolutely continuous spectrum on $[0, \infty)$, so its singular continuous and pure point spectra are empty. Therefore, we will take our $\mathcal M$ to be
$$
\mathrm{CL}_\ast(\C) = \mathrm{CL}(\C) \cup \set \emptyset.
$$
We will use the convention that $\dist(x, \emptyset) = \infty$ for each $x \in \C$. We define the homeomorphism $\theta : [0, \infty] \to [0, 1]$ by $\theta(t) = (1 + t)^{-1}$ for each $t \in [0, \infty]$. In particular, $\theta(\infty) = 0$. We set
$$
d_{\mathrm{AW}}(C, D) = \sum_{m = 1}^\infty 2^{-m} \sup_{|z| \le m} |\theta(\dist(z, C)) - \theta(\dist(z, D))|\qquad\forall C,D\in\mathrm{CL}_\ast(\C).
$$
This extension is compatible with the usual Attouch--Wets topology:
\begin{proposition}
The topology on $\mathrm{CL}(\C)$ induced by $d_{\mathrm{AW}}$ coincides with the topology induced by $d_{\mathrm{AW}_0}$.
\end{proposition}
Since $\theta$ is $1$-Lipschitz and $|\theta(\dist(z,C))-\theta(\dist(z,D))|\le1$, we have $d_{\mathrm{AW}}(C,D)\le d_{\mathrm{AW}_0}(C,D)$. Conversely, suppose that $C_n,C\in\mathrm{CL}(\mathbb C)$ and $d_{\mathrm{AW}}(C_n,C)\to0$. Fix a compact set $K$. Since $C\ne\emptyset$, the function $z\mapsto\dist(z,C)$ is bounded on $K$, and hence $\theta(\dist(z,C))\ge a$ on $K$ for some $a>0$. For all sufficiently large $n$, $\theta(\dist(z,C_n))\ge a/2$ on $K$. Since $\theta^{-1}$ is Lipschitz on $[a/2,1]$, uniform convergence of the transformed distance functions on $K$ implies uniform convergence of the distance functions on $K$. Finally, $C_n\to\emptyset$ if and only if, for every compact set $K\subseteq\C$, one has $C_n\cap K=\emptyset$ for all sufficiently large $n$. In particular, $C_n=\{n\}$ converges to $\emptyset$.

When $\mathcal M$ is fixed, we suppress it and denote the computational problem by $(\Omega, \Lambda, \Xi)$.

An algorithm for a computational problem is defined as follows.

\begin{definition}
Let $(\Omega, \Lambda, \mathcal M, \Xi)$ be a computational problem. A general algorithm $\Gamma:\Omega\rightarrow \mathcal{M}$ for $(\Omega, \Lambda, \mathcal M, \Xi)$ is a map that satisfies the following property. For each $A\in\Omega$, there exists a finite subset of evaluations $\Lambda_\Gamma(A) \subset\Lambda$ such that if $B\in\Omega$ with $f(A)=f(B)$ for every $f\in\Lambda_\Gamma(A)$, then $\Lambda_\Gamma(A)=\Lambda_\Gamma(B)$ and $\Gamma(A)=\Gamma(B)$.
\end{definition}

A general algorithm uses only finitely many evaluations on each input. For the infinite-dimensional problems considered here, one therefore asks first whether there is a sequence $\sequence{\Gamma_n}_{n\in\N}$ such that $\Gamma_n(A)\to\Xi(A)$ in $(\mathcal M,d)$ for every $A\in\Omega$. Even this need not suffice: \cite{ben2015can} gives problems for which no such sequence exists. Successive limits are then necessary.

This motivates the notion of a tower of algorithms. The terminology was introduced by Doyle and McMullen \cite{Doyle_McMullen} and is now standard in the SCI literature.

\begin{definition}[Tower of algorithms]
\label{def:tower}
Let $(\Omega, \Lambda, \mathcal M, \Xi)$ be a computational problem. Let $\mathcal A$ be a subclass of general algorithms. For each $(n_k, \ldots, n_1) \in \N^k$ let $\Gamma_{n_k, \ldots, n_1} : \Omega \to \mathcal M$ be an algorithm in $\mathcal A$. Suppose that, for every $A\in\Omega$, the successive limits
$$
\Gamma_{n_k,\ldots,n_{r+1}}(A)
:=
\lim_{n_r\to\infty}\Gamma_{n_k,\ldots,n_r}(A),
\qquad r=1,\ldots,k-1,
$$
and the final limit
$$
\Xi(A)=\lim_{n_k\to\infty}\Gamma_{n_k}(A)
$$
all exist in $(\mathcal M,d)$. Then we say that $\sequence {\Gamma_{n_k, \ldots, n_1}}_{(n_k, \ldots, n_1) \in \N^k}$ is a tower of $\mathcal A$-algorithms of height $k$ solving $(\Omega, \Lambda, \mathcal M, \Xi)$. We refer to an $\mathcal A$-algorithm by itself as a tower of height $0$. Height $1$ means a single limit $\lim_{n_1\to\infty}\Gamma_{n_1}$; height $2$ means $\lim_{n_2\to\infty}\lim_{n_1\to\infty}\Gamma_{n_2,n_1}$.
\end{definition}

Convergence $\Gamma_n(A)\to\Xi(A)$ need not come with a computable rate: in general, one cannot determine how large $n$ must be to ensure $d(\Gamma_n(A),\Xi(A))<\epsilon$. For spectral-set problems we therefore distinguish two forms of one-sided control in the Attouch--Wets metric. Convergence from above requires quantified coverage of the spectrum by the approximating set, whereas convergence from below requires every approximating point to lie within quantified distance of the spectrum. The following definition makes these requirements precise.

\begin{definition}[SCI Hierarchy]\label{def:scihierarchy}
Let $(\Omega, \Lambda, \mathcal M, \Xi)$ be a computational problem.
\begin{itemize}
    \item We say that $(\Omega, \Lambda, \mathcal M, \Xi)$ is in $\Delta_0^{\mathcal A}$ if $\Xi$ is an $\mathcal A$-algorithm;
    \item We say that $(\Omega, \Lambda, \mathcal M, \Xi)$ is in $\Delta_1^{\mathcal A}$ if there exists an $\mathcal A$-tower of algorithms $\sequence {\Gamma_n}_{n \in \N}$ solving $(\Omega, \Lambda, \mathcal M, \Xi)$ such that $d(\Gamma_n(A), \Xi(A)) \le 2^{-n}$;
    \item We say that $(\Omega, \Lambda, \mathcal M, \Xi)$ is in $\Delta_{k + 1}^{\mathcal A}$ for $k \ge 1$ if $(\Omega, \Lambda, \mathcal M, \Xi)$ can be solved by an $\mathcal A$-tower of algorithms of height $k$;
    \item For spectral-set problems with $\mathcal M=\mathrm{CL}_\ast(\C)$, we say that $(\Omega, \Lambda, \mathcal M, \Xi)$ is in $\Sigma_k^{\mathcal A}$ for $k \ge 1$ if it can be solved by an $\mathcal A$-tower of algorithms of height $k$, $\sequence {\Gamma_{n_k, \ldots, n_1}}_{(n_k, \ldots, n_1) \in \N^k}$, such that for each $n_k \in \N$ and $A \in \Omega$ there exists $X_{n_k}(A) \supseteq \Gamma_{n_k}(A)$ with $d_{\mathrm{AW}}(X_{n_k}(A), \Xi(A)) \le 2^{-n_k}$;
    \item For spectral-set problems with $\mathcal M=\mathrm{CL}_\ast(\C)$, we say that $(\Omega, \Lambda, \mathcal M, \Xi)$ is in $\Pi_k^{\mathcal A}$ for $k \ge 1$ if it can be solved by an $\mathcal A$-tower of algorithms of height $k$, $\sequence {\Gamma_{n_k, \ldots, n_1}}_{(n_k, \ldots, n_1) \in \N^k}$, such that for each $n_k\in \N$ and $A \in \Omega$ there exists $X_{n_k}(A) \supseteq\Xi(A)$ with $d_{\mathrm{AW}}(X_{n_k}(A), \Gamma_{n_k}(A)) \le 2^{-n_k}$.
\end{itemize}
\end{definition}

The auxiliary sets $X_{n_k}(A)$ encode this one-sided control in the Attouch--Wets metric; they are not required to be algorithmic outputs. We will not need the $\Sigma_k^{\mathcal A}$ and $\Pi_k^{\mathcal A}$ classifications in this paper until \cref{sec:upper}.

\subsection{Markov Model}

In the Markov model, each operator is presented by a finite source code. We formalize such codes using Turing machines and adopt the standard notation of computability theory \cite[Chapters 1 and 4]{soare2016}. By the Church--Turing thesis, this captures finite pen-and-paper arithmetic procedures.

A Turing machine need not halt on every input, so it naturally computes a \emph{partial} function $\N \rightharpoonup \N$: if the machine halts on input $n$, its output is the value of the function at $n$, and if it does not halt, the function is undefined at $n$. A partial function arising in this way is called \emph{partial computable}. Fix an effective enumeration $\mathcal T_1, \mathcal T_2, \ldots$ of the partial computable functions $\N\rightharpoonup\N$. There exists a \emph{universal Turing machine} which can accept an index $e$ and an input $n$, and simulate the action of $\mathcal T_e(n)$. We consider $e$ to be the \emph{source code}, also called a G\"odel number, of $\mathcal T_e$ (since $\mathcal T_e$ can be reconstructed from $e$). We use $\mathcal T_e(n) \downarrow$ to mean that the program with index $e$ halts on input $n$, and $\mathcal T_e(n) \uparrow$ otherwise; similarly, $\mathcal T_{e, s}(n) \downarrow$ means that this computation halts within $s$ steps. Set
$
W_e = \set {n \in \N : \mathcal T_e(n) \downarrow},
$
the set of inputs for which $\mathcal T_e$ halts. We approximate this set by
$
W_{e, s} = \set {n \le s : \mathcal T_{e, s}(n) \downarrow},
$
where $W_{e, 0} = \emptyset$. Then $W_{e, s} \subseteq W_{e, s + 1}$ and $W_e = \bigcup_{s \in \N} W_{e, s}$. For $s \ge 1$, we say that a new element enters $W_e$ at stage $s$ if $W_{e, s} \setminus W_{e, s - 1} \ne \emptyset$. A set $S \subseteq \N$ is called \emph{computably enumerable} if there exists $e \in \N$ with $S = W_e$. Equivalently, a pen-and-paper procedure lists exactly the elements of $S$, with every element appearing after finitely many stages. In general, no uniform procedure decides $n\notin S$ from an index for a computably enumerable set $S$.

Computational problems in the Markov model are formulated as follows. In this setting, the input is a finite source code describing the operator. For example, one may choose to represent the Schr\"odinger operator $-\Delta + V$ by a program that computes rational approximations to $V(x)$ for $x \in \Q$, or, when $V$ is analytic, by a program that computes its Taylor coefficients.

Fix a particular coding convention for the class of operators under consideration. Let $\Lambda\subseteq\mathbb N$ be the set of valid source codes in the chosen convention, and for each $e\in\Lambda$, let $A_e$ be the corresponding operator described by the program $\mathcal T_{e}$. These operators need not be distinct, and two different programs may describe the same operator. The algorithm is required to work for any valid source code.

\begin{definition}
Let $\mathcal H$ be a separable Hilbert space, let $\Lambda = \set {e_n}_{n \in \N}$ be an indexed family of valid source codes, and let $\Omega = \set {A_n}_{n \in \N}$ be the corresponding indexed family of closed operators on $\mathcal H$. Let $\Xi : \Omega \to \mathrm{CL}_\ast(\C)$ be a problem function. We call $(\Omega, \Lambda, \Xi)$ a Markov spectral problem.
\end{definition}

For a countable set $C$, a \emph{numbering} is a way of assigning natural numbers to the elements of $C$. Formally, it is a partial map $\nu:\mathbb N \rightharpoonup C$ whose range is all of $C$. Thus every element $x\in C$ has at least one natural number $n$ with $\nu(n)=x$; such an $n$ is called a $\nu$-index for $x$. Some natural numbers may not name any element of $C$. We call the numbering effective if the partial map $n\mapsto\nu(n)$ is computable on its domain. We use the standard computable encodings of finite tuples of natural numbers by single natural numbers. In particular, we fix effective numberings of $\mathbb Q$ and of $\mathbb Q+i\mathbb Q$. These induce an effective numbering of $\mathcal F(\mathbb Q+i\mathbb Q)$, the set of finite, possibly empty, subsets of $\mathbb Q+i\mathbb Q$, by coding finite lists of rational complex numbers. We denote this fixed numbering by $\pi_{\mathcal F(\mathbb Q+i\mathbb Q)}$.

At level $k$, the algorithm receives a tuple $(n_k,\ldots,n_1,e)$, where $(n_k,\ldots,n_1)$ are the tower indices and $e$ is a source code for the operator. It returns a $\pi_{\mathcal F(\Q+i\Q)}$-index for a finite set $\Gamma_{n_k,\ldots,n_1}(e)\in\mathcal F(\Q+i\Q)$. These finite sets must satisfy the successive convergence conditions in \cref{def:tower}.

\begin{definition}[Markov tower of algorithms]
Let $(\Omega, \Lambda, \Xi)$ be a Markov spectral problem. A Markov algorithm at level $k$ is a partial computable map $\Gamma:\N^{k+1}\rightharpoonup\N$ with the following property: for every valid source code $e=e_n\in\Lambda$ and every $(n_k,\ldots,n_1)\in\N^k$, the computation $\Gamma(n_k,\ldots,n_1,e)$ halts and returns a $\pi_{\mathcal F(\Q + i \Q)}$-index for a finite set in $\mathcal F(\Q + i \Q)$. We call the indexed family $\sequence {\Gamma_{n_k, \ldots, n_1}}$ a \emph{Markov tower of height $k$ solving $(\Omega, \Lambda, \Xi)$} if its successive limits exist in the sense of \cref{def:tower}, and the final limit is $\Xi(A_n)$ for every valid source code $e_n$.
\end{definition}

Equivalently, the algorithm accepts the tower indices together with the source code of the operator, and returns a finite subset of $\Q + i \Q$. With this identification, we define membership in $\Delta_k^M$, $\Sigma_k^M$ and $\Pi_k^M$ as in \cref{def:scihierarchy}, with Markov towers replacing arbitrary towers.

\subsection{The Arithmetical Hierarchy}

The arithmetical hierarchy classifies subsets of $\N$ according to the number and order of unbounded quantifiers ($\exists,\forall$) needed to define them by arithmetical formulas. A $\Delta_0$ formula is a first-order formula in the language of arithmetic, with symbols $0, 1, +, \times, =$ and the successor function, in which every quantifier is bounded. For $B \subseteq \N$ and $n \ge 1$, the classes $\Sigma_n^0$, $\Pi_n^0$ and $\Delta_n^0$ are defined as follows \cite[Definition 4.1.2]{soare2016}:
\begin{itemize}
	\item $B \in \Sigma_n^0$ if there is a $\Delta_0$ formula $\theta$ such that, for all $x \in \N$,
	$$
	x\in B \Longleftrightarrow
	(Q_1y_1)\cdots(Q_ny_n)\,\theta(x,y_1,\ldots,y_n),
	$$
	where $Q_i=\exists$ for odd $i$ and $Q_i=\forall$ for even $i$;
	\item $B\in\Pi_n^0$ if there is a $\Delta_0$ formula $\theta$ such that, for all $x\in\N$,
	$$
	x\in B \Longleftrightarrow
	(Q_1y_1)\cdots(Q_ny_n)\,\theta(x,y_1,\ldots,y_n),
	$$
	where $Q_i=\forall$ for odd $i$ and $Q_i=\exists$ for even $i$;
	\item $B\in\Delta_n^0$ if $B\in\Sigma_n^0\cap\Pi_n^0$.
\end{itemize}
We call such a formula a $\Sigma_n^0$ or $\Pi_n^0$ formula, respectively. A set is \emph{arithmetical} if it belongs to $\Sigma_n^0$ or $\Pi_n^0$ for some $n$.

A well-known example is the halting problem
$$
\mathbf K = \set {e \in \N : \mathcal T_e(e) \downarrow}.
$$
This set is $\Sigma_1^0$ since
$$
e \in \mathbf K \Longleftrightarrow (\exists s) \mathcal T_{e, s}(e) \downarrow,
$$
and the relation $\mathcal T_{e, s}(e)\downarrow$ is computable for fixed $e$ and $s$. A classic result of Turing says that $\mathbf K$ is not computable. If $e \in \mathbf K$, we can verify this inclusion in finite time by running the Turing machine $\mathcal T_e(e)$ for sufficiently many steps. On the other hand, if $e \notin \mathbf K$, since $\mathbf K$ is not computable, there is no uniform procedure to verify that $\mathcal T_e(e)$ does not halt. Dually, if a $\Pi_1^0$ statement $(\forall x) \theta(x)$ is false, then we can verify it in finite time by finding $x_0$ such that $\neg \theta(x_0)$, whereas if it is true, we cannot generally verify it in finite time.

For $X\in\{\Delta,\Sigma,\Pi\}$, a set $A$ is \emph{$X_n^0$-complete} if $A\in X_n^0$ and every set in $X_n^0$ admits a computable embedding into $A$. The arithmetical hierarchy is strict. Hence an $X_n^0$-complete set does not belong to $\Delta_n^0$ when $X\in\{\Sigma,\Pi\}$. In our lower bounds, we will use the following standard sets
$$
\Tot=\set{e\in\N:W_e=\N},\qquad
\Cof=\set{e\in\N:\N\setminus W_e\text{ is finite}}.
$$
$\Tot$ is the set of indices of programs that halt on every input, and $\Cof$ is the set of indices of programs that halt on all but finitely many inputs.

\begin{proposition}[e.g. {\cite[Theorems 4.3.2 and 4.3.3]{soare2016}}]
\label{prop:complete_class}
$\Tot$ is $\Pi_2^0$-complete, and hence is not in $\Sigma_2^0$. The set $\Cof$ is $\Sigma_3^0$-complete, and hence is not in $\Pi_3^0$.
\end{proposition}

A computable embedding of $\Tot$ therefore gives a $\notin\Sigma_2^0$ result, while a computable embedding of $\Cof$ gives a $\notin\Pi_3^0$ result.

\subsection{The Mechanism for Lower Bounds}

Our spectral lower bounds are reductions from complete arithmetical sets. If a spectral problem admitted a Markov tower below the claimed level, the reduction would induce a tower for the embedded set at the same lower level, contradicting its completeness. We begin with the reduction used for the $\notin\Delta_2^M$ conclusions of \cref{thm:intro_smooth_lower_bounds}.

\begin{proposition}
\label{prop:dichot}
Let $(\Omega, \Lambda, \Xi)$ be a Markov spectral problem. Let $S \subseteq \N$ and suppose that there is a computable map $e \mapsto s(e) \in \Lambda$. Denote by $A_e$ the operator described by the source code $s(e)$. Suppose that for some open ball $B = D_r(x)$, with $x \in \Q + i \Q$ and $r \in\mathbb{Q}_{>0}$, we have
$$
e\in S \Longrightarrow \Xi(A_e)\cap B\neq\emptyset,\qquad
e\notin S \Longrightarrow \Xi(A_e)\cap \overline B=\emptyset.
$$
We call this a spectral dichotomy. If $(\Omega,\Lambda,\Xi)\in\Delta_2^M$, then $S\in\Delta_2^0$.
\end{proposition}

In \cref{prop:dichot}, only a uniform procedure $e\mapsto s(e)$ producing valid source codes is required. The gap between inclusion in $B$ and exclusion from $\overline B$ removes the boundary case.

\begin{proof}[Proof of \cref{prop:dichot}]
Suppose that $(\Omega,\Lambda,\Xi)\in\Delta_2^M$. Then there exists a height-one Markov tower $\sequence{\Gamma_n}_{n\in\N}$ solving $(\Omega,\Lambda,\Xi)$. Thus, for every valid source code $s(e)$, $\Gamma_n(A_e)\rightarrow \Xi(A_e)$ in $(\mathrm{CL}_\ast(\C),d_{\mathrm{AW}})$.

For each $e\in\N$, the map $e\mapsto s(e)$ is computable and $s(e)$ is a valid source code. Hence the Markov algorithm computing $\Gamma_n$ halts on input $(n,s(e))$ and returns an index for a finite set
$$
F_n(e):=\Gamma_n(A_e)\in \mathcal F(\Q+i\Q).
$$
With the convention $\dist(x,\emptyset)=\infty$, define
$$
\beta_n(e):=\min\{r+1,\dist(x,F_n(e))\}.
$$
If $F_n(e)=\emptyset$, then $\beta_n(e)=r+1$. Otherwise, $\dist(x,F_n(e))$ is the minimum of finitely many computable real numbers of the form $|x-y|$, with $y\in\Q+i\Q$. Since $x\in\Q+i\Q$ and $r\in\mathbb{Q}$, we can compute $\beta_n(e)$ to arbitrary precision, uniformly in $n$ and $e$. In particular, setting $\varepsilon_n=\frac1{n+1}$, we can compute a rational number $\alpha_n(e)$ such that $|\alpha_n(e)-\beta_n(e)|<\varepsilon_n$. Define
$$
\Upsilon_n(e)=
\begin{cases}
1, & \alpha_n(e)<r,\\
0, & \alpha_n(e)\ge r.
\end{cases}
$$
Then the relation $\Upsilon_n(e)=1$ is decidable uniformly in $(n,e)$.

We claim that $\Upsilon_n(e)\rightarrow \1_S(e)$ as $n\to\infty$. First observe that, since $F_n(e)\to \Xi(A_e)$ in $d_{\mathrm{AW}}$, we have
$$
\theta(\dist(x,F_n(e)))
\longrightarrow
\theta(\dist(x,\Xi(A_e))),
$$
where $\theta(t)=(1+t)^{-1}$ on $[0,\infty]$. Choose $m\in\N$ with $|x|\le m$; the $m$th term in the definition of $d_{\mathrm{AW}}$ gives this convergence. Since $\theta$ is a homeomorphism $[0,\infty]\to[0,1]$, it follows that $\dist(x,F_n(e))\rightarrow\dist(x,\Xi(A_e))$ in $[0,\infty]$. Therefore $\beta_n(e)=\min\{r+1,\dist(x,F_n(e))\}\rightarrow\beta(e):=\min\{r+1,\dist(x,\Xi(A_e))\}$.

Suppose first that $e\in S$. Then $\Xi(A_e)\cap B\ne\emptyset$, so $\dist(x,\Xi(A_e))<r$. Thus $\beta(e)<r$. Choose $\delta>0$ such that $\beta(e)+2\delta<r$. For all sufficiently large $n$, we have $\beta_n(e)<\beta(e)+\delta$ and $\varepsilon_n<\delta$. Hence
$$
\alpha_n(e)
<
\beta_n(e)+\varepsilon_n
<
\beta(e)+2\delta
<
r.
$$
Thus $\Upsilon_n(e)=1$ for all sufficiently large $n$.

Now suppose that $e\notin S$. Then $\Xi(A_e)\cap \overline B=\emptyset$. If $\Xi(A_e)=\emptyset$, then $\dist(x,\Xi(A_e))=\infty$ and hence $\beta(e)=r+1>r$. If $\Xi(A_e)\ne\emptyset$, then $\Xi(A_e)$ is closed and disjoint from the closed ball $\overline B=\overline{D_r(x)}$, so $\dist(x,\Xi(A_e))>r$. Again $\beta(e)>r$. Choose $\delta>0$ such that $\beta(e)-2\delta>r$. For all sufficiently large $n$, we have $\beta_n(e)>\beta(e)-\delta$ and $\varepsilon_n<\delta$. Therefore
$$
\alpha_n(e)
>
\beta_n(e)-\varepsilon_n
>
\beta(e)-2\delta
>
r.
$$
Thus $\Upsilon_n(e)=0$ for all sufficiently large $n$.

Thus $\Upsilon_n(e)$ converges pointwise to $\1_S(e)$. Hence
$$
e\in S
\Longleftrightarrow
(\exists N)(\forall n\ge N)\,[\Upsilon_n(e)=1].
$$
Since $\Upsilon_n(e)=1$ is a computable relation, this is a $\Sigma_2^0$ definition of $S$. Moreover, because $\Upsilon_n(e)$ is eventually equal to $\1_S(e)$, we also have
$$
e\in S
\Longleftrightarrow
(\forall N)(\exists n\ge N)\,[\Upsilon_n(e)=1].
$$
This is a $\Pi_2^0$ definition of $S$. Hence $S\in \Sigma_2^0\cap \Pi_2^0=\Delta_2^0$.
\end{proof}

For height two, the inner limit need not be computable. The next reduction supplements \cref{prop:dichot} with a stabilization step before the outer limit.

\begin{proposition}
\label{prop:delta3bound}
Let $(\Omega, \Lambda, \Xi)$ be a Markov spectral problem and suppose that there is a computable map $e \mapsto s(e) \in \Lambda$. Denote by $A_e$ the operator described by the source code $s(e)$. Suppose that for some open ball $B = D_r(x)$, with $x \in \Q + i \Q$ and $r \in\mathbb{Q}_{>0}$, we have
$$
e\in \Cof \Longrightarrow \Xi(A_e)\cap B\neq\emptyset,\qquad
e\notin \Cof \Longrightarrow \Xi(A_e)\cap \overline B=\emptyset.
$$
Then $(\Omega, \Lambda, \Xi) \notin \Delta_3^M$.
\end{proposition}

\begin{proof}
Suppose, towards a contradiction, that the hypotheses hold and that $(\Omega,\Lambda,\Xi)\in\Delta_3^M$. Let $\sequence{\Gamma_{m,\ell}}_{m,\ell\in\mathbb N}$ be a height-two Markov tower for the problem. For $m\in\mathbb N$, set $\varepsilon_m=(m+1)^{-1}$, $ R_m=r+5\varepsilon_m$, and define the two closed threshold sets
$$
    J_{1,m}=(-\infty,r+2\varepsilon_m],
    \qquad
    J_{2,m}=[r+4\varepsilon_m,\infty).
$$
Their distance is $2\varepsilon_m>0$. Define
$$
    d_{m,\ell}(e)
    :=
    \min\{R_m,\dist(x,\Gamma_{m,\ell}(A_e))\},
    \qquad
    \delta_{m,\ell}:=2^{-\ell-2}\varepsilon_m.
$$
As in \cref{prop:dichot}, from $m,\ell,e$ we can compute a rational number $\alpha_{m,\ell}(e)$ such that
$$
    |\alpha_{m,\ell}(e)-d_{m,\ell}(e)|<\delta_{m,\ell}.
$$
The approximation tolerance tends to zero with $\ell$ for every fixed $m$, independently of the outer tolerance $\varepsilon_m$.

Define the computable value $\Upsilon_{m,\ell}(e)$ as follows. If no $k\le\ell$ satisfies $\alpha_{m,k}(e)\in J_{1,m}\cup J_{2,m}$, set $\Upsilon_{m,\ell}(e)=1$. Otherwise, let $k_*$ be the largest such $k$ and set
$$
    \Upsilon_{m,\ell}(e)
    =
    \begin{cases}
        1, & \alpha_{m,k_*}(e)\in J_{1,m},\\
        0, & \alpha_{m,k_*}(e)\in J_{2,m}.
    \end{cases}
$$

Fix $m$ and $e$. Since $\Gamma_{m,\ell}(A_e)\to\Gamma_m(A_e)$ in $\mathrm{CL}_*(\mathbb C)$, continuity of truncated distance gives
$$
    d_{m,\ell}(e)
    \longrightarrow
    d_m(e):=\min\{R_m,\dist(x,\Gamma_m(A_e))\}.
$$
Because $\delta_{m,\ell}\to0$, it follows that $\alpha_{m,\ell}(e)\to d_m(e)$. If $J_{1,m}\cup J_{2,m}$ is visited only finitely often, then $\Upsilon_{m,\ell}(e)$ is constant after the last visit. If the union is visited infinitely often, convergence of $\alpha_{m,\ell}(e)$ and the positive gap between $J_{1,m}$ and $J_{2,m}$ imply that the two sets cannot both be visited infinitely often. The other set is therefore visited only finitely often; after its last visit, every subsequent visit has the same label, and that label is retained between visits. Thus $\Upsilon_{m,\ell}(e)$ again stabilizes. Hence $ \Upsilon_m(e):=\lim_{\ell\to\infty}\Upsilon_{m,\ell}(e) $ exists for every $m,e$.

We next determine the outer limit. The convergence $\Gamma_m(A_e)\to\Xi(A_e)$ implies
$$
    \min\{r+1,\dist(x,\Gamma_m(A_e))\}
    \longrightarrow
    \min\{r+1,\dist(x,\Xi(A_e))\}.
$$
If $e\in\Cof$, then $d_\infty(e):=\dist(x,\Xi(A_e))<r$. Choose $\eta>0$ such that $d_\infty(e)<r-3\eta$. For all sufficiently large $m$,
$$
    \dist(x,\Gamma_m(A_e))<r-2\eta
    \quad\text{and}\quad
    \varepsilon_m<\eta.
$$
For each such $m$, we have $d_m(e)<r-2\eta$, and hence, for all sufficiently large $\ell$, $\alpha_{m,\ell}(e)<r-\eta<r+2\varepsilon_m$. Thus the sequence is eventually in $J_{1,m}$, so $\Upsilon_m(e)=1$.

If $e\notin\Cof$, then the closed set $\Xi(A_e)$ is disjoint from $\overline{D_r(x)}$. With the convention $\dist(x,\emptyset)=\infty$, this gives
$$
    \overline d(e)
    :=
    \min\{r+1,\dist(x,\Xi(A_e))\}>r.
$$
Choose $\eta>0$ so that $\overline d(e)>r+3\eta$. For all sufficiently large $m$,
$$
    \dist(x,\Gamma_m(A_e))>r+2\eta
    \quad\text{and}\quad
    5\varepsilon_m<2\eta.
$$
Thus,
$
    R_m=r+5\varepsilon_m<r+2\eta
    <\dist(x,\Gamma_m(A_e)),
$
so $d_m(e)=R_m$. Since $\alpha_{m,\ell}(e)\to R_m$, for all sufficiently large $\ell$,
$$
    \alpha_{m,\ell}(e)
    >R_m-\frac{\varepsilon_m}{2}
    >r+4\varepsilon_m.
$$
Thus the sequence is eventually in $J_{2,m}$, so $\Upsilon_m(e)=0$. We have proved $ \lim_{m\to\infty}\Upsilon_m(e)=\1_{\Cof}(e). $

It follows that
$$
\begin{aligned}
    e\in\Cof
    &\Longleftrightarrow
    (\forall M)(\exists m\ge M)\,[\Upsilon_m(e)=1]\\
    &\Longleftrightarrow
    (\forall M)(\exists m\ge M)(\exists L)(\forall\ell\ge L)\,
    [\Upsilon_{m,\ell}(e)=1].
\end{aligned}
$$
The two existential quantifiers can be merged by computable pairing, and $\Upsilon_{m,\ell}$ is computable. This is a $\Pi_3^0$ definition of $\Cof$, contradicting \cref{prop:complete_class}.
\end{proof}

\section{Lower Bound for Pure Point Spectrum}
\label{sec:pp}

We prove the pure point lower bound by a reduction from $\Tot$. For each program index $e$, we construct a smooth, bounded, real-valued potential $V_e$. If $e\in\Tot$, this is a generalized Gordon potential and the corresponding Schr\"odinger operator has no eigenvalues. If $e\notin\Tot$, the potential is compactly supported and contains a fixed well that forces an eigenvalue in $(-7/8,-5/8)$.

Let $L^1_{\mathrm{loc},\mathrm{unif}}(\mathbb R)$ denote the space of real-valued locally integrable functions $V$ for which
$$
    \|V\|_{1,\mathrm{unif}}
    =
    \sup_{x\in\mathbb R}
    \int_x^{x+1}|V(t)|\,\mathrm dt
    <\infty.
$$
Equivalently, the $L^1$-norm of $V$ is uniformly bounded across all unit intervals. Following \cite{damanik2000generalization}, we say that $V$ is a \emph{generalized Gordon potential} if $V\in L^1_{\mathrm{loc},\mathrm{unif}}(\mathbb R)$ and there exist locally integrable $T_m$-periodic potentials $V_m$, with $T_m\to\infty$, such that
$$
    \lim_{m\to\infty}
    \mathrm e^{C T_m}
    \int_{-T_m}^{2T_m}|V(x)-V_m(x)|\,\mathrm dx
    =
    0
    \qquad\text{for every } C>0.
$$
We use the following theorem of Damanik and Stolz. In the statement below, $-\Delta+V$ is understood in the standard self-adjoint realization for the one-dimensional Schr\"odinger expression with real $L^1_{\mathrm{loc},\mathrm{unif}}$ potential; in our application, the potentials are smooth and bounded.

\begin{proposition}[{\cite[Theorem 1.1]{damanik2000generalization}}]
\label{prop:gordon_pp}
If $V$ is a generalized Gordon potential, then $\specpp(-\Delta+V)=\emptyset$.
\end{proposition}

The other ingredient is a compactly supported wave packet which will force a negative spectral point. In the compactly supported case considered below, this negative spectral point is necessarily an eigenvalue.

\begin{lemma}
\label{lem:psi_eig}
There are explicit choices of $L \in \N$ and a computable function $\psi \in C_c^\infty((L/3, 2L/3); \C)$ with $\norm \psi_{L^2} = 1$ such that $\|(-\Delta - 1 /4) \psi\|_{L^2} < 1/8$.
\end{lemma}

\begin{proof}
Fix a non-zero computable function $\chi\in C_c^\infty((0,1))$ for which one can compute rational constants $A_0,A_1,A_2>0$ satisfying
$$
    0<A_0\le \|\chi\|_{L^2},
    \qquad
    \|\chi'\|_{L^2}\le A_1,
    \qquad
    \|\chi''\|_{L^2}\le A_2.
$$
For example, one may take a translated and dilated standard bump, such as
$$
    \chi(x)=
    \begin{cases}
        \exp\!\left(-\dfrac{1}{(x-1/4)(3/4-x)}\right),
        & 1/4<x<3/4,\\[2mm]
        0, & \text{otherwise}.
    \end{cases}
$$
The required $L^2$-bounds are computable. For $R>0$, define
$$
    \psi_R(x)
    =
    c_\chi R^{-1/2}\chi(x/R)\mathrm e^{ix/2},
    \qquad
    c_\chi=\|\chi\|_{L^2}^{-1}.
$$
The positive number $\|\chi\|_{L^2}$ is computable, so $c_\chi$ is computable. A change of variables gives $\|\psi_R\|_{L^2}=1$, and $\psi_R$ is supported in $(0,R)$. Differentiating gives
$$
    \left(-\Delta-\frac14\right)\psi_R
    =
    c_\chi R^{-1/2}\mathrm e^{ix/2}
    \left(
        -R^{-2}\chi''(x/R)-iR^{-1}\chi'(x/R)
    \right).
$$
Hence, again by the substitution $u=x/R$,
$$
    \left\|
        \left(-\Delta-\frac14\right)\psi_R
    \right\|_{L^2}
    \le
    \|\chi\|_{L^2}^{-1}
    \left(
        R^{-2}\|\chi''\|_{L^2}
        +
        R^{-1}\|\chi'\|_{L^2}
    \right)  \le
    A_0^{-1}
    \left(
        R^{-2}A_2+R^{-1}A_1
    \right).
$$
The final expression is a computable upper bound and tends to $0$ as $R\to\infty$. Choose $L\in\mathbb N$ large so that, with $R=L/3$,
$$
    A_0^{-1}
    \left(
        R^{-2}A_2+R^{-1}A_1
    \right)
    <
    \frac18.
$$
Finally, set $\psi(x)=\psi_{L/3}(x-L/3)$. Translation preserves the $L^2$-norm and the quantity $\|(-\Delta-\tfrac14)\psi\|_{L^2}$. Since $\psi_{L/3}$ is supported in $(0,L/3)$, the translated function $\psi$ is supported in $(L/3,2L/3)$.
\end{proof}

Using the bump construction from \cref{lem:psi_eig}, choose a computable function $\eta\in C_c^\infty((0,1))$ such that $0\le\eta\le1$ and $\eta(t)=1$ for $t\in[1/3,2/3]$. For every $\ell\in\mathbb N\cup\{0\}$, choose, computably in $\ell$, a rational number $c_\ell$ satisfying
$
    \sup_{x\in\mathbb R}|\eta^{(\ell)}(x)|\le c_\ell.
$
Define $\eta_L(x)=\eta(x/L)$. Then
$$
    \eta_L\in C_c^\infty((0,L)),\qquad
    0\le \eta_L\le 1,\qquad
    \eta_L=1\text{ on }[L/3,2L/3],\qquad
    \eta_L^{(\ell)}(x)=L^{-\ell}\eta^{(\ell)}(x/L).
$$
The potential is formed from disjoint translates of $\eta_L$.

For a program index $e$, set $M_0(e) = 0$ and, for $s \ge 1$, define
$$
M_s(e) = \max \left( \{0\}\cup\set {m \le s : \mathcal T_e(1), \ldots, \mathcal T_e(m) \text { all halt within } s \text { steps}}\right).
$$
We call a stage $s$ active if $M_s(e) > M_{s - 1}(e)$. We can decide whether $s$ is active in finite time by running $\mathcal T_e(1), \ldots, \mathcal T_e(s)$ for $s$ steps. Note that $e \in \Tot$ if and only if infinitely many stages are active. If $e \notin \Tot$ and $n_0$ is the least input on which $\mathcal T_e$ does not halt, then $M_s(e) \le n_0-1$ for all $s$, and eventually $M_s(e) = n_0-1$.

We construct $b_j(e)\in\{-1,0\}$, $j\in\mathbb Z$, from nested finite blocks. Define
$$
    N_m=3^m,
    \qquad
    I_m=\{j\in\mathbb Z:-N_m\le j<2N_m\}.
$$
Thus $I_m$ has length $3N_m=N_{m+1}$, and the intervals $I_m$ exhaust $\mathbb Z$. At level $m=0$, define
$$
    b_{-1}^{(0)}(e)=b_0^{(0)}(e)=b_1^{(0)}(e)=-1.
$$
Suppose that $b_j^{(m)}(e)$ has been defined for all $j\in I_m$. We define $b_j^{(m+1)}(e)$ for $j\in I_{m+1}$ as follows. Let $\pi_m(j)\in I_m$ be the unique integer satisfying
$
    \pi_m(j)\equiv j \pmod{N_{m+1}}.
$
Such a representative exists and is unique because $I_m$ has length $N_{m+1}$. If stage $m+1$ is active, set
$$
    b_j^{(m+1)}(e)
    =
    b_{\pi_m(j)}^{(m)}(e),
    \qquad j\in I_{m+1}.
$$
Thus the old block on $I_m$ is repeated periodically, with period $N_{m+1}$, across the larger interval $I_{m+1}$. If stage $m+1$ is not active, set
$$
    b_j^{(m+1)}(e)
    =
    \begin{cases}
        b_j^{(m)}(e), & j\in I_m,\\
        0, & j\in I_{m+1}\setminus I_m.
    \end{cases}
$$
In either case, the definitions are compatible on $I_m$. Hence, for each $j\in\mathbb Z$, once $m$ is large enough that $j\in I_m$, the value $b_j^{(m)}(e)$ is independent of all later levels. We therefore define
$$
    b_j(e)=b_j^{(m)}(e),
    \qquad j\in I_m.
$$
The value $b_j(e)$ is computable uniformly in $e$ and $j$. Given $j$, choose $m$ with $j\in I_m$. Determining $b_j^{(m)}(e)$ requires only the finite simulations that decide whether the stages $1,\ldots,m$ are active. Two consequences will be used below. If stage $m+1$ is active, then $b^{(m+1)}$ is $N_{m+1}$-periodic on $I_{m+1}=[-N_{m+1},2N_{m+1})$. If only finitely many stages are active, then every sufficiently distant newly created site receives the value zero, and hence $b_j(e)=0$ for all sufficiently large $|j|$.

Define the potential
$$
    V_e(x)
    =
    \sum_{j\in\mathbb Z} b_j(e)\eta_L(x-Lj).
$$
The supports of the translates $\eta_L(\cdot-Lj)$ are contained in the pairwise disjoint intervals
$$
    (Lj,L(j+1)),\qquad j\in\mathbb Z.
$$
Hence, the sum is locally finite; in fact, for each $x$, at most one summand is non-zero. More precisely,
$$
    V_e(x)
    =
    b_j(e)\eta_L(x-Lj),
    \qquad x\in(Lj,L(j+1)).
$$
At the lattice points $x=Lj$, all summands and all their derivatives vanish. Thus $V_e\in C^\infty(\mathbb R)$. Since $b_j(e)\in\{-1,0\}$ and $0\le\eta_L\le1$, we have $ -1\le V_e(x)\le0. $ Define
$$
    H_{e,0}^G f=-f''+V_e f,
    \qquad
    \operatorname{Dom}(H_{e,0}^G)=C_c^\infty(\mathbb R).
$$
Because $V_e$ is a bounded real-valued multiplication operator, the bounded perturbation theorem implies that $H_{e,0}^G$ is essentially self-adjoint. We denote its self-adjoint closure by
$$
    H_e^G:=\overline{H_{e,0}^G},
    \qquad
    \operatorname{Dom}(H_e^G)=H^2(\mathbb R).
$$
Thus $H_e^G\in\mathcal O_M^\infty$. Also, $\|V_e\|_{1,\mathrm{unif}}\le 1$, so $V_e\in L^1_{\mathrm{loc},\mathrm{unif}}(\mathbb R)$.

\begin{lemma}\label{source_comp_pp}
There exists a computable function $s_G : \N \to \mathsf{Code}_M^\infty$ such that $s_G(e)$ is a valid smooth source code for $H_e^G$, equivalently for its potential $V_e$.
\end{lemma}

\begin{proof}
We describe the source code uniformly in $e$. The code contains the fixed computable function $\eta$, the integer $L$, and the program index $e$. Given $j\in\mathbb Z$, the coefficient $b_j(e)$ is computable by the finite bounded simulation described above.

Let $\ell\in\mathbb N\cup\{0\}$. For $x\in\mathbb R$, we have
$$
    V_e^{(\ell)}(x)
    =
    \sum_{j\in\mathbb Z}
    b_j(e)L^{-\ell}\eta^{(\ell)}(x/L-j).
$$
Fix a computable real $x$, and obtain from its certified approximation a rational number $R_0$ such that $|x|\le R_0$. If $\eta^{(\ell)}(x/L-j)\ne0$, then $|j|\le R_0/L+1$. Hence the displayed sum contains only finitely many terms. The corresponding coefficients $b_j(e)$ and the values $\eta^{(\ell)}(x/L-j)$ are computable, uniformly in $e$, $x$, and $\ell$, so $V_e^{(\ell)}(x)$ can be approximated to any prescribed accuracy. The formula remains valid at the lattice points $x\in L\mathbb Z$, since the zero extension of $\eta$ and all its derivatives vanish there.

The certified local derivative bounds follow directly. Since at every point at most one translate contributes and $|b_j(e)|\le1$, we have, for every $R\in\mathbb N$,
$$
    \sup_{|x|\le R}|V_e^{(\ell)}(x)|
    \le
    L^{-\ell}\sup_{u\in\mathbb R}|\eta^{(\ell)}(u)|.
$$
The support of $\eta$ lies in $[0,1]$, so the right-hand side is bounded by $L^{-\ell}c_{\ell}$, where $c_{\ell}$ is one of the computable bounds fixed above. Therefore the source code can output, uniformly in $R$ and $\ell$, the certified bound $C_{R,\ell}:=L^{-\ell}c_{\ell}$. For $\ell=0$ one may also use the sharper bound $C_{R,0}=1$. This provides the required smooth source code for $V_e$, uniformly in $e$.
\end{proof}

\begin{lemma}
\label{lem:gordon_compact_supp}
If $e\in\Tot$, then $V_e$ is a generalized Gordon potential. If $e\notin\Tot$, then $V_e$ is compactly supported.
\end{lemma}

\begin{proof}
If stage $m$ is active, then the block $\{b_j(e):j\in I_m\}$ is $N_m$-periodic on $I_m$ with $b_{j+N_m}(e)=b_j(e)$ whenever both $j$ and $j+N_m$ belong to $I_m$. Equivalently, the three blocks
$$
    \{b_j(e):-N_m\le j<0\},\qquad
    \{b_j(e):0\le j<N_m\},\qquad
    \{b_j(e):N_m\le j<2N_m\}
$$
are identical.

Suppose first that $e\in\Tot$. Then there are infinitely many active stages $p_1<p_2<\cdots$. For $r\ge1$, set
$$
    P_r=N_{p_r}=3^{p_r},
    \qquad
    T_r=P_rL.
$$
Then $T_r\to\infty$. We define a $T_r$-periodic potential $V_{e,r}$ as follows. Let $\widetilde b_j^{(r)}$ be the $P_r$-periodic extension of the finite block $b_0(e),b_1(e),\ldots,b_{P_r-1}(e)$. Set
$$
    V_{e,r}(x)
    =
    \sum_{j\in\mathbb Z}
    \widetilde b_j^{(r)}\eta_L(x-Lj).
$$
Since the coefficient sequence $\widetilde b^{(r)}$ is $P_r$-periodic and the bumps are translates by multiples of $L$, the potential $V_{e,r}$ is $T_r=P_rL$-periodic. Because $p_r$ is active, the original coefficients $b_j(e)$ are $P_r$-periodic throughout
$$
    I_{p_r}=\{-P_r,-P_r+1,\ldots,2P_r-1\}.
$$
Therefore
$$
    \widetilde b_j^{(r)}=b_j(e),
    \qquad
    -P_r\le j<2P_r.
$$
It follows that
$$
    V_{e,r}(x)=V_e(x),
    \qquad
    -T_r\le x\le 2T_r.
$$
On each open cell $(Lj,L(j+1))$ with $-P_r\le j<2P_r$, both potentials are the same multiple of the same bump $\eta_L(x-Lj)$; at the endpoints $x\in L\mathbb Z$, all translated bumps and all their derivatives vanish. Thus, for every $C>0$ and every $r$,
$$
    \mathrm e^{C T_r}
    \int_{-T_r}^{2T_r}
    |V_e(x)-V_{e,r}(x)|\,\mathrm dx
    =
    0.
$$
Since $V_e$ is bounded, it belongs to $L^1_{\mathrm{loc},\mathrm{unif}}(\mathbb R)$. Together with $T_r\to\infty$, this shows that $V_e$ is a generalized Gordon potential.

Now suppose that $e\notin\Tot$. Then only finitely many stages are active. Let $m_\ast$ be larger than every active stage. After level $m_\ast$, every newly created site is assigned the value $0$, while the previously constructed block on $I_{m_\ast}$ is left unchanged. Hence $b_j(e)=0$ for all $j\notin I_{m_\ast}$. Thus,
$$
    V_e(x)
    =
    \sum_{j\in I_{m_\ast}}
    b_j(e)\eta_L(x-Lj),
$$
a finite sum of compactly supported smooth functions. Therefore $V_e\in C_c^\infty(\mathbb R)$, and in particular $V_e$ is compactly supported.
\end{proof}

\begin{proposition}
\label{prop:pp_dichotomy}
If $e \in \Tot$, then $\specpp(H_e^G) = \emptyset$. If $e \notin \Tot$, then $\specpp(H_e^G) \cap (-7/8, -5/8) \ne \emptyset$.
\end{proposition}

\begin{proof}
If $e \in \Tot$, then by \cref{lem:gordon_compact_supp}, $V_e$ is a generalized Gordon potential. By \cref{prop:gordon_pp}, we then have $\specpp(H_e^G) = \emptyset$.

Now suppose that $e \notin \Tot$. Recall that we have constructed $\psi \in C_c^\infty((L/3, 2L/3))$ such that $\norm \psi_{L^2} = 1$ and $\norm {(-\Delta - 1/4) \psi}_{L^2} < 1/8$. Since $\psi\in C_c^\infty(\mathbb R)\subseteq\operatorname{Dom}(H_e^G)$ is supported in $(L/3, 2L/3)$, $V_e(x) \psi(x) = -\eta_L(x) \psi(x)$ for each $x \in (L/3, 2L/3)$. Since $\eta_L=1$ on $(L/3, 2L/3)$, we have
$$
    \left(H_e^G+\frac34\right)\psi
    =
    \left(-\Delta-\frac14\right)\psi
$$
and
$$
\norm {\left(H_e^G + \frac 3 4\right) \psi}_{L^2} = \norm {\left(-\Delta - \frac 1 4\right) \psi}_{L^2} < \frac 1 8.
$$
Hence $\dist(- 3/4, \spec(H_e^G)) < 1/8$ so that $\spec(H_e^G) \cap (-7/8, -5/8) \ne \emptyset$. Since $V_e$ is compactly supported, by \cite[Theorem 9.38]{terschl}, $\specac(H_e^G) = [0, \infty)$ and $\specsc(H_e^G) = \emptyset$. Therefore $\spec(H_e^G)\cap(-7/8,-5/8)\subseteq\specpp(H_e^G)$. Since the set on the left is non-empty, $\specpp(H_e^G)\cap(-7/8,-5/8)\ne\emptyset$.
\end{proof}

Define the subfamily $(\mathcal O_{M, G}^\infty, \mathsf{Code}_{M, G}^\infty) = (\sequence {H_e^G}_{e \in \N}, \sequence {s_G(e)}_{e \in \N})$. The pure-point lower bound follows by applying \cref{prop:dichot} to this family.

\begin{theorem}
\label{thm:pp}
We have $(\mathcal O_{M, G}^\infty, \mathsf{Code}_{M, G}^\infty, \specpp) \notin \Delta_2^M$ and hence $(\mathcal O_M^\infty, \mathsf{Code}_M^\infty, \specpp) \notin \Delta_2^M$.
\end{theorem}

\begin{proof}
Let $B = D_{1/8}(-3/4)$ so that $B \cap \R = (-7/8, -5/8)$. If $e \in \Tot$, then by \cref{prop:pp_dichotomy} we have $\specpp(H_e^G) = \emptyset$ and hence $\specpp(H_e^G) \cap \overline B = \emptyset$. If $e \notin \Tot$, then by \cref{prop:pp_dichotomy}, we have $\specpp(H_e^G) \cap (-7/8, -5/8) \ne \emptyset$ and hence $\specpp(H_e^G) \cap B \ne \emptyset$. Recall from \cref{source_comp_pp} that the map $e \mapsto s_G(e)$ is computable. Suppose for contradiction that $(\mathcal O_{M, G}^\infty, \mathsf{Code}_{M, G}^\infty, \specpp) \in \Delta_2^M$. Then, by \cref{prop:dichot}, $\N \setminus \Tot$ is in $\Delta_2^0$. This would imply that $\Tot$ is in $\Delta_2^0$, contradicting \cref{prop:complete_class}.
\end{proof}

\section{Lower Bound for Absolutely Continuous Spectrum}
\label{sec:ac}

We prove the absolutely continuous lower bound by a second reduction from $\Tot$. The construction uses a theorem of Spencer and Simon: high-barrier Schr\"odinger potentials have empty absolutely continuous spectrum. We construct $U_e$ to have high barriers when $e\in\Tot$ and compact support otherwise. In the latter case the absolutely continuous spectrum is $[0,\infty)$.

\begin{proposition}[{\cite[Theorem 3.1]{traceclass}}]
\label{prop:spencer_simon}
Let $V : \R \to \R$ be measurable. Let $\sequence {x_n}_{n \in \Z}$ be a sequence of real numbers with $x_n \to \pm\infty$ as $n \to \pm\infty$. Let $\sequence {\ell_n}_{n \in \Z}$ and $\sequence {h_n}_{n \in \Z}$ be positive real numbers such that
\begin{itemize}
	\item $V$ is bounded below;
	\item $V(x) \ge h_n$ whenever $|x - x_n| \le \ell_n$;
	\item $h_n \to \infty$ as $n \to \pm \infty$;
	\item $h_n \ell_n^2 \to \infty$ as $n \to \pm \infty$.
\end{itemize}
Then $\specac(-\Delta + V) = \emptyset$.
\end{proposition}

Define a sequence of tuples $\sequence {(n_j, s_j)}_{j \in \N}$ by $(n_1, s_1) = (1, 1)$ and
\begin{equation}
\label{eqn:tuple_seq}
(n_{j + 1}, s_{j + 1}) = \begin{cases}(n_j + 1, 1) & \mathcal T_{e, s_j}(n_j) \downarrow, \\ (n_j, s_j + 1) & \mathcal T_{e, s_j}(n_j) \uparrow.\end{cases}
\end{equation}
The first component $n_j$ identifies the current input, while $s_j$ is the simulation bound. When $\mathcal T_{e,s_j}(n_j)\downarrow$, the recursion advances to the next input and resets the simulation bound to $1$; otherwise, it increases the simulation bound by one. Hence the pair $(n_j,s_j)$ is computable uniformly from $e$ and $j$ by a finite simulation. Define $a_e:\N\to\N$ by
$$
a_e(j) = \begin{cases}
n_j + 1 & \text { if } {\mathcal T}_{e, s_j}(n_j) \downarrow, \\
0 & \text{otherwise}.
\end{cases}
$$

\begin{lemma}
\label{lem:ae_dichot}
If $e \notin \Tot$, then $a_e(j) = 0$ for sufficiently large $j$. If $e \in \Tot$, then $a_e(j)$ is unbounded in $j$.
\end{lemma}

\begin{proof}
If $e \notin \Tot$, then there exists a smallest $n_\ast\in\N$ such that $\mathcal T_e(n_\ast) \uparrow$. If $n_\ast=1$, the condition that $\mathcal T_e(k)$ halts for $k<n_\ast$ is vacuous. Since $\mathcal T_e(k)$ halts for every $k<n_\ast$, there exists $j_0$ such that $n_{j_0}=n_\ast$. For $j>j_0$, the value $s_j$ increases by one at each subsequent stage because $\mathcal T_{e,s}(n_\ast)$ fails to halt for every $s\ge1$. In particular, $s_{j+1}\ne1$ for $j>j_0$, and hence $a_e(j)=0$ for all sufficiently large $j$.

If $e\in\Tot$, then for each $k\in\N$, let $m_k$ be the least stage such that $\mathcal T_{e,m_k}(k)\downarrow$. By induction on $k$, the pair $(k,m_k)$ occurs in the sequence $(n_j,s_j)$. Hence there is an index $j_k$ such that $a_e(j_k)=k+1$. Moreover, the recursion processes the inputs in increasing order, so the indices may be chosen with $j_k<j_{k+1}$. In particular, $j_k\to\infty$ and $a_e(j_k)\to\infty$, so $a_e$ is unbounded.
\end{proof}

By the same argument as in \cref{sec:pp}, we fix a computable $\beta \in C_c^\infty((-1/2, 1/2); [0, 1])$ such that
$$
\beta(x) = 1, \quad x \in [-1/3, 1/3].
$$
For each $k\in\N\cup\{0\}$, choose, computably in $k$, a rational number $B_k$ such that $\|\beta^{(k)}\|_\infty\le B_k$. For $j \ge 0$, set
$$
c_j = j + 1, \quad \beta_j^+(x) = \beta(x - c_j), \quad \beta_j^-(x) = \beta(x + c_j).
$$
Then $\beta_j^\pm$ is supported in $(\pm (j + 1) - 1/2, \pm (j + 1) + 1/2)$ and equals $1$ on $[\pm (j + 1) - 1/3, \pm (j + 1) + 1/3]$. For $e \in \N$, define
$$
U_e(x) = \sum_{j = 1}^\infty a_e(j) (\beta_j^+(x) + \beta_j^-(x)).
$$
The functions $\beta_j^+$ and $\beta_j^-$ also take values in $[0,1]$. Since at each point at most one translated bump is non-zero, $U_e\ge0$ and $U_e$ is locally bounded. Define
$$
    H_{e,0}^B f=-f''+U_e f,
    \qquad
    \operatorname{Dom}(H_{e,0}^B)=C_c^\infty(\mathbb R).
$$
Since $U_e\in L^2_{\mathrm{loc}}(\mathbb R)$ and $U_e(x)\ge0$, the Faris--Lavine theorem \cite[Theorem X.38]{reedsimon1975} implies that $H_{e,0}^B$ is essentially self-adjoint. We denote its self-adjoint closure by
$$
    H_e^B:=\overline{H_{e,0}^B}.
$$
This is the self-adjoint realization of $-\Delta+U_e$ used below.

\begin{lemma}
\label{lem:smooth_source_we}
For each $e \in \N$, $U_e$ is smooth. Further, there exists a computable function $s_B : \N \to \mathsf{Code}_M^\infty$ such that $s_B(e)$ is a valid smooth source code for $H_e^B$, equivalently for its potential $U_e$.
\end{lemma}

\begin{proof}
The supports of the functions $\beta_j^\pm$ are pairwise disjoint. For every $R\in\mathbb N$, only the indices $1\le j\le R+1$ can contribute at points of $[-R,R]$. Hence the defining sum is locally finite, $U_e\in C^\infty(\mathbb R)$, and, for $|x|\le R$, termwise differentiation gives $U_e^{(k)}(x)=\sum_{j=1}^{R+1}a_e(j)\bigl((\beta_j^+)^{(k)}(x)+(\beta_j^-)^{(k)}(x)\bigr)$. Given a computable real $x$, a suitable integer $R$ is computed from its certified approximation. The finitely many coefficients $a_e(j)$ and values $(\beta_j^\pm)^{(k)}(x)$ are computable. At the support endpoints the same formula applies because the smooth zero extension of $\beta$ and all its derivatives vanish there. Thus $U_e^{(k)}(x)$ is computable uniformly in $e$, $k$, and $x$.
The same finite range gives certified local bounds. Since at most one translated bump contributes at each point of $[-R,R]$, the $k$th derivative is bounded there by $B_k\max_{1\le j\le R+1}a_e(j)$. This quantity is computable from $e$, $k$, and $R$. Hence $U_e$ is smooth and $s_B(e)$ is a valid smooth source code.
\end{proof}

Define the subfamilies $\mathcal O_{M, B}^\infty = \sequence {H_e^B}_{e \in \N}$ and $\mathsf{Code}_{M, B}^\infty = \sequence {s_B(e)}_{e \in \N} \subseteq \mathsf{Code}_M^\infty$. The spectral dichotomy for $\mathcal O_{M, B}^\infty$ is the following.

\begin{proposition}
\label{prop:ae_dichot}
If $e \in \Tot$, then $\specac(H_e^B) = \emptyset$. If $e \notin \Tot$, then $\specac(H_e^B) = [0, \infty)$.

\end{proposition}

\begin{proof}
If $e \in \Tot$, then by \cref{lem:ae_dichot}, for each integer $n\ge1$ there exists $j_n \in \N$ such that $a_e(j_n) = n + 1$, with $j_n \to \infty$. Hence $U_e \ge n + 1$ on the intervals
$$
[c_{j_n} - 1/3, c_{j_n} + 1/3], \quad [-c_{j_n} - 1/3, -c_{j_n} + 1/3].
$$
Define the sequences in \cref{prop:spencer_simon} for every $n\in\mathbb Z$ by
$$
    x_n
    =
    \begin{cases}
        c_{j_n}, & n\ge1,\\
        c_{j_1}, & n=0,\\
        -c_{j_{-n}}, & n\le-1,
    \end{cases}
    \qquad
    \ell_n=\frac13,
    \qquad
    h_n=|n|+1.
$$
For $n=0$, the identity $a_e(j_1)=2$ gives $U_e(x)\ge2\ge h_0$ whenever $|x-x_0|\le\ell_0$. For $n\ne0$, the required barrier inequalities are precisely those displayed above. Moreover,
$$
    x_n\longrightarrow\pm\infty
    \quad\text{as }n\longrightarrow\pm\infty,
    \qquad
    h_n\longrightarrow\infty,
    \qquad
    h_n\ell_n^2=\frac{|n|+1}{9}\longrightarrow\infty.
$$
Thus all the hypotheses of \cref{prop:spencer_simon}, including the one indexed by $n=0$, hold, and hence $\specac(H_e^B) = \emptyset$.

If $e\notin\Tot$, then \cref{lem:ae_dichot} shows that $a_e$ is finitely supported. Thus $U_e\in C_c^\infty(\mathbb R)$, and \cite[Theorem 9.38]{terschl} gives $\specac(H_e^B)=[0,\infty)$.
\end{proof}

\begin{theorem}
\label{thm:ac}
We have $(\mathcal O_{M, B}^\infty, \mathsf{Code}_{M, B}^\infty, \specac) \notin \Delta_2^M$ and hence $(\mathcal O_M^\infty, \mathsf{Code}_M^\infty, \specac) \notin \Delta_2^M$.
\end{theorem}

\begin{proof}
Let $B=D_{1/4}(1)$, so that $B\cap\mathbb R=(3/4,5/4)$. If $e\in\Tot$, then by \cref{prop:ae_dichot} we have $\specac(H_e^B)=\emptyset$, and hence $\specac(H_e^B)\cap\overline B=\emptyset$. If $e\notin\Tot$, then by \cref{prop:ae_dichot} we have $\specac(H_e^B)=[0,\infty)$, and hence $\specac(H_e^B)\cap B\ne\emptyset$. Recall from \cref{lem:smooth_source_we} that the map $e\mapsto s_B(e)$ is computable. Suppose for contradiction that $(\mathcal O_{M, B}^\infty, \mathsf{Code}_{M, B}^\infty, \specac) \in \Delta_2^M$. Then \cref{prop:dichot}, applied with $S=\N\setminus\Tot$ and the ball $B$ above, implies that $\N \setminus \Tot$ is $\Delta_2^0$. This implies that $\Tot$ is $\Delta_2^0$, contradicting \cref{prop:complete_class}.
\end{proof}

\section{Lower Bound for Singular Continuous Spectrum}
\label{sec:sc}

The singular continuous lower bound uses a more involved inverse spectral construction. We encode the desired dichotomy in a measure, incorporate it into the spectral measure of a half-line Schr\"odinger operator, and extend the reconstructed potential smoothly to the whole line by reflection.

\subsection{Roadmap of the Construction}
The construction proceeds through the chain
$$
        A \longmapsto \nu_A \longmapsto \sigma_A
        \longmapsto \rho_A \longmapsto q_A
        \longmapsto Q_A(x)=q_A(|x|)
        \longmapsto H_A=-\Delta+Q_A.
$$
Here $A=(A_{k,m})$ is a computable binary matrix. Its columns are first encoded by Riesz-product measures and assembled into a measure $\nu_A$ on $\mathcal{I}=[5/4,7/4]$, so that $\nu_A$ has singular continuous mass on $\mathcal{I}$ if and only if some column of $A$ has infinitely many ones. We then add a computable absolutely continuous signed measure $\kappa_A$, supported away from $\mathcal{I}$, so that $\sigma_A=\nu_A+\kappa_A$ has all polynomial moments equal to zero. For a fixed small $\epsilon>0$, the positive measure $\rho_A=\rho_0+\epsilon\sigma_A$ is realized, via the Gelfand--Levitan equation, as the Dirichlet spectral measure of a half-line Schr\"odinger operator $-\frac{\mathrm d^2}{\mathrm dx^2}+q_A$. The moment cancellation makes the reconstructed potential flat at the origin, allowing the smooth reflection $Q_A(x)=q_A(|x|)$. The resulting whole-line operator decomposes into Dirichlet and Neumann half-line channels: the Dirichlet channel contains the encoded singular continuous dichotomy, while a Weyl $m$-function estimate excludes spurious singular continuous spectrum from the Neumann channel on the test interval. Embedding $\Cof$ into the column condition for $A$ then gives the desired $\Delta^M_3$ lower bound.

\subsection{Spectral Measures and Infinite Binary Sequences}

We begin with a Riesz-product measure parametrized by a binary sequence.

\begin{proposition}[{\cite[Theorems 1 and 2]{keogh_riesz_products}}]
\label{prop:zygmund}
Let $\{n_k\}_{k\in\mathbb{N}}$ be a sequence of positive integers such that $n_{k + 1}/n_k \ge 3$ for each $k$, and let $\{\alpha_k\}_{k\in\mathbb{N}}$ be a sequence of real numbers satisfying $0 \leq |\alpha_k| \le 1$. Define the trigonometric polynomial
$$
p_n(x) = \prod_{k = 1}^n (1 + \alpha_k \cos(2\pi n_k x)).
$$
Then there exists a continuous $F:\R\to\R$ such that
$$
F(x) - F(0) = \lim_{n \to \infty} \int_0^x p_n(t)\, \mathrm dt.
$$
If $\sum_k \alpha^2_k = \infty$, then $F' = 0$ almost everywhere.
\end{proposition}

The required measure is the weak limit in the following proposition.

\begin{proposition}[{\cite[Proposition 2.1.11]{simon2005orthogonal}}]
\label{prop:eta_tau}
Let $\mathrm d \theta$ be the normalized Haar measure on the torus $\mathbb T = \mathbb R / \mathbb Z$. For a binary sequence $\tau \in \set {0, 1}^\N$, define the measure
\begin{equation}
\label{eqn:etatau}
\mathrm d \eta_{\tau, N}(\theta) = \prod_{\substack{1 \le k \le N \\ \tau(k) = 1}} \left(1 + \frac{1}{2} \cos(2 \pi 3^k \theta)\right) \mathrm d \theta.
\end{equation}
Then the measures $\eta_{\tau, N}$ are probability measures and have a weak limit $\eta_\tau$ as $N\rightarrow\infty$.
\end{proposition}

For later use, let
$$
p_N(\theta)
:=
\frac{\mathrm d\eta_{\tau,N}}{\mathrm d\theta}(\theta)
=
\prod_{\substack{1\le k\le N\\ \tau(k)=1}}
\left(1+\frac12\cos(2\pi 3^k\theta)\right).
$$
These are the polynomials in \cref{prop:zygmund} obtained from the \emph{fixed} sequences
$$
n_k=3^k,
\qquad
\alpha_k=\frac12\tau(k),
\qquad k\in\mathbb N.
$$
Thus the binary sequence determines which Riesz factors are present. Finitely many factors give an absolutely continuous trigonometric density, whereas infinitely many factors will give a continuous singular measure.

Define the function
$$
\Phi(\theta) = \frac 3 2 + \frac 1 4 \cos(2 \pi \theta).
$$
The range of $\Phi$ is $\mathcal I=[5/4,7/4]$, so $\nu_\tau=\Phi_\#\eta_\tau$ is supported in $\mathcal I$. The following adaptation of \cite[Theorem 2.11.3]{simon2005orthogonal} relates its spectral type to the number of ones in $\tau$.

\begin{lemma}
\label{lem:dichotomy1}
Let $\tau \in \set {0, 1}^\N$. The following statements hold:
\begin{enumerate}
	\item if $\tau$ has finitely many ones, then $\nu_\tau$ is absolutely continuous with respect to the Lebesgue measure;
	\item if $\tau$ has infinitely many ones, then $\nu_\tau$ is singular continuous with respect to the Lebesgue measure and $\spec(\nu_\tau) = \mathcal{I}$.
\end{enumerate}
\end{lemma}

\begin{proof}
If $\tau$ has finitely many ones, choose $N_0$ larger than every $k$ for which $\tau(k)=1$. For every $N\ge N_0$, the product in \eqref{eqn:etatau} then contains exactly the finitely many nontrivial factors selected by $\tau$. Hence $\eta_{\tau,N}$ is independent of $N\ge N_0$ and equals its weak limit $\eta_\tau$. In particular, $\eta_\tau$ has a trigonometric-polynomial density with respect to $\mathrm d\theta$ and is absolutely continuous. We can partition $\mathbb T = (0, 1/2) \cup (1/2, 1) \cup \set {0, 1/2,1}$. On the open sets $(0, 1/2)$ and $(1/2, 1)$, $\Phi$ is a diffeomorphism, while $\set {0, 1/2,1}$ has measure zero. Hence $\Phi^{-1}(E)$ is Haar-null for each Lebesgue-null set $E$, and $\Phi_\# \eta_\tau$ is absolutely continuous.

Suppose that $\tau$ has infinitely many ones. Set
$$
\alpha_k=\frac12\tau(k),
\qquad
n_k=3^k.
$$
Then $n_{k+1}/n_k=3$ and
$$
\sum_{k=1}^{\infty}\alpha_k^2
=
\frac14\#\{k:\tau(k)=1\}
=
\infty.
$$
Moreover, $p_N$ is precisely the density of $\eta_{\tau,N}$. Hence \cref{prop:zygmund} yields a continuous function $F:\mathbb R\to\mathbb R$ such that
\begin{equation}
\label{eqn:weak_conv}
F(x)-F(0)
=
\lim_{N\to\infty}\int_0^x p_N(t)\,\mathrm dt
\end{equation}
for every $x\in\mathbb R$, and $F'=0$ almost everywhere.

For $0\le x\le1$, let
$$
G_N(x)=\int_0^x p_N(t)\,\mathrm dt,
\qquad
G(x)=F(x)-F(0).
$$
Each $G_N$ is non-decreasing, $G_N(0)=0$, and $G_N(1)=1$. Consequently, $G$ is a continuous non-decreasing function satisfying $G(0)=0$ and $G(1)=1$, and hence is the distribution function of a Lebesgue--Stieltjes probability measure $\mu_G$ on $[0,1]$, viewed also as a measure on $\mathbb T$. Since $G_N\to G$ pointwise and $G$ is continuous, the standard convergence theorem for distribution functions gives
$
\eta_{\tau,N}\rightarrow\mu_G.
$
On the other hand, \cref{prop:eta_tau} gives $\eta_{\tau,N}\rightarrow\eta_\tau$. Uniqueness of weak limits therefore implies
$
\eta_\tau=\mu_G.
$
The continuity of $G$ shows that $\eta_\tau$ has no atoms. Moreover, $G'=0$ almost everywhere, so the absolutely continuous part of $\mu_G$ is zero. Thus $\eta_\tau$ is singular continuous.

Choose a Haar-null Borel set $E\subseteq\mathbb T$ with $\eta_\tau(E)=1$. Since $\Phi$ is Lipschitz, $\Phi(E)$ has Lebesgue measure zero, and $E\subseteq\Phi^{-1}(\Phi(E))$. Thus $1\ge\nu_\tau(\Phi(E))=\eta_\tau(\Phi^{-1}(\Phi(E)))\ge\eta_\tau(E)=1$. Hence $\nu_\tau(\Phi(E))=1$, so $\nu_\tau$ is singular. For $y\in\mathcal I$, the fiber $\Phi^{-1}(\{y\})$ contains at most two points. Since $\eta_\tau$ has no atoms, $\nu_\tau(\{y\})=0$. Therefore $\nu_\tau$ is singular continuous.

To determine the support, first observe that $ \operatorname{supp}\eta_\tau=\mathbb T. $ Define the triadic interval
$$
J = \left[\frac m {3^M}, \frac {m + 1} {3^M}\right] \subseteq \mathbb T.
$$
Take $N \ge M$ and factor
$$
\prod_{\substack{1 \le k \le N \\ \tau(k) = 1}} \left(1 + \frac{1}{2} \cos(2 \pi 3^k \theta)\right) = L_M(\theta) H_{M, N}(\theta),
$$
where
$$
L_M(\theta) = \prod_{\substack{1 \le k < M \\ \tau(k) = 1}} \left(1 + \frac{1}{2} \cos(2 \pi 3^k \theta)\right), \quad H_{M, N}(\theta) = \prod_{\substack{M \le k \le N \\ \tau(k) = 1}} \left(1 + \frac{1}{2} \cos(2 \pi 3^k \theta)\right).
$$
Every factor in both products is contained in $[1/2, 3/2]$, so in particular we have $L_M(\theta) \ge 1/2^M$. Rewrite
$$
H_{M, N}(\theta) = \prod_{\substack{0 \le k \le N - M \\ \tau(M + k) = 1}} \left(1 + \frac{1}{2} \cos(2 \pi 3^{M+k} \theta)\right).
$$
With $S_{M,N}=\set {0\le k\le N-M:\tau(M+k)=1}$, the expansion is
$$
H_{M, N}(\theta) = \sum_{\epsilon \in \set {-1, 0, 1}^{S_{M, N}}} \left(\prod_{k \in S_{M, N}} c_{\epsilon_k}\right) \exp \left(2 \pi i \theta \sum_{k \in S_{M, N}} \epsilon_k 3^{M + k}\right).
$$
Since the endpoints of $J$ are triadics of level $M$, every nonconstant exponential has a nonzero frequency that is an integer multiple of $3^M$ and therefore integrates to zero on $J$. Thus, the only term that contributes to the integral has $\epsilon_k = 0$ for all $k$. This gives $\int_J H_{M, N}(\theta) \mathrm d \theta = |J|$. It follows that
$$
\eta_{\tau, N}(J) \ge \frac{|J|}{2^M}  \quad \forall N \ge M.
$$
Since $\eta_{\tau, N}$ converges weakly to $\eta_\tau$ and $J$ is closed, the Portmanteau theorem gives
$$
\limsup_{N \to \infty} \eta_{\tau, N}(J) \le \eta_\tau(J).
$$
We therefore have $\eta_\tau(J)>0$. Every non-empty open subset of $\mathbb T$ contains a triadic interval, and hence has positive $\eta_\tau$-measure. Thus $ \operatorname{supp}\eta_\tau=\mathbb T. $ Now let $U$ be a non-empty relatively open subset of $\mathcal I$. Since $\Phi:\mathbb T\to\mathcal I$ is continuous and surjective, $\Phi^{-1}(U)$ is a non-empty open subset of $\mathbb T$. Therefore
$$
\nu_\tau(U)
=
\eta_\tau\bigl(\Phi^{-1}(U)\bigr)
>
0.
$$
Since $\nu_\tau$ is supported in $\mathcal I$, it follows that $ \spec(\nu_\tau)=\mathcal I. $
\end{proof}

\subsection{Weighted Sums of Measures}

Consider a computable binary matrix $A = \sequence {A_{k, m}}_{k, m \in \N} \in \set {0, 1}^{\N \times \N}$. Define the $m$th column $\tau_m(k) = A_{k, m}$ and $\Lambda_\Phi = \Phi_\# \mathrm d \theta$. This is an absolutely continuous measure since $\Phi$ is $C^1$ with only two critical points. We combine the measures $\Lambda_\Phi$ and $\nu_{\tau_m}$ in the convergent sum
$$
\nu_A = \frac 1 2 \Lambda_\Phi + \sum_{m = 1}^\infty 2^{-m - 1} \nu_{\tau_m}.
$$
Both $\Lambda_\Phi$ and every $\nu_{\tau_m}$ are probability measures supported in $\mathcal I$. Moreover,
$$
\frac12+\sum_{m=1}^{\infty}2^{-m-1}
=
\frac12+\frac12
=
1.
$$
Hence the series converges in total variation and defines a positive probability measure $\nu_A$ supported in $\mathcal I$. The $m$th column contributes a spectral measure of weight $2^{-(m+1)}$. Thus, a single infinite column is enough to force singular continuous mass, while if all columns are finite, the entire sum remains absolutely continuous.

\begin{lemma}\label{extra_needed_lemma}
The measure $\nu_A$ satisfies the following dichotomy:
\begin{enumerate}
	\item if some column of $A$ has infinitely many ones, then $(\nu_A)_{\mathrm{sc}}$ gives positive measure to every non-empty open subinterval of $\mathcal{I}$;
	\item if every column of $A$ has only finitely many ones, then $\nu_A$ is absolutely continuous.
\end{enumerate}
\end{lemma}

\begin{proof}
Suppose first that some column of $A$ has infinitely many ones. Lebesgue decomposition commutes with positive countable sums. Since $\tfrac 1 2 \Lambda_\Phi$ is absolutely continuous, we have
$$
(\nu_A)_{\mathrm{sc}} = \sum_{m = 1}^\infty 2^{-m - 1} (\nu_{\tau_m})_{\mathrm{sc}}.
$$
If the $m$th column has infinitely many ones, then by \cref{lem:dichotomy1}, $\nu_{\tau_m}$ is singular continuous with support $\mathcal{I}$. Hence $(\nu_A)_{\mathrm{sc}}$ gives positive measure to every non-empty open subinterval of $\mathcal{I}$.

If every column of $A$ is finitely supported, then by \cref{lem:dichotomy1}, each $\nu_{\tau_m}$ is absolutely continuous. Then $\nu_A$ is absolutely continuous as the sum of absolutely continuous measures.
\end{proof}

A $C^2$ source code for the test function, together with its derivative bounds, makes integration against $\nu_A$ computable.

\begin{lemma}[Uniform computability of the integral]
\label{lem:integral_com_nu}
Assume that the data defining $\nu_A$, namely the sequence $\{\tau_m\}_{m\ge 1}$, are given effectively, uniformly in $m$. Then there is an algorithm which, given (i) a $C^2$ source code for a computable function $\phi\in C^2(\mathcal I)$, (ii) a rational number $B\ge0$ satisfying $\|\partial^j\phi\|_\infty\le B$ for $j=0,1,2$, and (iii) a requested precision $p\in\mathbb N$, outputs a rational number $q$ satisfying
$$
    \left|q-\int_{\mathcal{I}} \phi\,\mathrm{d}\nu_A\right|<2^{-p}.
$$
\end{lemma}

\begin{proof}
Let $g=\phi\circ\Phi$. Then $g$ is a computable $C^2$ function on $\mathbb T$, uniformly in the input $\phi$. Moreover, $\|g\|_\infty\le B$, and, by the chain rule,
$$
    g''(x)
    =
    \phi''(\Phi(x))(\Phi'(x))^2
    +
    \phi'(\Phi(x))\Phi''(x),
$$
which we may bound explicitly by some $D$ using $B$.

For a given sequence $\tau$, the integral $\int_{\mathbb T}g\,\mathrm d\eta_\tau$ is computed from the Fourier coefficients of $\eta_\tau$. For the finite Riesz products,
$$
    \widehat \eta_{\tau,N}(n)
    =
    \sum_{\substack{\epsilon\in\{-1,0,1\}^{S_N}\\
    \sum_{k\in S_N}\epsilon_k3^k=n}}
    \prod_{k\in S_N} c_{\epsilon_k},
$$
where $S_N=\{1\le k\le N:\tau(k)=1\}$, $c_0=1$, $c_{1}=c_{-1}=1/4$. The representation
$$
    n=\sum_k \epsilon_k3^k,
    \qquad
    \epsilon_k\in\{-1,0,1\},
$$
is unique when it exists. Indeed, if a non-trivial relation existed, then the largest power $3^K$ appearing would dominate the contribution of all smaller powers $3^K>2\sum_{j<K}3^j=3^K-1$, which is impossible.

Thus $\widehat\eta_\tau(n)$ can be computed from finitely many values of $\tau$. More explicitly, for $n\neq 0$, compute the balanced ternary expansion
$$
    n=\sum_{k=0}^K \epsilon_k3^k,
    \qquad
    \epsilon_k\in\{-1,0,1\}.
$$
If $\epsilon_0\neq 0$, or if $\epsilon_k\neq 0$ for some $k\ge 1$ with $\tau(k)=0$, then $\widehat\eta_\tau(n)=0$. Otherwise, if $r=\#\{k\ge 1:\epsilon_k\neq 0\}$, then $\widehat\eta_\tau(n)=\left(\frac 14\right)^r$. For $n=0$, of course, $\widehat\eta_\tau(0)=1$. Since every non-zero digit in the balanced ternary expansion satisfies $3^k\le 2|n|$, only the values $\tau(k)$ with $k\le \lceil \log_3(2|n|)\rceil$ have to be queried. Hence $\widehat\eta_\tau(n)$ is computable, uniformly in $n$ and in $\tau$.

Next, for $n\neq 0$, integration by parts twice gives $\widehat g(n)=\tfrac{\widehat{g''}(n)}{(2\pi i n)^2}, $ and therefore
$$
    |\widehat g(n)|
    \le
    \frac{\|g''\|_\infty}{4\pi^2 n^2}
    \le
    \frac{D}{4\pi^2 n^2}.
$$
Thus $\sum_{n\in\mathbb Z} |\widehat g(n)|<\infty$. Since $|\widehat\eta_\tau(n)|\le 1$, the series $\sum_{n\in\mathbb Z}\widehat g(n)\widehat\eta_\tau(-n)$ is absolutely convergent. Moreover,
$$
    \sum_{|n|>R}
    |\widehat g(n)\widehat\eta_\tau(-n)|
    \le
    2\sum_{n>R}\frac{D}{4\pi^2 n^2}
    \le
		\frac{D}{2\pi^2}\int_{R}^\infty \frac{\mathrm{d}x}{x^2}
    =\frac{D}{2\pi^2 R}.
$$
Therefore, the tail of the series is bounded in terms of $B$ and $R$.

The absolute convergence of the Fourier series of $g$ justifies termwise integration:
$$
    \int_{\mathbb T} g\,\mathrm{d}\eta_\tau=
    \int_{\mathbb T}
    \left(\sum_{n\in\mathbb Z}\widehat g(n)\mathrm e^{2\pi i n x}\right)
    \,\mathrm{d}\eta_\tau(x) =
    \sum_{n\in\mathbb Z}
    \widehat g(n)
    \int_{\mathbb T} \mathrm e^{2\pi i n x}\,\mathrm{d}\eta_\tau(x) =
    \sum_{n\in\mathbb Z}\widehat g(n)\widehat\eta_\tau(-n).
$$
The preceding tail estimate gives a uniform algorithm for $\int_{\mathbb T}g\,\mathrm d\eta_\tau$. Given $\delta>0$, choose $R$ so that $D/(2\pi^2R)<\delta/2$. For each $|n|\le R$, compute $\widehat\eta_\tau(-n)$ by the finite procedure above. Each $\widehat g(n)$ is computable uniformly in $\phi$: the $C^2$ source code provides a computable modulus of continuity for the integrand in its defining Fourier integral. Compute these coefficients accurately enough that the finite-sum error is below $\delta/2$. The omitted tail is also smaller than $\delta/2$. Hence $\int_{\mathbb T}g\,\mathrm d\eta_\tau$, and therefore $\int_{\mathcal I}\phi\,\mathrm d\nu_{\tau_m}$, is computable uniformly in $m$ and $\phi$.

The definition of $\nu_A$ gives
$$
    \int_{\mathcal{I}} \phi\,\mathrm{d}\nu_A
    =
    \frac12\int_{\mathcal{I}}\phi\,\mathrm{d}\Lambda_\Phi
    +
    \sum_{m=1}^\infty 2^{-m-1}
    \int_{\mathcal{I}}\phi\,\mathrm{d}\nu_{\tau_m}
		=\frac12\widehat g(0)
    +
    \sum_{m=1}^\infty 2^{-m-1}
    \int_{\mathcal{I}}\phi\,\mathrm{d}\nu_{\tau_m}.
$$
Given $p\in\mathbb N$, choose $J$ so large that $B\,2^{-J-1}<2^{-p-1}$. Since each $\nu_{\tau_m}$ is a probability measure,
$$
    \left|
    \sum_{m=J+1}^\infty 2^{-m-1}
    \int_{\mathcal{I}}\phi\,\mathrm{d}\nu_{\tau_m}
    \right|
    \le
    \sum_{m=J+1}^\infty 2^{-m-1}\|\phi\|_\infty
    \le
    B\,2^{-J-1}
    <
    2^{-p-1}.
$$
Thus, it suffices to compute the finite expression
$$
    \frac12\widehat g(0)
    +
    \sum_{m=1}^{J}2^{-m-1}
    \int_{\mathcal{I}}\phi\,\mathrm{d}\nu_{\tau_m}
$$
to within $2^{-p-1}$. Each term in this finite sum is computable uniformly in $\phi$, as shown above. Since the total weight of these finitely many terms is at most $1$, it is enough to compute each integral to accuracy $2^{-p-1}$. Combining this finite approximation error with the tail estimate gives an approximation to $\int_{\mathcal{I}}\phi\,\mathrm{d}\nu_A$ within $2^{-p}$. All choices made above are computable from the input $C^2$ source code of $\phi$, the bound $B$, the requested precision $p$, and the data defining $\nu_A$.
\end{proof}

\subsection{The Moment-Killing Measure}

The measure $\nu_A$ contains the required spectral information, but the inverse spectral argument also requires a perturbation with vanishing moments. The following Schwartz function supplies the correction.

\begin{lemma}
\label{lem:h_moments}
There exists a real-valued computable Schwartz function $h$ with $\operatorname{supp} h\subset [4,\infty)$ such that
$$
    \int_4^\infty u^n h(u)\,\mathrm{d}u = -1,
    \qquad n=0,1,2,\ldots.
$$
Moreover, $h$ may be chosen with computable bounds on all Schwartz seminorms: for every $M,\ell\in \mathbb{N}\cup\{0\}$ one can compute a rational number $C_{M,\ell}$ such that
$$
    \sup_{u\in\mathbb R}
    (1+|u|)^M |\partial_u^\ell h(u)|
    \le C_{M,\ell}.
$$
\end{lemma}

\begin{proof}
Choose a non-negative computable function $\alpha\in C_c^\infty((0,1))$ such that $\int_{\mathbb R} \alpha(u)\,\mathrm{d}u = 1$, and such that there is a computable sequence of rational numbers $A_r$ satisfying
$$
    \|\alpha^{(r)}\|_\infty \le A_r,
    \qquad r=0,1,2,\ldots.
$$
For example, one may take a translated and dilated standard $C^\infty$ bump function supported in $(0,1)$, and then normalize it by its integral. For rational numbers $S>4$ and $L>0$, define
$$
    \psi_{S,L}(u)
    =
    L^{-1}\alpha\left(\frac{u-S}{L}\right).
$$
Then $\psi_{S,L}$ is supported in $[S,S+L]$ and satisfies $\int_{\mathbb R}\psi_{S,L}(u)\,\mathrm{d}u=1$. For $j\ge 0$, define
$$
    \psi_{j,S,L}(u)
    =
    \frac{(-1)^j}{j!}\partial_u^j\psi_{S,L}(u).
$$
If $k\ge j$, then integration by parts gives
$$
    \int_{\mathbb R} u^k \psi_{j,S,L}(u)\,\mathrm{d}u
    =
    \frac{(-1)^j}{j!}
    \int_{\mathbb R} u^k \partial_u^j\psi_{S,L}(u)\,\mathrm{d}u
    =
    \frac{1}{j!}
    \int_{\mathbb R} \partial_u^j(u^k)\psi_{S,L}(u)\,\mathrm{d}u
    =
    \binom{k}{j}
    \int_{\mathbb R} u^{k-j}\psi_{S,L}(u)\,\mathrm{d}u.
$$
If $k<j$, then $\partial_u^j(u^k)=0$, and therefore $\int_{\mathbb R} u^k \psi_{j,S,L}(u)\,\mathrm{d}u=0$. Thus
\begin{equation}
\label{eqn:moments_psi}
    \int_{\mathbb R} u^k \psi_{j,S,L}(u)\,\mathrm{d}u
    =
    \begin{cases}
        0, & k<j, \\[2mm]
        1, & k=j, \\[2mm]
        \displaystyle
        \binom{k}{j}
        \int_{\mathbb R} u^{k-j}\psi_{S,L}(u)\,\mathrm{d}u,
        & k>j.
    \end{cases}
\end{equation}

We now construct pairwise disjoint rational intervals
$$
    [S_j,S_j+L_j],
    \qquad j=0,1,2,\ldots,
$$
with $S_j\to\infty$, and computable coefficients $c_j$. Set $\psi_j=\psi_{j,S_j,L_j}$, $c_0=-1$ and $S_0=5$. More generally, suppose that $c_i,\ S_i,\ L_i$ have already been constructed for $i<j$. Define
$$
    c_j
    =
    -1
    -
    \sum_{i=0}^{j-1}
    c_i
    \int_{\mathbb R} u^j\psi_i(u)\,\mathrm{d}u.
$$
For $j=0$, this is interpreted as $c_0=-1$. The number $c_j$ is computable from the previously constructed data, since each integral $\int_{\mathbb R} u^j\psi_i(u)\,\mathrm{d}u$ is computable from $\alpha$, $S_i$, and $L_i$. Choose a rational number $\varsigma_j>|c_j|$. For $j\ge 1$, choose a rational number $S_j$ such that
$
    S_j>S_{j-1}+L_{j-1}+1.
$
This guarantees that the intervals are pairwise disjoint. Having chosen $S_j$, choose a rational number $L_j>0$ so large that, for every $M,\ell\le \lfloor j/2\rfloor$,
\begin{equation}
\label{eqn:choice_L_j}
    \frac{\varsigma_j A_{j+\ell}}{j!}
    L_j^{-j-\ell-1}
    (1+S_j+L_j)^M
    \le 2^{-j}.
\end{equation}
For fixed $j$, $S=S_j$, and $M,\ell\le \lfloor j/2\rfloor$, the left-hand side is $O(L^{M-j-\ell-1})$ as $L\to\infty$. Since
$$
    M-j-\ell-1
    \le
    -j/2-\ell-1
    <0,
$$
it tends to $0$. One can find a rational $L_j$ satisfying \eqref{eqn:choice_L_j}. For $j=0$, choose $L_0>0$ so that the same condition holds for $M=\ell=0$.

This choice is computable. Indeed,
$$
    \partial_u^\ell\psi_{j,S,L}(u)
    =
    \frac{(-1)^j}{j!}
    L^{-j-\ell-1}
    \alpha^{(j+\ell)}
    \left(\frac{u-S}{L}\right),
$$
and on the support of $\psi_{j,S,L}$ we have $|u|\le S+L$. Hence
$$
    \sup_{u\in\mathbb R}
    (1+|u|)^M
    |c_j \partial_u^\ell\psi_{j,S,L}(u)|
    \le
    \frac{\varsigma_j A_{j+\ell}}{j!}
    L^{-j-\ell-1}
    (1+S+L)^M.
$$
Consequently,
\begin{equation}
\label{eqn:unif_schwartz}
    \sup_{u\in\mathbb R}
    (1+|u|)^M
    |c_j\partial_u^\ell\psi_j(u)|
    \le 2^{-j}
\end{equation}
whenever $M,\ell\le \lfloor j/2\rfloor$. Define
$
    h(u)
    =
    \sum_{j=0}^\infty c_j\psi_j(u).
$
The intervals $[S_j,S_j+L_j]$ are pairwise disjoint, and each $\psi_j$ is supported in $[S_j,S_j+L_j]$. Hence at each point $u$ at most one term in the sum is non-zero. In particular, $h$ is supported in $[S_0,\infty)\subset [4,\infty)$.

The same construction gives computable bounds for every Schwartz seminorm. Taking $M=\ell=0$ in \eqref{eqn:unif_schwartz}, we obtain $\|c_j\psi_j\|_\infty\le 2^{-j}$. Thus the defining series for $h$ converges uniformly with a computable rate, and since the partial sums are computable functions, $h$ is computable. Similarly, for fixed $\ell$, the series $\sum_{j=0}^\infty c_j\partial_u^\ell\psi_j(u)$ converges uniformly, because \eqref{eqn:unif_schwartz} with $M=0$ applies to all sufficiently large $j$. Therefore $h\in C^\infty(\mathbb R)$, and $\partial_u^\ell h(u)=\sum_{j=0}^\infty c_j\partial_u^\ell\psi_j(u)$. Now fix $M,\ell\in\mathbb{N}\cup\{0\}$, and set $J=2\max(M,\ell)$. For $j\ge J$, estimate \eqref{eqn:unif_schwartz} gives
$$
    \sup_{u\in\mathbb R}
    (1+|u|)^M
    |c_j\partial_u^\ell\psi_j(u)|
    \le 2^{-j}.
$$
For $j<J$, the explicit bound
$$
    \sup_{u\in\mathbb R}
    (1+|u|)^M
    |c_j\partial_u^\ell\psi_j(u)|
    \le
    \frac{\varsigma_j A_{j+\ell}}{j!}
    L_j^{-j-\ell-1}
    (1+S_j+L_j)^M
$$
is computable. Hence
$$
    \sup_{u\in\mathbb R}
    (1+|u|)^M
    |\partial_u^\ell h(u)|
    \le
    \sum_{j=0}^{J-1}
    \frac{\varsigma_j A_{j+\ell}}{j!}
    L_j^{-j-\ell-1}
    (1+S_j+L_j)^M
    +
    \sum_{j=J}^\infty 2^{-j}.
$$
The right-hand side is a computable finite expression plus the explicit tail $\sum_{j=J}^\infty 2^{-j}=2^{1-J}$. Thus we can compute a rational constant $C_{M,\ell}$ such that
$$
    \sup_{u\in\mathbb R}
    (1+|u|)^M
    |\partial_u^\ell h(u)|
    \le C_{M,\ell}.
$$
This proves the claimed computable bounds on all Schwartz seminorms of $h$.

For the moments, fix $n\ge0$. For all sufficiently large $j$, namely for $j\ge 2(n+2)$, estimate \eqref{eqn:unif_schwartz} with $M=n+2$ and $\ell=0$ gives $|c_j\psi_j(u)|\le2^{-j}(1+|u|)^{-n-2}$. Therefore
$$
    |u|^n |c_j\psi_j(u)|
    \le
    2^{-j}
    \frac{|u|^n}{(1+|u|)^{n+2}},
$$
and the function ${|u|^n}/{(1+|u|)^{n+2}}$ is integrable on $\mathbb R$. Hence
$$
    \sum_{j=0}^\infty
    \int_{\mathbb R}
    |u|^n |c_j\psi_j(u)|\,\mathrm{d}u
    <\infty.
$$
We may therefore interchange the sum and the integral:
$$
    \int_{\mathbb R} u^n h(u)\,\mathrm{d}u
    =
    \sum_{i=0}^\infty
    c_i
    \int_{\mathbb R} u^n\psi_i(u)\,\mathrm{d}u.
$$
By \eqref{eqn:moments_psi}, the terms with $i>n$ vanish, and the term with $i=n$ is $c_n$. Hence
$$
    \int_{\mathbb R} u^n h(u)\,\mathrm{d}u
    =
    \sum_{i=0}^{n}
    c_i
    \int_{\mathbb R} u^n\psi_i(u)\,\mathrm{d}u
    =
    \sum_{i=0}^{n-1}
    c_i
    \int_{\mathbb R} u^n\psi_i(u)\,\mathrm{d}u
    +
    c_n.
$$
The recursive definition of $c_n$ makes the last expression equal to $-1$. Since $\operatorname{supp}h\subset[4,\infty)$, this is precisely $\int_4^\infty u^nh(u)\,\mathrm du=-1$.
\end{proof}

Fix $h$ as in \cref{lem:h_moments}. For $t>0$, let $h_t(E)=t^{-1}h(E/t)$. Then $\operatorname{supp}h_t\subset[4t,\infty)$. Since $\nu_A$ is supported in $\mathcal I$, define the absolutely continuous signed measure $\kappa_A$ by
$$
    \,\mathrm{d}\kappa_A(E)=g_A(E)\,\mathrm{d}E,
\quad\text{where}\quad
    g_A(E)
    =
    \int_{5/4}^{7/4} h_t(E)\,\mathrm{d}\nu_A(t).
$$

\begin{proposition}
\label{lem:kappa_a_fv}
The function $g_A$ is computable and Schwartz, with computable Schwartz bounds. In particular, $\kappa_A$ has finite variation.
\end{proposition}

\begin{proof}
First note that, since $h_t$ is supported in $[4t,\infty)$ and $t\in\mathcal{I}$, we have $\operatorname{supp} g_A \subset [5,\infty)$. For $\ell\ge 0$, differentiation gives $\partial_E^\ell h_t(E)=t^{-\ell-1}h^{(\ell)}(E/t)$. To differentiate under the integral sign, let $s\ne0$. By the mean value theorem, for each fixed $t\in\mathcal{I}$, there is a point $E_\ast$ between $E$ and $E+s$ such that
$$
    \left|
    \frac{h_t(E+s)-h_t(E)}{s}
    \right|
    =
    |h_t'(E_\ast)|
    \le
    t^{-2}\|h'\|_\infty.
$$
Because $t^{-2}\le(4/5)^2$ on $\mathcal I$, this bound is $\nu_A$-integrable. The dominated convergence theorem therefore gives
$$
    g_A'(E)
    =
    \int_{5/4}^{7/4} h_t'(E)\,\mathrm{d}\nu_A(t).
$$
Repeating the same argument, using the uniform bounds
$$
    |\partial_E^{\ell+1}h_t(E)|
    \le
    \left(\frac45\right)^{\ell+2}
    \|h^{(\ell+1)}\|_\infty,
$$
we obtain, for every $\ell\ge 0$,
$$
    g_A^{(\ell)}(E)
    =
    \int_{5/4}^{7/4}
    t^{-\ell-1}h^{(\ell)}(E/t)\,\mathrm{d}\nu_A(t).
$$
Thus $g_A\in C^\infty(\mathbb R)$.

For the Schwartz estimates, fix $M,\ell\in\mathbb N\cup\{0\}$. Since $t\in\mathcal{I}$, we have
$$
    1+|E|
    \le
    \frac74(1+|E/t|).
$$
Therefore
$$
\begin{aligned}
    (1+|E|)^M |g_A^{(\ell)}(E)|
    &\le
    \int_{5/4}^{7/4}
    (1+|E|)^M t^{-\ell-1}
    |h^{(\ell)}(E/t)|\,\mathrm{d}\nu_A(t) \\
    &\le
    \left(\frac74\right)^M
    \left(\frac45\right)^{\ell+1}
    \int_{5/4}^{7/4}
    (1+|E/t|)^M
    |h^{(\ell)}(E/t)|\,\mathrm{d}\nu_A(t).
\end{aligned}
$$
Since $\nu_A$ is a probability measure and computable bounds are available for all Schwartz seminorms of $h$, there is a computable constant $C_{M,\ell}^{(h)}$ such that $\sup_{u\in\mathbb R}(1+|u|)^M |h^{(\ell)}(u)|\le C_{M,\ell}^{(h)}$. Consequently,
$$
    \sup_{E\in\mathbb R}
    (1+|E|)^M |g_A^{(\ell)}(E)|
    \le
    \left(\frac74\right)^M
    \left(\frac45\right)^{\ell+1}
    C_{M,\ell}^{(h)}.
$$
The right-hand side is computable from the computable Schwartz bounds for $h$. Hence $g_A$ is Schwartz with computable Schwartz bounds.

To compute $g_A$, fix $E$. The function $t\mapsto t^{-1}h(E/t)$ is a computable $C^2$ function on $\mathcal{I}$, uniformly in $E$ on compact sets, with computable $C^2$ bounds obtained from the computable Schwartz bounds for $h$. Therefore, by the uniform computability of integration against $\nu_A$ (\cref{lem:integral_com_nu}), the value
$$
    g_A(E)
    =
    \int_{5/4}^{7/4} t^{-1}h(E/t)\,\mathrm{d}\nu_A(t)
$$
is computable uniformly in $E$. The same argument applies to $g_A^{(\ell)}$, using the formula above.

Since $g_A\in\mathcal S(\mathbb R)\subset L^1(\mathbb R)$, the signed measure $\mathrm d\kappa_A(E)=g_A(E)\,\mathrm dE$ has finite variation.
\end{proof}

Let $\sigma_A=\nu_A+\kappa_A$, with $\nu_A$ viewed as a measure on $[0,\infty)$. This signed measure is finite because $|\sigma_A|\le\nu_A+|\kappa_A|$. The next lemma gives the moment cancellation.

\begin{lemma}
\label{lem:sigma_moments}
For every $n\in\mathbb N\cup\{0\}$, we have
$
    \int_0^\infty E^n\,\mathrm{d}\sigma_A(E)=0.
$
\end{lemma}

\begin{proof}
Since $g_A$ is Schwartz and supported in $[5,\infty)$, all polynomial moments of $\kappa_A$ are finite. Moreover, since $h$ is Schwartz and $t\in\mathcal{I}$, we have
$$
    \int_{5/4}^{7/4}
    \int_0^\infty
    E^n t^{-1}|h(E/t)|\,\mathrm{d}E\,\mathrm{d}\nu_A(t)
    =
    \int_{5/4}^{7/4}
    t^n
    \int_0^\infty
    u^n |h(u)|\,\mathrm{d}u\,\mathrm{d}\nu_A(t) <\infty.
$$
Fubini's theorem therefore yields
$$
    \int_0^\infty E^n\,\mathrm{d}\kappa_A(E)
    =
    \int_0^\infty
    E^n
    \left(
        \int_{5/4}^{7/4} t^{-1}h(E/t)\,\mathrm{d}\nu_A(t)
    \right)
    \,\mathrm{d}E =
    \int_{5/4}^{7/4}
    \left(
        \int_0^\infty E^n t^{-1}h(E/t)\,\mathrm{d}E
    \right)
    \,\mathrm{d}\nu_A(t).
$$
Changing variables $E=tu$ in the inner integral gives
$$
\begin{aligned}
    \int_0^\infty E^n t^{-1}h(E/t)\,\mathrm{d}E
    &=
    t^n\int_0^\infty u^n h(u)\,\mathrm{d}u.
\end{aligned}
$$
By \Cref{lem:h_moments}, $\int_0^\infty u^n h(u)\,\mathrm{d}u=-1$. Therefore
$$
    \int_0^\infty E^n\,\mathrm{d}\kappa_A(E)
    =
    -\int_{5/4}^{7/4} t^n\,\mathrm{d}\nu_A(t).
$$
Since $\nu_A$ is supported in $\mathcal{I}$, we also have
$$
    \int_0^\infty E^n\,\mathrm{d}\nu_A(E)
    =
    \int_{5/4}^{7/4} t^n\,\mathrm{d}\nu_A(t).
$$
Adding the last two identities gives $\int_0^\infty E^n\,\mathrm{d}\sigma_A(E)=0$.
\end{proof}

We perturb the free measure
$$
    \,\mathrm{d}\rho_0(E)
    =
    \1_{(0,\infty)}(E)\frac{\sqrt E}{\pi}\,\mathrm{d}E
$$
by a sufficiently small multiple of $\sigma_A$.

\begin{proposition}
\label{prop:rhoA}
There exists a rational number $\epsilon>0$, independent of $A$, such that $\rho_A=\rho_0+\epsilon\sigma_A$ is a positive measure and $2\rho_A\ge \rho_0$.
\end{proposition}

\begin{proof}
The density $g_A$ admits a bound independent of $A$. Since $h$ is supported in $[4,\infty)$, the quantity $h(E/t)$ can be non-zero only if $E/t\ge 4$, or equivalently $E\ge 4t$. Hence, whenever $h(E/t)\neq 0$, $E^{-1/2}\le\tfrac{1}{2}t^{-1/2}$. Therefore, for $E>0$,
$$
    E^{-1/2}\sup_{t\in\mathcal{I}} t^{-1}|h(E/t)|
    \le
    \frac12
    \sup_{t\in\mathcal{I}} t^{-3/2}|h(E/t)|
    \le
    \frac12
    \left(\frac45\right)^{3/2}
    \|h\|_\infty.
$$
Using the computable Schwartz bounds for $h$, choose a rational constant $C_h>0$, independent of $A$, such that
$$
    E^{-1/2}\sup_{t\in\mathcal{I}} t^{-1}|h(E/t)|
    \le
    C_h\quad \forall E>0.
$$
Since $\nu_A$ is a probability measure, it follows that
$$
    |g_A(E)|
    =
    \left|
    \int_{5/4}^{7/4} t^{-1}h(E/t)\,\mathrm{d}\nu_A(t)
    \right|
    \le
    \int_{5/4}^{7/4} t^{-1}|h(E/t)|\,\mathrm{d}\nu_A(t)
    \le
    C_h\sqrt E\quad \forall E>0.
$$
Choose a rational number $\epsilon>0$ such that $\epsilon C_h\le \frac{1}{2\pi}$. For instance, since $\pi<4$, it is enough to take $\epsilon=\frac{1}{8C_h}$. Then, for every $E>0$,
$$
    \frac{\sqrt E}{\pi}+\epsilon g_A(E)
    \ge
    \frac{\sqrt E}{\pi}-\epsilon |g_A(E)|
    \ge
    \frac{\sqrt E}{\pi}-\epsilon C_h\sqrt E
    \ge
    \frac12\frac{\sqrt E}{\pi}.
$$
Hence the absolutely continuous measure $\rho_0+\epsilon\kappa_A$ satisfies $\rho_0+\epsilon\kappa_A\ge \tfrac12\rho_0$. Since $\nu_A$ is a positive measure, we obtain $\rho_A=\rho_0+\epsilon\kappa_A+\epsilon\nu_A\ge\rho_0/2$.
\end{proof}

\subsection{Identifying the Potential}

We consider boundary conditions of the form
$$
    y(0)\cos\alpha+y'(0)\sin\alpha=0.
$$
We call this the \emph{Dirichlet} boundary condition if $\sin\alpha=0$, since then $y(0)=0$. We call it the \emph{Neumann} boundary condition if $\cos\alpha=0$, since then $y'(0)=0$. Let
$$
    s_E(x)
    =
    \begin{cases}
        \displaystyle \frac{\sin(\sqrt E x)}{\sqrt E}, & E>0, \\[2mm]
        x, & E=0, \\[2mm]
        \displaystyle \frac{\sinh(\sqrt{-E}x)}{\sqrt{-E}}, & E<0.
    \end{cases}
$$
Thus $s_E$ is the solution of the free equation
$$
    -s_E''=E s_E,
    \qquad
    s_E(0)=0,
    \qquad
    s_E'(0)=1.
$$
Let $\rho_0$ denote the free Dirichlet spectral measure,
$$
    \,\mathrm d\rho_0(E)
    =
    \1_{(0,\infty)}(E)\frac{\sqrt E}{\pi}\,\mathrm dE.
$$
To realize $\rho_A$ as the Dirichlet spectral measure of the half-line Schr\"odinger operator
$$
    -\frac{\mathrm d^2}{\mathrm dx^2}+q_A
$$
with Dirichlet boundary condition at the origin, we use Remling's reconstruction theorem \cite[Theorem 19.1(a)--(b)]{remling2002}. In the Dirichlet normalization, part~(a) reconstructs a unique locally integrable potential from each $\rho\in\mathrm{GL}$, while parts~(a)--(b) characterize the spectral measures that arise in this way. The relevant definition is as follows.

\begin{definition}[Remling]
\label{def:gl}
A positive Borel measure $\rho$ on $\mathbb R$ belongs to $\mathrm{GL}$ if the following conditions hold.
\begin{enumerate}[label=\textup{(\roman*)}]
    \item For every $N>0$ and every $f\in L^2(0,N)$, the identity
    $$
        \int_{\mathbb R}
        \left|
            \int_0^N f(x)s_E(x)\,\mathrm dx
        \right|^2
        \,\mathrm d\rho(E)
        =
        0
    $$
    implies that $f=0$ almost everywhere.
    \item For every $g\in C_c^\infty(\mathbb R)$, the integral
    $$
        \int_{\mathbb R}
        \left|
            \int_{\mathbb R} g(x)s_E(x)\,\mathrm dx
        \right|
        \,\mathrm d|\rho-\rho_0|(E)
    $$
    is finite.
    \item There exists an odd, real-valued function $\phi\in L^1_{\mathrm{loc}}(\mathbb R)$ such that, for every $g\in C_c^\infty(\mathbb R)$,
    $$
        \int_{\mathbb R}
        \left(
            \int_{\mathbb R} g(x)s_E(x)\,\mathrm dx
        \right)
        \,\mathrm d(\rho-\rho_0)(E)
        =
        -\int_{\mathbb R} g(x)\phi(x)\,\mathrm dx.
    $$
\end{enumerate}
\end{definition}

Remling's theorem gives the required potential.

\begin{proposition}[Remling {\cite[Theorem 19.1(a)]{remling2002}}]
\label{prop:meas_to_pot}
For every $\rho\in\mathrm{GL}$, there exists a unique potential $V_\rho\in L^1_{\mathrm{loc}}([0,\infty))$ such that $\rho$ is a spectral measure of $-\frac{\mathrm d^2}{\mathrm dx^2}+V_\rho$ on $[0,\infty)$ with Dirichlet boundary condition at the origin.
\end{proposition}

For the measures considered below, the reconstruction takes the classical Dirichlet Gelfand--Levitan form; see \cite[Chapter~2, Section~2.3]{levitan1987}. Define
\begin{equation}
\label{eqn:f_first}
    F_\rho(x,y)
    =
    \int_{\mathbb R}
    s_E(x)s_E(y)\,\mathrm d(\rho-\rho_0)(E).
\end{equation}
In our application, $\rho-\rho_0$ has finite variation and is supported in $[0,\infty)$, so the integral in \eqref{eqn:f_first} is an ordinary absolutely convergent integral for each fixed $x,y\ge0$. For $x>0$, the Dirichlet Gelfand--Levitan equation is
\begin{equation}
\label{eqn:first_integral}
    K_\rho(x,y)
    +
    F_\rho(x,y)
    +
    \int_0^x K_\rho(x,t)F_\rho(t,y)\,\mathrm dt
    =
    0,
    \qquad
    0\le y\le x.
\end{equation}
Once $\rho\in\mathrm{GL}$ has been verified, this equation has a unique solution $K_\rho(x,\cdot)\in L^2(0,x)$. Indeed, the quadratic form of its homogeneous Fredholm operator is
$$
    \int_{\mathbb R}
    \left|
        \int_0^x h(t)s_E(t)\,\mathrm dt
    \right|^2
    \,\mathrm d\rho(E),
$$
which vanishes only when $h=0$ by condition~(i) in \cref{def:gl}; the Fredholm alternative therefore gives uniqueness and existence. The classical reconstruction formula is
$$
    V_\rho(x)
    =
    2\frac{\mathrm d}{\mathrm dx}K_\rho(x,x)
    \quad\text{for almost every $x>0$},
$$
as in \cite[Chapter~2, Section~2.3]{levitan1987}. It is therefore enough to verify that $\rho_A$ satisfies the conditions in \cref{def:gl}. The resulting potential is $q_A=V_{\rho_A}$, for which $\rho_A$ is the Dirichlet spectral measure of $-{\mathrm d^2}/{\mathrm dx^2}+q_A$. Here $\rho_A-\rho_0=\epsilon\sigma_A$.

\begin{lemma}
\label{lem:fourier_sine}
For each $f\in C_c^\infty(0,\infty)$, we have
$$
    \|f\|_{L^2(0,\infty)}^2
    =
    \frac{2}{\pi}
    \int_0^\infty
    \left|
        \int_0^\infty f(x)\sin(kx)\,\mathrm{d}x
    \right|^2
    \,\mathrm{d}k.
$$
In particular, if the right-hand side is zero, then $f=0$ almost everywhere.
\end{lemma}

\begin{proof}
We use the Fourier transform convention
$$
    \widehat F(k)
    =
    \int_{\mathbb R} F(x)\mathrm e^{-ikx}\,\mathrm{d}x,
$$
so that Plancherel's theorem takes the form $2\pi\|F\|_{L^2(\mathbb R)}^2=\|\widehat F\|_{L^2(\mathbb R)}^2$. Extend $f$ oddly to the real line by
$$
    f_o(x)
    =
    \begin{cases}
        f(x), & x>0, \\
        0, & x=0, \\
        -f(-x), & x<0.
    \end{cases}
$$
Since $f\in C_c^\infty(0,\infty)$, we have $f_o\in L^1(\mathbb R)\cap L^2(\mathbb R)$. Moreover,
$$
    \|f_o\|_{L^2(\mathbb R)}^2
    =
    2\int_0^\infty |f(x)|^2\,\mathrm{d}x
    =
    2\|f\|_{L^2(0,\infty)}^2.
$$
Since $f_o$ is odd, the function $x\mapsto f_o(x)\cos(kx)$ is odd, while $x\mapsto f_o(x)\sin(kx)$ is even. Therefore
$$
    \widehat{f_o}(k)
    =
    \int_{\mathbb R} f_o(x)\cos(kx)\,\mathrm{d}x
    -
    i\int_{\mathbb R} f_o(x)\sin(kx)\,\mathrm{d}x
    =
    -2i\int_0^\infty f(x)\sin(kx)\,\mathrm{d}x.
$$
Applying Plancherel gives
$$
    2\pi\|f_o\|_{L^2(\mathbb R)}^2
    =
    \int_{-\infty}^{\infty} |\widehat{f_o}(k)|^2\,\mathrm{d}k
    =
    4\int_{-\infty}^{\infty}
    \left|
        \int_0^\infty f(x)\sin(kx)\,\mathrm{d}x
    \right|^2
    \,\mathrm{d}k.
$$
The integrand is even in $k$, since $\sin(-kx)=-\sin(kx)$. Hence
$$
    4\pi\|f\|_{L^2(0,\infty)}^2
    =
    8\int_0^\infty
    \left|
        \int_0^\infty f(x)\sin(kx)\,\mathrm{d}x
    \right|^2
    \,\mathrm{d}k.
$$
Dividing by $4\pi$ proves the identity.
\end{proof}

\begin{corollary}
\label{cor:isom}
The map initially defined on $C_c^\infty(0,\infty)$ by
$$
    f
    \mapsto
    \left(
        E\mapsto
        \int_0^\infty f(x)s_E(x)\,\mathrm{d}x
    \right)
$$
extends uniquely to an isometry $L^2(0,\infty)\rightarrow L^2((0,\infty),\rho_0)$. In other words, for every $f\in L^2(0,\infty)$, we have
$$
    \int_0^\infty
    \left|
        \mathcal S f(E)
    \right|^2
    \,\mathrm{d}\rho_0(E)
    =
    \|f\|_{L^2(0,\infty)}^2,
$$
where $\mathcal S f$ denotes the $L^2(\rho_0)$-extension of $\int_0^\infty f(x)s_E(x)\,\mathrm{d}x$. Moreover, if $N>0$, $f\in L^2(0,N)$, and
$$
    \int_0^\infty
    \left|
        \int_0^N f(x)s_E(x)\,\mathrm{d}x
    \right|^2
    \,\mathrm{d}\rho_A(E)
    =
    0,
$$
then $f=0$ almost everywhere on $(0,N)$.
\end{corollary}

\begin{proof}
First take $f\in C_c^\infty(0,\infty)$. For $E=k^2$ with $k>0$,
$$
    s_E(x)
    =
    \frac{\sin(kx)}{k},
    \qquad
    \,\mathrm{d}\rho_0(E)
    =
    \frac{\sqrt E}{\pi}\,\mathrm{d}E
    =
    \frac{2k^2}{\pi}\,\mathrm{d}k.
$$
Therefore
$$
    \int_0^\infty
    \left|
        \int_0^\infty f(x)s_E(x)\,\mathrm{d}x
    \right|^2
    \,\mathrm{d}\rho_0(E)
    =
    \int_0^\infty
    \left|
        \frac{1}{k}
        \int_0^\infty f(x)\sin(kx)\,\mathrm{d}x
    \right|^2
    \frac{2k^2}{\pi}\,\mathrm{d}k
    =
    \frac{2}{\pi}
    \int_0^\infty
    \left|
        \int_0^\infty f(x)\sin(kx)\,\mathrm{d}x
    \right|^2
    \,\mathrm{d}k.
$$
By \cref{lem:fourier_sine}, this is equal to $\|f\|_{L^2(0,\infty)}^2$. Thus, the transform is an isometry on $C_c^\infty(0,\infty)$, and since $C_c^\infty(0,\infty)$ is dense in $L^2(0,\infty)$, it extends uniquely to an isometry on all of $L^2(0,\infty)$.

Now let $N>0$ and $f\in L^2(0,N)$. The pointwise integral $\int_0^N f(x)s_E(x)\,\mathrm{d}x$ is well-defined for every $E\ge 0$, since $s_E\in L^2(0,N)$. Moreover, this pointwise transform agrees $\rho_0$-almost everywhere with the isometric extension above, after extending $f$ by zero outside $(0,N)$. Suppose that
$$
    \int_0^\infty
    \left|
        \int_0^N f(x)s_E(x)\,\mathrm{d}x
    \right|^2
    \,\mathrm{d}\rho_A(E)
    =
    0.
$$
Since $2\rho_A\ge \rho_0$, we have
$$
    \int_0^\infty
    \left|
        \int_0^N f(x)s_E(x)\,\mathrm{d}x
    \right|^2
    \,\mathrm{d}\rho_0(E)
    \le
    2
    \int_0^\infty
    \left|
        \int_0^N f(x)s_E(x)\,\mathrm{d}x
    \right|^2
    \,\mathrm{d}\rho_A(E)
    =
    0.
$$
By the isometry just proved, $\|f\|_{L^2(0,N)}^2=0$. Hence $f=0$ almost everywhere on $(0,N)$.
\end{proof}

Applying \cref{cor:isom} on $L^2(0,N)$ for arbitrary $N>0$ verifies condition~(i) in \cref{def:gl}. Conditions~(ii) and~(iii) are treated next.

\begin{proposition}
\label{prop:verify_2}
Let $f\in C_c^\infty(\mathbb R)$, and define
$$
    f_s(E)
    =
    \int_{-\infty}^{\infty} f(x)s_E(x)\,\mathrm{d}x.
$$
Then $f_s\in L^1(|\rho_A-\rho_0|)$. Equivalently, since $\rho_A-\rho_0=\epsilon\sigma_A$, we have $f_s\in L^1(\epsilon|\sigma_A|)$.
\end{proposition}

\begin{proof}
Recall that $\sigma_A=\nu_A+\kappa_A$. The measure $\nu_A$ is supported in $\mathcal{I}$, while $\kappa_A$ is supported in $[5,\infty)$. Hence $\operatorname{supp}\sigma_A\subset [5/4,\infty)$. For $E\ge 5/4$, we have
$$
    |s_E(x)|
    =
    \left|\frac{\sin(\sqrt E x)}{\sqrt E}\right|
    \le
    E^{-1/2}
    \le
    \sqrt{\frac45}
$$
for all $x\in\mathbb R$. Therefore
$$
    |f_s(E)|
    \le
    \sqrt{\frac45}\int_{-\infty}^{\infty}|f(x)|\,\mathrm{d}x
    =
    \sqrt{\frac45}\,\|f\|_{L^1(\mathbb R)}
$$
for every $E\in\operatorname{supp}\sigma_A$.

By \cref{lem:kappa_a_fv}, $\kappa_A$ has finite variation; since $\nu_A$ is finite, so does $\sigma_A$. The preceding bound therefore gives $f_s\in L^1(|\sigma_A|)$. As $|\rho_A-\rho_0|=\epsilon|\sigma_A|$, the claim follows.
\end{proof}

Condition~(iii) follows from the next proposition.

\begin{proposition}
\label{prop:verify_3}
Define
$$
    \phi(x)
    =
    -\epsilon
    \int_{5/4}^{\infty} s_E(x)\,\mathrm{d}\sigma_A(E).
$$
Then $\phi$ is odd, real-valued, and bounded. In particular, $\phi\in L^1_{\mathrm{loc}}(\mathbb R)$. Moreover, for every $f\in C_c^\infty(\mathbb R)$,
$$
    \epsilon
    \int_{-\infty}^{\infty}
    \left(
        \int_{-\infty}^{\infty} f(x)s_E(x)\,\mathrm{d}x
    \right)
    \,\mathrm{d}\sigma_A(E)
    =
    -\int_{-\infty}^{\infty} f(x)\phi(x)\,\mathrm{d}x.
$$
Equivalently,
$$
    \int_{-\infty}^{\infty}
    \left(
        \int_{-\infty}^{\infty} f(x)s_E(x)\,\mathrm{d}x
    \right)
    \,\mathrm{d}(\rho_A-\rho_0)(E)
    =
    -\int_{-\infty}^{\infty} f(x)\phi(x)\,\mathrm{d}x.
$$
\end{proposition}

\begin{proof}
The defining integral for $\phi$ is absolutely convergent for every $x$. Indeed, using again that $\operatorname{supp}\sigma_A\subset[5/4,\infty)$, we have
$$
    |s_E(x)|
    \le
    E^{-1/2}
    \le
    \sqrt{\frac45}
$$
for $E\in\operatorname{supp}\sigma_A$. Therefore
$$
    |\phi(x)|
    \le
    \epsilon\sqrt{\frac45}\,|\sigma_A|(\mathbb R).
$$
The displayed bound proves that $\phi$ is bounded. It is real-valued because $s_E$ and $\sigma_A$ are real, and odd because $s_E$ is odd in $x$.

For the identity, let $f\in C_c^\infty(\mathbb R)$. The following estimate proves absolute integrability with respect to $\mathrm dx\,\mathrm d|\sigma_A|(E)$:
$$
    \int_{5/4}^{\infty}
    \int_{-\infty}^{\infty}
    |f(x)|\,|s_E(x)|
    \,\mathrm{d}x\,\mathrm{d}|\sigma_A|(E)
    \le
    \sqrt{\frac45}\,\|f\|_{L^1(\mathbb R)}|\sigma_A|(\mathbb R)
    <\infty.
$$
Fubini's theorem therefore gives
$$
    -\int_{-\infty}^{\infty} f(x)\phi(x)\,\mathrm{d}x
    =
    \epsilon
    \int_{-\infty}^{\infty}
    f(x)
    \left(
        \int_{5/4}^{\infty} s_E(x)\,\mathrm{d}\sigma_A(E)
    \right)
    \,\mathrm{d}x
    =
    \epsilon
    \int_{5/4}^{\infty}
    \left(
        \int_{-\infty}^{\infty} f(x)s_E(x)\,\mathrm{d}x
    \right)
    \,\mathrm{d}\sigma_A(E).
$$
Since $\sigma_A$ is supported in $[5/4,\infty)$, the result follows.
\end{proof}

The preceding propositions show that $\rho_A\in\mathrm{GL}$. Hence \cref{prop:meas_to_pot} gives a unique $q_A\in L^1_{\mathrm{loc}}([0,\infty))$ for which $\rho_A$ is a Dirichlet spectral measure of
$$
\tau_A=-\frac{\mathrm d^2}{\mathrm dx^2}+q_A
$$
on $[0,\infty)$. It remains to establish the regularity and growth needed for the corresponding self-adjoint half-line realizations. Define
$$
    F_A(x,y)
    =
    \int_0^\infty s_E(x)s_E(y)\,\mathrm{d}(\rho_A-\rho_0)(E)
    =
    \epsilon
    \int_0^\infty s_E(x)s_E(y)\,\mathrm{d}\sigma_A(E).
$$
The associated Gelfand--Levitan kernel $K_A$ satisfies
\begin{equation}
\label{eqn:integral_restate}
    K_A(x,y)
    +
    F_A(x,y)
    +
    \int_0^x K_A(x,t)F_A(t,y)\,\mathrm{d}t
    =
    0,
    \qquad
    0\le y\le x.
\end{equation}

\begin{lemma}
\label{lem:fa_vanish}
The function $F_A$ is smooth, and all of its mixed derivatives are bounded. Moreover,
$$
    \partial_x^a\partial_y^b F_A(0,0)=0
$$
for all non-negative integers $a,b$.
\end{lemma}

\begin{proof}
Since
$
    \rho_A-\rho_0=\epsilon\sigma_A
    =
    \epsilon\nu_A+\epsilon\kappa_A.
$
Since $\nu_A$ is supported in $\mathcal I$ and $\kappa_A$ has density $g_A$ supported in $[5,\infty)$, the preceding identity becomes
$$
    F_A(x,y)
    =
    \epsilon
    \int_{5/4}^{7/4}
    s_E(x)s_E(y)\,\mathrm{d}\nu_A(E)
    +
    \epsilon
    \int_5^\infty
    s_E(x)s_E(y)g_A(E)\,\mathrm{d}E.
$$
For every $a\ge 0$, $\partial_x^a s_E(x)$ is a sine or cosine factor multiplied by a power of $E$. More precisely,
$$
    |\partial_x^a s_E(x)|
    \le
    \begin{cases}
        E^{-1/2}, & a=0, \\[1mm]
        E^{(a-1)/2}, & a\ge 1.
    \end{cases}
$$
So for fixed $a,b$, the product $\partial_x^a s_E(x)\,\partial_y^b s_E(y)$ is bounded by a fixed power of $1+E$, uniformly in $x,y$.

On the compact interval $\mathcal{I}$, this gives a uniform bound, so the $\nu_A$-integral may be differentiated under the integral sign. On $[5,\infty)$, the same polynomial bounds are integrable against $|g_A(E)|\,\mathrm{d}E$, since $g_A$ is Schwartz by \cref{lem:kappa_a_fv}. Thus, by the dominated convergence theorem, $F_A$ is smooth and
$$
    \partial_x^a\partial_y^b F_A(x,y)
    =
    \epsilon
    \int_0^\infty
    \partial_x^a s_E(x)\,\partial_y^b s_E(y)
    \,\mathrm{d}\sigma_A(E).
$$
The same estimates also show that every mixed derivative of $F_A$ is bounded.

At $x=0$, the derivatives satisfy
$$
    \partial_x^a s_E(0)
    =
    \begin{cases}
        0, & a \text{ is even}, \\[1mm]
        (-1)^{(a-1)/2}E^{(a-1)/2}, & a \text{ is odd}.
    \end{cases}
$$
Thus $\partial_x^a\partial_y^b\bigl(s_E(x)s_E(y)\bigr)\big|_{(x,y)=(0,0)}$ is identically zero unless both $a$ and $b$ are odd. If both $a$ and $b$ are odd, then it is equal to $\pm E^{(a+b)/2-1}$. Here $(a+b)/2-1\in\mathbb{N}\cup\{0\}$. By \cref{lem:sigma_moments},
$$
    \int_0^\infty E^n\,\mathrm{d}\sigma_A(E)=0
$$
for every $n\in\mathbb{N}\cup\{0\}$. Hence, in the remaining case, we also obtain
$$
    \partial_x^a\partial_y^b F_A(0,0)
    =
    \pm\epsilon
    \int_0^\infty E^{(a+b)/2-1}\,\mathrm{d}\sigma_A(E)
    =
    0.
$$
Therefore all mixed derivatives of $F_A$ vanish at $(0,0)$.
\end{proof}

\begin{corollary}
\label{cor:ka_smooth}
The kernel $K_A$ is smooth on the closed triangle $\{(x,y):0\le y\le x\}$. Consequently, the function $x\mapsto K_A(x,x)$ is smooth on $[0,\infty)$.
\end{corollary}

\begin{proof}
By \cite[Lemma 2.3.1 and the discussion following it]{levitan1987}, the solution of the Gelfand--Levitan equation inherits the smoothness of $F_A$. Since $F_A$ is smooth by \cref{lem:fa_vanish}, the kernel $K_A$ is smooth on the closed triangle.
\end{proof}

The value at the degenerate endpoint follows by setting $y=0$ in \eqref{eqn:integral_restate} and letting $x\downarrow0$. By \cref{cor:ka_smooth} and \cref{lem:fa_vanish}, $K_A$ and $F_A$ are continuous on the closed triangle, and
$$
\left|
\int_0^x K_A(x,t)F_A(t,0)\,\mathrm dt
\right|
\le
x
\sup_{0\le t\le x}|K_A(x,t)|
\sup_{0\le t\le x}|F_A(t,0)|
\longrightarrow0.
$$
Hence $K_A(0,0)+F_A(0,0)=0$, and \cref{lem:fa_vanish} gives $K_A(0,0)=0$.

The flatness of $F_A$ transfers to $q_A$.

\begin{proposition}
\label{prop:qa_flat}
All derivatives of $K_A(x,x)$ vanish at $x=0$. Consequently, after taking the smooth representative
$$
    q_A(x)=2\frac{\mathrm d}{\mathrm dx}K_A(x,x),
$$
all derivatives of $q_A$ vanish at $x=0$.
\end{proposition}

\begin{proof}
Let $R_A(x)=K_A(x,x)$. Evaluating \eqref{eqn:integral_restate} at $y=x$ gives
$$
    R_A(x)
    =
    -F_A(x,x)
    -
    \int_0^x K_A(x,t)F_A(t,x)\,\mathrm{d}t.
$$
By \cref{cor:ka_smooth}, $K_A$ is bounded on every compact triangle $0\le t\le x\le \delta$. Fix $N\in\mathbb N$. Since all mixed derivatives of $F_A$ vanish at $(0,0)$, Taylor's theorem gives a constant $C_N$ and $\delta_N>0$ such that
$$
    |F_A(u,v)|
    \le
    C_N (u+v)^N
$$
whenever $0\le u,v\le \delta_N$. In particular, for $0\le t\le x\le \delta_N$,
$$
    |F_A(x,x)|
    \le
    C_N (2x)^N,
    \qquad
    |F_A(t,x)|
    \le
    C_N (t+x)^N
    \le
    C_N (2x)^N.
$$
Let $B_N=\sup_{0\le t\le x\le \delta_N}|K_A(x,t)|$. Then, for $0\le x\le \delta_N$,
$$
    |R_A(x)|
    \le
    |F_A(x,x)|
    +
    \int_0^x |K_A(x,t)|\,|F_A(t,x)|\,\mathrm{d}t
    \le
    C_N 2^N x^N
    +
    B_N C_N 2^N x^{N+1}.
$$
The estimate gives $R_A(x)=O(x^N)$ as $x\downarrow0$. Since $N$ is arbitrary and $R_A$ is smooth, every derivative of $R_A$ vanishes at zero. The reconstruction formula $q_A=2R_A'$ then yields $q_A^{(m)}(0)=0$ for every $m\ge0$.
\end{proof}

\subsection{Computing the Potential}
\label{subsec:comput_potential}

We construct a smooth source code for $q_A$. Given a compact interval and a derivative order, the algorithm computes certified uniform approximations and derivative bounds. The construction approximates $F_A$ in \eqref{eqn:f_first}, solves \eqref{eqn:first_integral}, and differentiates the resulting kernel, with explicit error control at each stage.

Recall that
$$
    F_A(x,y)
    =
    \int_0^\infty s_E(x)s_E(y)\,\mathrm{d}(\rho_A-\rho_0)(E)
    =
    \epsilon
    \int_0^\infty s_E(x)s_E(y)\,\mathrm{d}\sigma_A(E),
    \qquad x,y\ge 0,
$$
where $\epsilon$ is the rational number chosen above.

\begin{lemma}
\label{lem:FA}
Let $X\in\mathbb Q_{>0}$, $N\in\mathbb{N}\cup\{0\}$, and $\delta\in\mathbb Q_{>0}$. We can compute a tensor-product $C^N$ spline $F_{A,N,\delta,X}$ on $[0,X]^2$, together with a computable bound $C_{N,X}(\delta)$, such that
$$
    \max_{0\le i,j\le N}
    \left\|
        \partial_x^i\partial_y^j
        \bigl(F_A-F_{A,N,\delta,X}\bigr)
    \right\|_{L^\infty([0,X]^2)}
    \le
    C_{N,X}(\delta),
$$
where $C_{N,X}(\delta)\to 0$ as $\delta\to 0^+$. In particular, given $p\in\mathbb N$, one can compute $\delta$ so that the right-hand side is at most $2^{-p}$.
\end{lemma}

\begin{proof}
The derivatives of $F_A$ are computable pointwise, uniformly on compact sets, with computable uniform bounds. Fix $a,b\in\mathbb{N}\cup\{0\}$, and set
$
    \Psi_{a,b,x,y}(E)
    =
    \partial_x^a s_E(x)\,\partial_y^b s_E(y).
$
Since $\rho_A-\rho_0=\epsilon\sigma_A=\epsilon\nu_A+\epsilon\kappa_A$, and since $\,\mathrm{d}\kappa_A(E)=g_A(E)\,\mathrm{d}E$, we have
$$
    \partial_x^a\partial_y^b F_A(x,y)
    =
    \epsilon
    \int_{5/4}^{7/4}
    \Psi_{a,b,x,y}(E)\,\mathrm{d}\nu_A(E)
    +
    \epsilon
    \int_0^\infty
    \Psi_{a,b,x,y}(E)g_A(E)\,\mathrm{d}E.
$$
Every derivative $\partial_x^a s_E(x)$ is a sine or cosine factor multiplied by a power of $E$. More precisely, for $E\ge 5/4$,
$$
    |\partial_x^a s_E(x)|
    \le
    C_a(1+E)^{a/2},
$$
with a computable constant $C_a$. Hence, for each fixed $a,b$, the function $\Psi_{a,b,x,y}(E)$ and the finitely many $E$-derivatives needed below are bounded by computable constants times fixed powers of $1+E$, uniformly for $(x,y)\in[0,X]^2$.

Consider first the integral against $\nu_A$. For fixed rational $x,y$, the function $ E\mapsto \Psi_{a,b,x,y}(E) $ is a computable $C^2$ function on $\mathcal{I}$, with computable $C^2$ bounds, uniformly for $x,y\in[0,X]$. Therefore, by \cref{lem:integral_com_nu},
$$
    \int_{5/4}^{7/4}
    \Psi_{a,b,x,y}(E)\,\mathrm{d}\nu_A(E)
$$
can be computed to arbitrary accuracy, uniformly in $x,y$.

For the absolutely continuous term, the definition of $g_A$ gives
$$
\begin{aligned}
    \int_0^\infty \Psi_{a,b,x,y}(E)g_A(E)\,\mathrm{d}E
    &=
    \int_0^\infty
    \Psi_{a,b,x,y}(E)
    \left(
        \int_{5/4}^{7/4} t^{-1}h(E/t)\,\mathrm{d}\nu_A(t)
    \right)
    \,\mathrm{d}E.
\end{aligned}
$$
The polynomial growth of $\Psi_{a,b,x,y}$ and the uniform rapid decay of $h(E/t)$ make the integrand absolutely integrable. Fubini's theorem gives
$$
\begin{aligned}
    \int_0^\infty \Psi_{a,b,x,y}(E)g_A(E)\,\mathrm{d}E
    &=
    \int_{5/4}^{7/4}
    \left(
        \int_0^\infty
        \Psi_{a,b,x,y}(E)t^{-1}h(E/t)\,\mathrm{d}E
    \right)
    \,\mathrm{d}\nu_A(t).
\end{aligned}
$$
Changing variables $E=tu$ gives
$$
    \int_0^\infty \Psi_{a,b,x,y}(E)g_A(E)\,\mathrm{d}E
    =
    \int_{5/4}^{7/4}
    \Theta_{a,b,x,y}(t)\,\mathrm{d}\nu_A(t),\qquad\text{where}\qquad
    \Theta_{a,b,x,y}(t)
    =
    \int_0^\infty
    \Psi_{a,b,x,y}(tu)h(u)\,\mathrm{d}u.
$$
Since $h$ is supported in $[4,\infty)$, this is equivalently $\Theta_{a,b,x,y}(t)=\int_4^\infty\Psi_{a,b,x,y}(tu)h(u)\,\mathrm{d}u$.

The function $\Theta_{a,b,x,y}$ is computable in $C^2(\mathcal I)$, uniformly for $x,y\in[0,X]$, with computable $C^2$-bounds. For $r=0,1,2$, the derivative $\partial_t^r\Psi_{a,b,x,y}(tu)$ is bounded by a computable polynomial in $u$, uniformly for $t\in\mathcal{I}$ and $x,y\in[0,X]$. Thus there are computable constants $C$ and $m$ such that
$$
    \left|
        \partial_t^r\Psi_{a,b,x,y}(tu)
    \right|
    \le
    C(1+u)^m,
    \qquad r=0,1,2.
$$
The computable Schwartz-seminorm bounds for $h$ allow us to compute $H_m$ such that
$$
    |h(u)|
    \le
    H_m(1+u)^{-m-2}.
$$
Consequently the tails
$$
    \int_U^\infty
    \partial_t^r\Psi_{a,b,x,y}(tu)h(u)\,\mathrm{d}u,
    \qquad r=0,1,2,
$$
are bounded uniformly by a computable quantity tending to $0$ as $U\to\infty$. On $[4,U]$, certified quadrature applies using the computable derivative bounds for $\Psi_{a,b,x,y}$ and $h$. Together with the tail estimate, this computes $\Theta_{a,b,x,y}$ and its first two derivatives uniformly on $\mathcal I$. Applying \cref{lem:integral_com_nu} then computes
$$
    \int_{5/4}^{7/4}
    \Theta_{a,b,x,y}(t)\,\mathrm{d}\nu_A(t)
$$
to arbitrary accuracy. Combined with the $\nu_A$-term, this computes $\partial_x^a\partial_y^bF_A$ uniformly on $[0,X]^2$; the same estimates give the bounds below. More precisely, for every $R\in\mathbb{N}\cup\{0\}$ one can compute a constant $B_{R,X}$ such that
$
    \max_{0\le a,b\le R}
    \left\|
        \partial_x^a\partial_y^b F_A
    \right\|_{L^\infty([0,X]^2)}
    \le
    B_{R,X}.
$

For the spline approximation, define
$$
M_\delta
=
\max\left\{1,\left\lceil\frac{X}{\delta}\right\rceil\right\},
\qquad
h_\delta=\frac{X}{M_\delta},
\qquad
x_r=rh_\delta,
\quad
0\le r\le M_\delta.
$$
Then $h_\delta\le\delta$ and $h_\delta\to0$ as $\delta\downarrow0$. On each cell $ [x_r,x_{r+1}]\times[x_s,x_{s+1}], $ take the tensor-product Hermite polynomial of degree at most $2N+1$ in each variable that matches all data
$$
\partial_x^i\partial_y^jF_A(x_u,x_v),
\qquad
0\le i,j\le N,
\qquad
u\in\{r,r+1\},
\quad
v\in\{s,s+1\}.
$$
Thus data are imposed at all four vertices of every cell. The pieces match with all derivatives through order $N$, and therefore define a global $C^N$ tensor-product spline, denoted by $H_{h_\delta}F_A$.

With exact nodal data, the standard Hermite interpolation estimate \cite{birkhoff1968piecewise} gives
$$
\max_{0\le i,j\le N}
\left\|
\partial_x^i\partial_y^j
\bigl(F_A-H_{h_\delta}F_A\bigr)
\right\|_{L^\infty([0,X]^2)}
\le
P_{N,X}(h_\delta),
$$
where $P_{N,X}$ is computable from the derivative bounds through order $2N+2$, and $P_{N,X}(h)\to0$ as $h\downarrow0$.

Compute every required nodal value with error at most $\eta$. Since the interpolation operator is finite-dimensional and linear in the nodal data, one can compute a rational bound $D_{N,X,h_\delta}$ such that the induced $C^N$-error is at most $D_{N,X,h_\delta}\eta$. Choose
$$
\eta
\le
\frac{\delta}{1+D_{N,X,h_\delta}}.
$$
For the resulting computable spline $F_{A,N,\delta,X}$, we obtain
$$
\max_{0\le i,j\le N}
\left\|
\partial_x^i\partial_y^j
\bigl(F_A-F_{A,N,\delta,X}\bigr)
\right\|_{L^\infty([0,X]^2)}
\le
P_{N,X}(h_\delta)+\delta.
$$
Define
$
C_{N,X}(\delta)
=
P_{N,X}(h_\delta)+\delta.
$
Since $h_\delta\to0$, the required bound tends to zero as $\delta\downarrow0$.
\end{proof}

The fixed interpolation nodes in \cref{lem:FA} turn the Gelfand--Levitan equation into a finite-dimensional system after $F_A$ is replaced by its spline approximation. We solve
\begin{equation}
\label{eqn:second_integral}
    K_A(x,y)
    +
    F_A(x,y)
    +
    \int_0^x K_A(x,t)F_A(t,y)\,\mathrm{d}t
    =
    0,
    \qquad
    0\le y\le x,
\end{equation}
with $F_A$ replaced by a finite-rank approximation. The stability and derivative estimates below give certified approximations to $K_A$ and $q_A$.

For $x\ge 0$, introduce the operator
$$
    (\mathcal F_{A,x}f)(y)
    =
    \int_0^x F_A(y,t)f(t)\,\mathrm{d}t,
    \qquad 0\le y\le x.
$$
By symmetry of $F_A$, \eqref{eqn:second_integral} takes the operator form
$$
    (I+\mathcal F_{A,x})K_A(x,\cdot)
    =
    -F_A(x,\cdot)
$$
on $L^2(0,x)$.

\begin{lemma}
\label{lem:ifax_invert}
For every $x\ge 0$, the operator $I+\mathcal F_{A,x}$ is invertible on $L^2(0,x)$, and
$$
    \|(I+\mathcal F_{A,x})^{-1}\|_{L^2(0,x)\to L^2(0,x)}
    \le 2.
$$
\end{lemma}

\begin{proof}
The case $x=0$ is trivial, so assume $x>0$. For $f\in L^2(0,x)$, set
$$
    f_s(E)
    =
    \int_0^x s_E(t)f(t)\,\mathrm{d}t.
$$
By \cref{cor:isom}, after extending $f$ by zero outside $(0,x)$,
$$
    \|f\|_{L^2(0,x)}^2
    =
    \int_0^\infty |f_s(E)|^2\,\mathrm{d}\rho_0(E).
$$
Moreover, using
$$
    F_A(y,t)
    =
    \epsilon
    \int_0^\infty s_E(y)s_E(t)\,\mathrm{d}\sigma_A(E),
$$
we obtain
$$
    \langle \mathcal F_{A,x}f,f\rangle_{L^2(0,x)}
    =
    \int_0^x\int_0^x
    F_A(y,t)f(t)\overline{f(y)}
    \,\mathrm{d}t\,\mathrm{d}y
    =
    \epsilon
    \int_0^\infty |f_s(E)|^2\,\mathrm{d}\sigma_A(E).
$$
Therefore
$$
\begin{aligned}
    \langle (I+\mathcal F_{A,x})f,f\rangle
    &=
    \int_0^\infty |f_s(E)|^2\,\mathrm{d}\rho_A(E).
\end{aligned}
$$
Since $2\rho_A\ge \rho_0$, we have
$$
    \langle (I+\mathcal F_{A,x})f,f\rangle
    \ge
    \frac12
    \int_0^\infty |f_s(E)|^2\,\mathrm{d}\rho_0(E)
    =
    \frac12\|f\|_{L^2(0,x)}^2.
$$
Since $I+\mathcal F_{A,x}$ is bounded and self-adjoint,
$$
    \|(I+\mathcal F_{A,x})f\|_{L^2(0,x)}
    \ge
    \frac12\|f\|_{L^2(0,x)}
$$
and the result follows.
\end{proof}

The substitutions $y=xr$ and $t=xs$, with $0\le r,s\le1$, remove the $x$-dependence from the integration interval. Define $L_A(x,r)=K_A(x,xr)$. Then \eqref{eqn:second_integral} becomes
\begin{equation}
\label{eqn:transformed_integral}
    L_A(x,r)
    +
    F_A(x,xr)
    +
    x\int_0^1 L_A(x,s)F_A(xs,xr)\,\mathrm{d}s
    =
    0,
    \qquad
    x\ge 0,\quad 0\le r\le 1.
\end{equation}
For $x\ge 0$, let $\mathcal K_{A,x}$ be the integral operator on $L^2(0,1)$ with kernel $B_x(r,s)=xF_A(xs,xr)$:
$$
    (\mathcal K_{A,x}u)(r)
    =
    \int_0^1 B_x(r,s)u(s)\,\mathrm{d}s.
$$
For $x>0$, this operator is unitarily equivalent to $\mathcal F_{A,x}$ under the unitary map
$$
    U_x:L^2(0,x)\to L^2(0,1),
    \qquad
    (U_x f)(r)=x^{1/2}f(xr).
$$
Thus, for every $x\ge 0$,
$$
\|(I+\mathcal K_{A,x})^{-1}\|_{L^2(0,1)\to L^2(0,1)}
\le 2.
$$
Define $b(x,r)=-F_A(x,xr)$. Then
$$
(I+\mathcal K_{A,x})L_A(x,\cdot)=b(x,\cdot).
$$
We approximate $\mathcal K_{A,x}$ and $b(x,\cdot)$ in finite-dimensional spline spaces; the next lemma supplies the uniform stability estimate.

\begin{lemma}
\label{lem:finite_rank}
Let $T$ be a bounded operator on a Hilbert space, and suppose that $I+T$ is invertible with $\|(I+T)^{-1}\|\le M$. Let $u$ solve $(I+T)u=b$. Let $T_n$ be a finite-rank operator and let $b_n$ belong to a finite-dimensional subspace, with
$$
    \|T-T_n\|\le \eta,
    \qquad
    \|b-b_n\|\le \delta,
    \qquad
    M\eta<1.
$$
Let $u_n$ be the exact solution of $(I+T_n)u_n=b_n$. Then $I+T_n$ is invertible and
$$
    \|u-u_n\|
    \le
    \frac{M}{1-M\eta}
    \left(
        \delta+\eta M\|b\|
    \right).
$$
\end{lemma}

\begin{proof}
We factor
$$
    I+T_n
    =
    (I+T)
    \left[
        I+(I+T)^{-1}(T_n-T)
    \right].
$$
Since
$$
    \|(I+T)^{-1}(T_n-T)\|
    \le
    M\eta
    <
    1,
$$
the second factor is invertible by the Neumann series. Hence $I+T_n$ is invertible and
$$
    \|(I+T_n)^{-1}\|
    \le
    \frac{M}{1-M\eta}.
$$
Furthermore,
$$
    (I+T_n)(u-u_n)
    =
    (I+T_n)u-b_n
    =
    b+(T_n-T)u-b_n.
$$
Thus
$$
    \|u-u_n\|
    \le
    \|(I+T_n)^{-1}\|
    \left(
        \|b-b_n\|+\|T_n-T\|\,\|u\|
    \right).
$$
Since $\|u\|\le M\|b\|$, the claimed estimate follows.
\end{proof}

\begin{lemma}
\label{lem:compute_K_diag}
Given $X\in\mathbb Q_{>0}$, $N\in\mathbb N\cup\{0\}$, and $\delta\in\mathbb Q_{>0}$, we can compute, uniformly on $[0,X]$, approximations to
$$
R_A^{(j)}(x),
\qquad
R_A(x)=K_A(x,x),
\qquad
0\le j\le N+1,
$$
and to
$$
q_A^{(j)}(x),
\qquad
0\le j\le N,
$$
with certified error smaller than $\delta$.
\end{lemma}

\begin{proof}
Set
$
J=N+1.
$
We work with
$$
B(x,r,s)=xF_A(xs,xr),
\qquad
b(x,r)=-F_A(x,xr).
$$
Thus $B(x,\cdot,\cdot)$ is the kernel of $\mathcal K_{A,x}$, and
$$
(I+\mathcal K_{A,x})L_A(x,\cdot)=b(x,\cdot).
$$
By \cref{lem:FA} and the chain rule, for every $m\in\mathbb N\cup\{0\}$, the functions $B$ and $b$ are computable in $C^m$ on $[0,X]\times[0,1]^2$ and $[0,X]\times[0,1]$, respectively, with computable $C^m$-bounds.

Choose positive rational tolerances
$
\eta_{B,0},\ldots,\eta_{B,J}$ and $\eta_{b,0},\ldots,\eta_{b,J},
$
with $\eta_{B,0}<1/2$. Using \cref{lem:FA} with sufficiently high order, compute tensor-product spline approximations $B_n$ and $b_n$ such that
$$
\left\|
\partial_x^\ell(B-B_n)
\right\|_{L^\infty([0,X]\times[0,1]^2)}
\le \eta_{B,\ell},
\qquad
0\le\ell\le J,
$$
and
$$
\left\|
\partial_x^\ell(b-b_n)
\right\|_{L^\infty([0,X]\times[0,1])}
\le \eta_{b,\ell},
\qquad
0\le\ell\le J.
$$
Let $\mathcal K_{A,x,n}$ be the integral operator with kernel $B_n(x,\cdot,\cdot)$. Since $[0,1]^2$ has measure one, the Hilbert--Schmidt estimate gives
$$
\left\|
\partial_x^\ell\mathcal K_{A,x}
-
\partial_x^\ell\mathcal K_{A,x,n}
\right\|_{L^2\to L^2}
\le
\eta_{B,\ell},
\qquad
0\le\ell\le J,
$$
uniformly in $x$.

For each $x\in[0,X]$, define $L_{A,n}(x,\cdot)$ by
$$
(I+\mathcal K_{A,x,n})L_{A,n}(x,\cdot)=b_n(x,\cdot).
$$
Let $V_n\subset L^2(0,1)$ be a fixed computable finite-dimensional spline space containing $b_n(x,\cdot)$ and the range of $\mathcal K_{A,x,n}$ for every $x$. Projecting the equation onto $V_n^\perp$ shows that $L_{A,n}(x,\cdot)\in V_n$. In a fixed computable orthonormal basis of $V_n$, the equation becomes
$$
M_n(x)c_n(x)=d_n(x),
$$
where the entries of $M_n$ and $d_n$ are computable $C^J$ functions of $x$.

By \cref{lem:finite_rank}, the finite-dimensional systems are uniformly invertible on $[0,X]$, with
$$
\sup_{0\le x\le X}\|M_n(x)^{-1}\|
\le
\frac{2}{1-2\eta_{B,0}}.
$$
Moreover, $c_n$ is computable in $C^J([0,X])$. Its derivatives are obtained recursively from
$$
c_n^{(j)}
=
M_n^{-1}
\left(
d_n^{(j)}
-
\sum_{\ell=1}^j
\binom{j}{\ell}
M_n^{(\ell)}c_n^{(j-\ell)}
\right),
\qquad
1\le j\le J.
$$
The uniform inverse bound and the computable derivative bounds for $M_n$ and $d_n$ give certified uniform bounds at every stage.

For the zeroth derivative, let
$
C_{b,X}
=
\sup_{0\le x\le X}\|b(x,\cdot)\|_{L^2(0,1)}.
$
The computable bounds for $F_A$ provide a computable upper bound for $C_{b,X}$. Applying \cref{lem:finite_rank} with $M=2$ gives the estimate
$$
\sup_{0\le x\le X}
\|L_A(x,\cdot)-L_{A,n}(x,\cdot)\|_{L^2(0,1)}
\le
\frac{2}{1-2\eta_{B,0}}
\left(
\eta_{b,0}+2\eta_{B,0}C_{b,X}
\right).
$$
Also,
$$
\|L_{A,n}(x,\cdot)\|_{L^2(0,1)}
\le
\frac{2}{1-2\eta_{B,0}}
\bigl(C_{b,X}+\eta_{b,0}\bigr).
$$
Comparing the two integral equations and applying Cauchy--Schwarz yields
$$
\begin{aligned}
|L_A(x,r)-L_{A,n}(x,r)|
&\le
|b(x,r)-b_n(x,r)|\\
&\quad+
\|B(x,r,\cdot)\|_{L^2}
\|L_A(x,\cdot)-L_{A,n}(x,\cdot)\|_{L^2}\\
&\quad+
\|B(x,r,\cdot)-B_n(x,r,\cdot)\|_{L^2}
\|L_{A,n}(x,\cdot)\|_{L^2}.
\end{aligned}
$$
Hence the zeroth-order error can be made uniformly and computably small on $[0,X]\times[0,1]$.

For higher derivatives, $x$-differentiation is taken with $r$ and $s$ fixed. For $0\le j\le J$, define
$$
L_j(x,r)=\partial_x^jL_A(x,r),
\qquad
L_{j,n}(x,r)=\partial_x^jL_{A,n}(x,r),
$$
and
$$
B_j=\partial_x^jB,
\qquad
B_{j,n}=\partial_x^jB_n,
\qquad
b_j=\partial_x^jb,
\qquad
b_{j,n}=\partial_x^jb_n.
$$
For example,
$
b_j(x,r)
=
-(\partial_1+r\partial_2)^jF_A(x,xr);
$
it is not merely $-\partial_1^jF_A(x,xr)$.

Differentiating the exact and approximate equations $j$ times gives
$$
(I+\mathcal K_{A,x})L_j=G_j,
\qquad
(I+\mathcal K_{A,x,n})L_{j,n}=G_{j,n},
$$
where
$$
G_j
=
b_j
-
\sum_{\ell=1}^j
\binom{j}{\ell}
\mathcal K_{A,x}^{(\ell)}L_{j-\ell},
$$
and
$$
G_{j,n}
=
b_{j,n}
-
\sum_{\ell=1}^j
\binom{j}{\ell}
\mathcal K_{A,x,n}^{(\ell)}L_{j-\ell,n}.
$$
Here $\mathcal K_{A,x}^{(\ell)}$ and $\mathcal K_{A,x,n}^{(\ell)}$ have kernels $B_\ell$ and $B_{\ell,n}$, respectively. In particular, only $1\le\ell\le j$ occurs on the right-hand side, so all solution derivatives there have order strictly smaller than $j$.

We argue by induction on $j\le J$. The case $j=0$ was established above. Suppose that certified uniform approximations and bounds are known through order $j-1$. Set
$$
C_{B,\ell}
=
\sup_{\substack{0\le x\le X\\0\le r\le1}}
\|B_\ell(x,r,\cdot)\|_{L^2(0,1)},
\qquad
C_k
=
\sup_{0\le x\le X}\|L_k(x,\cdot)\|_{L^2(0,1)},
$$
and let
$$
e_k
=
\sup_{0\le x\le X}
\|L_k(x,\cdot)-L_{k,n}(x,\cdot)\|_{L^2(0,1)}.
$$
All the $C_{B,\ell}$ are computably bounded, and the $C_k$ are computably bounded inductively from $\|L_k\|_{L^2}\le2\|G_k\|_{L^2}$. The identity
$$
\begin{aligned}
\mathcal K_{A,x}^{(\ell)}L_{j-\ell}
-
\mathcal K_{A,x,n}^{(\ell)}L_{j-\ell,n}
&=
\bigl(
\mathcal K_{A,x}^{(\ell)}
-
\mathcal K_{A,x,n}^{(\ell)}
\bigr)L_{j-\ell}\\
&\quad+
\mathcal K_{A,x,n}^{(\ell)}
\bigl(L_{j-\ell}-L_{j-\ell,n}\bigr)
\end{aligned}
$$
therefore gives the certified bound
$$
\sup_{0\le x\le X}
\|G_j-G_{j,n}\|_{L^2(0,1)}
\le
\eta_{b,j}
+
\sum_{\ell=1}^j
\binom{j}{\ell}
\left[
\eta_{B,\ell}C_{j-\ell}
+
\bigl(C_{B,\ell}+\eta_{B,\ell}\bigr)e_{j-\ell}
\right].
$$
The same expression bounds $\sup_{x,r}|G_j(x,r)-G_{j,n}(x,r)|$, because the kernel-row $L^2$-bounds above are uniform in $r$.

Another application of \cref{lem:finite_rank} yields
$$
\|L_j-L_{j,n}\|_{L^2(0,1)}
\le
\frac{2}{1-2\eta_{B,0}}
\left(
\|G_j-G_{j,n}\|_{L^2(0,1)}
+
2\eta_{B,0}\|G_j\|_{L^2(0,1)}
\right),
$$
uniformly in $x$. The induction hypothesis gives a computable bound for $\|G_j\|_{L^2(0,1)}$, uniformly in $x$. Comparing
$$
L_j=G_j-\mathcal K_{A,x}L_j,
\qquad
L_{j,n}=G_{j,n}-\mathcal K_{A,x,n}L_{j,n},
$$
and applying the same Cauchy--Schwarz estimate converts the $L^2$-error into a certified pointwise error. Since the spline tolerances are arbitrary, the induction gives certified uniform approximations to $L_j$ on $[0,X]\times[0,1]$ for $0\le j\le J=N+1$.

Since
$
R_A(x)=K_A(x,x)=L_A(x,1),
$
we have
$$
R_A^{(j)}(x)=\partial_x^jL_A(x,1)=L_j(x,1),
\qquad
0\le j\le N+1.
$$
Choose the spline tolerances so that the pointwise errors in these quantities are smaller than $\delta/3$, and define
$
\widetilde R_j(x)=L_{j,n}(x,1).
$
Then
$$
\widetilde q_j(x)=2\widetilde R_{j+1}(x),
\qquad
0\le j\le N,
$$
satisfies
$$
\|\widetilde q_j-q_A^{(j)}\|_{L^\infty([0,X])}
<
\frac{2\delta}{3}
<
\delta.
$$
The approximations to $R_A^{(j)}$ themselves also have error smaller than $\delta$, as required.
\end{proof}

\subsection{Establishing the Dichotomy}
\label{subsec:dichot}

Define
$
    Q_A(x)=q_A(|x|).
$
The next proposition gives the growth bound needed for self-adjointness.

\begin{proposition}
\label{prop:qa_growth}
There exists a constant $C_A>0$ such that
$$
    |q_A(x)|\le C_A(1+x^2),
    \qquad x\ge 0.
$$
Consequently, $Q_A(x)\ge -C_A(1+x^2)$.
\end{proposition}

\begin{proof}
Recall the Gelfand--Levitan equation
$$
    K_A(x,y)
    +
    F_A(x,y)
    +
    \int_0^x K_A(x,t)F_A(t,y)\,\mathrm{d}t
    =
    0,
    \qquad 0\le y\le x.
$$
Equivalently,
$$
    K_A(x,\cdot)
    =
    -(I+\mathcal F_{A,x})^{-1}F_A(x,\cdot),
$$
where
$$
    (\mathcal F_{A,x}f)(y)
    =
    \int_0^x F_A(y,t)f(t)\,\mathrm{d}t,
$$
and, by \cref{lem:ifax_invert}, $\|(I+\mathcal F_{A,x})^{-1}\|_{L^2(0,x)\to L^2(0,x)}\le 2$.

The function $F_A$ and its first mixed derivatives are bounded on $[0,\infty)^2$. Indeed,
$$
    \partial_x^a\partial_y^b F_A(x,y)
    =
    \epsilon
    \int_0^\infty
    \partial_x^a s_E(x)\partial_y^b s_E(y)
    \,\mathrm{d}\sigma_A(E),
$$
and $\sigma_A$ is supported in $[5/4,\infty)$. For $E\ge 5/4$, the factors $\partial_x^a s_E(x)$ and $\partial_y^b s_E(y)$ are bounded, uniformly in $x,y$, by powers of $E$. Since $\nu_A$ is compactly supported and $g_A$ is Schwartz, all polynomial moments of $|\sigma_A|$ are finite. Hence, the required derivatives of $F_A$ are bounded.

Let $C$ be a bound for $F_A$, $\partial_x F_A$, and $\partial_y F_A$. From the diagonal form of the Gelfand--Levitan equation,
$$
    K_A(x,x)
    =
    -F_A(x,x)
    -
    \int_0^x K_A(x,t)F_A(t,x)\,\mathrm{d}t,
$$
we obtain
$$
    |K_A(x,x)|
    \le
    C
    +
    C\int_0^x |K_A(x,t)|\,\mathrm{d}t
    \le
    C
    +
    Cx^{1/2}\|K_A(x,\cdot)\|_{L^2(0,x)}.
$$
Moreover, $\|F_A(x,\cdot)\|_{L^2(0,x)}\le Cx^{1/2}$. Therefore $\|K_A(x,\cdot)\|_{L^2(0,x)}\le2Cx^{1/2}$, and so $|K_A(x,x)|\le C(1+x)$.

Next differentiate the Gelfand--Levitan equation with respect to the first variable $x$, keeping $y$ fixed:
$$
    \partial_x K_A(x,y)
    +
    \partial_x F_A(x,y)
    +
    K_A(x,x)F_A(x,y)
    +
    \int_0^x
    \partial_x K_A(x,t)F_A(t,y)\,\mathrm{d}t
    =
    0.
$$
Equivalently,
$$
    (I+\mathcal F_{A,x})\partial_x K_A(x,\cdot)
    =
    -\partial_x F_A(x,\cdot)
    -
    K_A(x,x)F_A(x,\cdot).
$$
Using $\|(I+\mathcal F_{A,x})^{-1}\|\le 2$, we get
$$
\begin{aligned}
    \|\partial_x K_A(x,\cdot)\|_{L^2(0,x)}
    &\le
    2\|\partial_x F_A(x,\cdot)\|_{L^2(0,x)}
    +
    2|K_A(x,x)|\|F_A(x,\cdot)\|_{L^2(0,x)} \\
    &\le
    Cx^{1/2}
    +
    C(1+x)x^{1/2}
    \le
    C(x^{1/2}+x^{3/2}).
\end{aligned}
$$

Recall that $R_A(x)=K_A(x,x)$. Differentiating
$$
    R_A(x)
    =
    -F_A(x,x)
    -
    \int_0^x K_A(x,t)F_A(t,x)\,\mathrm{d}t
$$
gives
$$
\begin{aligned}
    R_A'(x)
    &=
    -(\partial_x F_A+\partial_y F_A)(x,x)
    -
    K_A(x,x)F_A(x,x) \\
    &\quad
    -
    \int_0^x
    \partial_x K_A(x,t)F_A(t,x)\,\mathrm{d}t
    -
    \int_0^x
    K_A(x,t)\partial_y F_A(t,x)\,\mathrm{d}t.
\end{aligned}
$$
The first term is bounded, and the second term is $O(1+x)$. For the third term, Cauchy--Schwarz gives
$$
    \left|
    \int_0^x
    \partial_x K_A(x,t)F_A(t,x)\,\mathrm{d}t
    \right|
    \le
    \|\partial_x K_A(x,\cdot)\|_{L^2(0,x)}
    \|F_A(\cdot,x)\|_{L^2(0,x)}
    \le
    C(x^{1/2}+x^{3/2})x^{1/2}
    \le
    C(x+x^2).
$$
Similarly,
$$
    \left|
    \int_0^x
    K_A(x,t)\partial_y F_A(t,x)\,\mathrm{d}t
    \right|
    \le
    \|K_A(x,\cdot)\|_{L^2(0,x)}
    \|\partial_y F_A(\cdot,x)\|_{L^2(0,x)}
    \le
    Cx.
$$
Hence $|R_A'(x)|\le C(1+x^2)$. Since $q_A(x)=2R_A'(x)$, we obtain $|q_A(x)|\le C_A(1+x^2)$. The corresponding bound for $Q_A(x)=q_A(|x|)$ is immediate.
\end{proof}

By \cref{prop:qa_flat}, $q_A$ is smooth on $[0,\infty)$ and flat at the origin. Hence $Q_A(x)=q_A(|x|)$ is smooth on $\mathbb R$. By \cref{prop:qa_growth}, $Q_A$ satisfies the quadratic lower bound $Q_A(x)\ge -C_A(1+x^2)$. Therefore, by the Faris--Lavine theorem \cite[Theorem X.38]{reedsimon1975}, the operator $-\Delta+Q_A$ with initial domain $C_c^\infty(\mathbb R)$ is essentially self-adjoint. We denote its self-adjoint closure by $H_A$.

The essential self-adjointness of the whole-line minimal operator implies, by Weyl's limit-point/limit-circle alternative \cite[Section 9.2]{terschl}, that the differential expression $\tau_A$ is limit point at $+\infty$. Since $0$ is a regular endpoint, the conditions $f(0)=0$ and $f'(0)=0$ therefore define unique self-adjoint half-line realizations
$$
\begin{aligned}
\operatorname{Dom}(H_{A,D}^+)
&=
\left\{
f\in L^2(0,\infty):
f,f'\in\operatorname{AC}_{\mathrm{loc}}([0,\infty)),
\ \tau_Af\in L^2(0,\infty),
\ f(0)=0
\right\},\\
\operatorname{Dom}(H_{A,N}^+)
&=
\left\{
f\in L^2(0,\infty):
f,f'\in\operatorname{AC}_{\mathrm{loc}}([0,\infty)),
\ \tau_Af\in L^2(0,\infty),
\ f'(0)=0
\right\}.
\end{aligned}
$$
On these domains, $H_{A,D}^+f=\tau_Af$ and $H_{A,N}^+f=\tau_Af$, respectively.

Remling's notion of spectral measure allows additional moment-problem measures in the limit-circle case, but the spectral measure is unique in the limit-point case \cite[Section 17]{remling2002}. Consequently, $\rho_A$ is the Dirichlet Weyl--Titchmarsh measure of $H_{A,D}^+$. Let $\varphi_A(E,\cdot)$ be the solution of
$$
\tau_A\varphi_A(E,\cdot)=E\varphi_A(E,\cdot),
\qquad
\varphi_A(E,0)=0,
\qquad
\partial_x\varphi_A(E,0)=1.
$$
The generalized Fourier transform
$$
(\mathcal U_A f)(E)
=
\int_0^\infty \varphi_A(E,x)f(x)\,\mathrm dx,
\qquad
f\in C_c^\infty(0,\infty),
$$
extends to a unitary map
$$
\mathcal U_A:
L^2(0,\infty)
\longrightarrow
L^2(\mathbb R,\mathrm d\rho_A)
$$
satisfying
$$
\mathcal U_A H_{A,D}^+\mathcal U_A^{-1}=M_E;
$$
see \cite[Lemma 9.13 and Equations (9.46)--(9.47)]{terschl}. Thus $\rho_A$ is a maximal spectral measure and the Dirichlet half-line problem has spectral multiplicity one.

\begin{proposition}
\label{prop:spec_type_encode}
For every $\diamond\in\{\mathrm{ac},\mathrm{sc},\mathrm{pp}\}$,
$
\spec_\diamond(H_{A,D}^+)
=
\spec\bigl((\rho_A)_\diamond\bigr).
$
\end{proposition}

\begin{proof}
Under $\mathcal U_A$, the absolutely continuous, singular continuous, and pure point reducing subspaces are respectively
$$
L^2(\mathbb R,\mathrm d(\rho_A)_{\mathrm{ac}}),
\qquad
L^2(\mathbb R,\mathrm d(\rho_A)_{\mathrm{sc}}),
\qquad
L^2(\mathbb R,\mathrm d(\rho_A)_{\mathrm{pp}}).
$$
The spectrum of multiplication by $E$ on each of these spaces is the closed support of the corresponding measure, which is precisely the right-hand side above.
\end{proof}

Let $L_E^2(\mathbb R)$ and $L_O^2(\mathbb R)$ be the closed subspaces of $L^2(\mathbb R)$ consisting of even and odd functions, respectively. Then $L^2(\mathbb R)=L_E^2(\mathbb R)\oplus L_O^2(\mathbb R)$. The maps
$$
    (U_N f)(x)=2^{-1/2}f(|x|),\qquad
    (U_D f)(x)=2^{-1/2}\operatorname{sgn}(x)f(|x|)
$$
define unitary maps
$$
    U_N:L^2(0,\infty)\to L_E^2(\mathbb R),
    \qquad
    U_D:L^2(0,\infty)\to L_O^2(\mathbb R).
$$
Even functions satisfy the Neumann condition at the origin, while odd functions satisfy the Dirichlet condition. Since $Q_A$ is even, the whole-line operator decomposes as
$$
    H_A
    =
    \bigl(U_N H_{A,N}^+ U_N^{-1}\bigr)
    \oplus
    \bigl(U_D H_{A,D}^+ U_D^{-1}\bigr)
$$
on $L_E^2(\mathbb R)\oplus L_O^2(\mathbb R)$. In particular, $H_A$ acts as $-\Delta+Q_A$ on the whole line. Moreover, for each spectral type $\diamond\in\{\mathrm{ac},\mathrm{sc},\mathrm{pp}\}$, we have $\spec_\diamond(H_A)=\spec_\diamond(H_{A,N}^+)\cup\spec_\diamond(H_{A,D}^+)$.

The Dirichlet half-line operator carries the spectral information encoded by $\rho_A$. It remains to exclude singular continuous spectrum from the Neumann half-line operator on the test interval.

\begin{proposition}
\label{prop:dichotomy2}
Let $A$ be a computable binary matrix, and let $H_A$ be the associated whole-line Schr\"odinger operator. Then the following dichotomy holds.
\begin{enumerate}
    \item If some column of $A$ has infinitely many ones, then $\specsc(H_A)\cap D_{1/16}(3/2)\ne \emptyset$;
    \item If every column of $A$ has finitely many ones, then $\specsc(H_A)\cap \overline{D_{1/16}(3/2)}=\emptyset$.
\end{enumerate}
\end{proposition}

\begin{proof}
Let $B=D_{1/16}(3/2)$. Then
$$
    B\cap\mathbb R=(23/16,25/16)\subset [5/4,7/4]=\mathcal{I},\quad
    \overline B\cap\mathbb R=[23/16,25/16].
$$
Suppose first that some column of $A$ has infinitely many ones. By \cref{extra_needed_lemma}, the singular continuous part $(\nu_A)_{\mathrm{sc}}$ gives positive measure to every non-empty open subinterval of $\mathcal{I}$. The measures $\rho_0$ and $\kappa_A$ are absolutely continuous, so adding them does not change the singular continuous part except for the factor $\epsilon$ in front of $\nu_A$. Therefore $(\rho_A)_{\mathrm{sc}}$ gives positive measure to every non-empty open subinterval of $\mathcal{I}$. In particular, $(\rho_A)_{\mathrm{sc}}(B\cap\mathbb R)>0$. By \cref{prop:spec_type_encode}, $\specsc(H_{A,D}^+)\cap B\ne\emptyset$. Since
$$
    \specsc(H_A)
    =
    \specsc(H_{A,N}^+)\cup\specsc(H_{A,D}^+),
$$
we conclude that $\specsc(H_A)\cap B\ne\emptyset$.

Now suppose that every column of $A$ has finitely many ones. By \cref{extra_needed_lemma}, $\nu_A$ is absolutely continuous on $\mathcal{I}$. Since $\rho_0$ and $\kappa_A$ are also absolutely continuous, it follows that $\rho_A$ is absolutely continuous. Hence, by \cref{prop:spec_type_encode}, the Dirichlet half-line operator $H_{A,D}^+$ contributes no singular continuous spectrum.

For the Neumann half-line operator, let $m_D$ and $m_N$ denote the corresponding Weyl functions. With the present normalization,
$
m_N(z)=-\tfrac{1}{m_D(z)}.
$
By \cite[Equation (9.84)]{terschl}, the set
$$
\left\{
E\in\mathbb R:
\limsup_{\varepsilon\downarrow0}
\operatorname{Im}m_N(E+i\varepsilon)=\infty
\right\}
$$
is a support for the singular continuous part of the Neumann spectral measure.

Set
$
J=\left[\frac{11}{8},\frac{13}{8}\right].
$
Then
$$
\overline B\cap\mathbb R
=
\left[\frac{23}{16},\frac{25}{16}\right]
\subset J^\circ
\subset J
\subset\mathcal I.
$$
It suffices to bound $\operatorname{Im}m_N(E+i\varepsilon)$ uniformly for $E\in J$ and $0<\varepsilon<1$.

By \cite[Theorem 9.17]{terschl}, the Dirichlet Weyl function has the Herglotz representation
$$
m_D(z)
=
c_0
+
\int_{\mathbb R}
\left(
\frac{1}{\lambda-z}
-
\frac{\lambda}{1+\lambda^2}
\right)
\,\mathrm d\rho_A(\lambda)
$$
for some $c_0\in\mathbb R$. Hence
$$
\operatorname{Im}m_D(E+i\varepsilon)
=
\int_{\mathbb R}
\frac{\varepsilon}
     {(\lambda-E)^2+\varepsilon^2}
\,\mathrm d\rho_A(\lambda).
$$
Since $2\rho_A\ge\rho_0$, on $\mathcal I=[5/4,7/4]$ we have
$$
\mathrm d\rho_A(\lambda)
\ge
c_1\,\mathrm d\lambda,
\qquad
c_1:=\frac{\sqrt{5/4}}{2\pi}.
$$
For $E\in J$, both $E-5/4$ and $7/4-E$ are at least $1/8$. Therefore, for $0<\varepsilon<1$,
$$
\begin{aligned}
\operatorname{Im}m_D(E+i\varepsilon)
&\ge
c_1\int_{5/4}^{7/4}
\frac{\varepsilon}
     {(\lambda-E)^2+\varepsilon^2}
\,\mathrm d\lambda\\
&=
c_1\left[
\arctan\left(\frac{7/4-E}{\varepsilon}\right)
+
\arctan\left(\frac{E-5/4}{\varepsilon}\right)
\right]\\
&\ge
2c_1\arctan\left(\frac{1}{8\varepsilon}\right)
\ge
2c_1\arctan\left(\frac18\right)
=:c_2>0.
\end{aligned}
$$
Let $m_D(E+i\varepsilon)=u+iv$. Since $v\ge c_2$,
$$
\operatorname{Im}m_N(E+i\varepsilon)
=
\operatorname{Im}\left(-\frac1{u+iv}\right)
=
\frac{v}{u^2+v^2}
\le
\frac1v
\le
c_2^{-1}.
$$
Thus the singular continuous part of the Neumann spectral measure vanishes on $J$. Since $[23/16,25/16]\subset J^\circ$, every point of $\overline B\cap\mathbb R$ has a neighborhood of zero singular continuous measure and
$
\specsc(H_{A,N}^+)\cap\overline B=\emptyset.
$
\end{proof}

\subsection{The Lower Bound}

We reduce $\Cof$ to a column condition on computable infinite binary matrices. Recall that
$
    \Cof
    =
    \{e\in\mathbb N:\mathbb N\setminus W_e \text{ is finite}\}.
$
The required reduction is as follows.

\begin{proposition}
\label{prop:matrix_reduction}
There is a computable map $e\mapsto A^{(e)}$, where $A^{(e)}:\mathbb N^2\to\{0,1\}$ is a computable binary matrix, such that $A^{(e)}$ has at least one column with infinitely many ones if and only if $e\in\Cof$. If $e\notin\Cof$, then every column of $A^{(e)}$ has only finitely many ones.
\end{proposition}

\begin{proof}
Since $\Cof$ is a $\Sigma_3^0$ set, there is a $\Pi_2^0$ formula $\psi(e,m)$ such that
$$
    e\in\Cof
    \quad\Longleftrightarrow\quad
    (\exists m)\,\psi(e,m).
$$
Since $\Tot$ is $\Pi_2^0$-complete, there is a computable function $g$ such that
$$
    \psi(e,m)
    \quad\Longleftrightarrow\quad
    g(e,m)\in\Tot.
$$
Hence
$$
    e\in\Cof
    \quad\Longleftrightarrow\quad
    (\exists m)\bigl(g(e,m)\in\Tot\bigr).
$$
For each $m$, define the $m$th column of $A^{(e)}$ by following $\mathcal T_{g(e,m)}$ on successive inputs. Set $c^{(e)}_{1,m}=(1,1)$. For $j>1$, suppose $c^{(e)}_{j-1,m}=(k^{(e)}_{j-1,m},s^{(e)}_{j-1,m})$. Define
$$
    c^{(e)}_{j,m}
    =
    \begin{cases}
        (k^{(e)}_{j-1,m}+1,1),
        &
        \mathcal T_{g(e,m),s^{(e)}_{j-1,m}}
        (k^{(e)}_{j-1,m})\downarrow,
        \\[1mm]
        (k^{(e)}_{j-1,m},s^{(e)}_{j-1,m}+1),
        &
        \mathcal T_{g(e,m),s^{(e)}_{j-1,m}}
        (k^{(e)}_{j-1,m})\uparrow.
    \end{cases}
$$
Thus, once the computation on the current input is seen to halt, we move to the next input; otherwise, we wait one more stage on the same input. Now define
$$
    A^{(e)}(j,m)
    =
    \begin{cases}
        1,
        &
        \mathcal T_{g(e,m),s^{(e)}_{j,m}}
        (k^{(e)}_{j,m})\downarrow,
        \\[1mm]
        0,
        &
        \mathcal T_{g(e,m),s^{(e)}_{j,m}}
        (k^{(e)}_{j,m})\uparrow.
    \end{cases}
$$
This is computable uniformly in $e,j,m$, since only finite-time simulations are used in computing $c^{(e)}_{j,m}$ and then $A^{(e)}(j,m)$.

If $g(e,m)\in\Tot$, then the computation on each successive input eventually halts. Hence, the process moves through infinitely many inputs, and the $m$th column of $A^{(e)}$ contains infinitely many ones. Conversely, if $g(e,m)\notin\Tot$, then there is an input on which $\mathcal T_{g(e,m)}$ never halts. Once the construction reaches the first such input, it remains there, increasing only the stage parameter; hence every subsequent entry in the $m$th column is zero. Thus, the $m$th column contains only finitely many ones. Therefore $A^{(e)}$ has some column with infinitely many ones if and only if there exists $m$ such that $g(e,m)\in\Tot$, which is equivalent to $e\in\Cof$.
\end{proof}

The preceding construction is uniform in the input matrix $A$. More precisely, there is a single algorithm which, given an index for a computable binary matrix $A$, computes the associated objects
$$
    \nu_A,\quad \rho_A,\quad q_A,\quad Q_A(x)=q_A(|x|),
    \quad\text{and}\quad
    H_A=-\Delta+Q_A
$$
in the sense used above. In particular, given $X\in\mathbb Q_{>0}, N\in\mathbb N, p\in\mathbb N$, the algorithm outputs certified $C^N([-X,X])$-approximations to $Q_A$, with error at most $2^{-p}$ in all derivatives up to order $N$. Equivalently, the map $A\mapsto Q_A$ is computable, uniformly in $A$, as a map from computable binary matrices to smooth potential codes in $(\mathcal O_M^\infty,\mathsf{Code}_M^\infty)$.

\begin{theorem}
\label{thm:sc}
We have $(\mathcal O_M^\infty,\mathsf{Code}_M^\infty,\specsc)\notin\Delta_3^M$.
\end{theorem}

\begin{proof}
Suppose, for contradiction, that $(\mathcal O_M^\infty,\mathsf{Code}_M^\infty,\specsc)\in\Delta_3^M$. By the uniformity of the construction above, for every computable infinite binary matrix $A$ we can compute a smooth potential code for $Q_A(x)=q_A(|x|)$, and hence for the associated self-adjoint whole-line Schr\"odinger operator $H_A=-\Delta+Q_A$. Moreover, by \cref{prop:dichotomy2}, $\specsc(H_A)\cap D_{1/16}(3/2)\ne\emptyset$ if $A$ has a column with infinitely many ones, while $\specsc(H_A)\cap \overline{D_{1/16}(3/2)}=\emptyset$ if every column of $A$ has only finitely many ones.

Now apply \cref{prop:matrix_reduction}. Given $e$, compute the matrix $A^{(e)}$. Then $A^{(e)}$ has a column with infinitely many ones if and only if $e\in\Cof$. Set $V_e=Q_{A^{(e)}}$ and $H_e=H_{A^{(e)}}$. The map $e\mapsto V_e$ is computable as a map into $(\mathcal O_M^\infty,\mathsf{Code}_M^\infty)$, again by the uniformity of the construction. Therefore, if $e\in\Cof$ then $\specsc(H_e)\cap D_{1/16}(3/2)\ne\emptyset$, whereas if $e\notin\Cof$ then $\specsc(H_e)\cap \overline{D_{1/16}(3/2)}=\emptyset$. Thus the assumed $\Delta_3^M$-computability of $\specsc$, together with \cref{prop:delta3bound}, would imply that $\Cof\in\Delta_3^0$. This contradicts \cref{prop:complete_class}, since $\Cof$ is $\Sigma_3^0$-complete and so not $\Delta_3^0$.
\end{proof}

\section{Upper Bounds}
\label{sec:upper}

To prove \cref{thm:main_pos}, we combine the resolvent algorithms of \cite{colbrook2019computing} with a compactly supported, sufficiently smooth wavelet basis. Its finite linear span is a core for every operator in $\mathcal D_{N,d}^{\mathrm{BV}}$, and the required matrix elements are computable from $\mathsf{Eval}_{N,d}^{\mathrm{BV}}$. The same construction gives the Markov version through queries to the relevant source codes.

\subsection{Computing the Resolvent}

We use the following form of the positive result from \cite{colbrook2019computing}.

\begin{proposition}
\label{prop:resolv_comput}
Let $\Omega$ be a class of self-adjoint operators on a separable Hilbert space $\mathcal H$. Let $\Lambda$ be an evaluation set for $\Omega$, and let $\mathcal S\subseteq \mathcal H$ be a countable set with dense linear span. Suppose that there are arithmetic algorithms which, given $T\in\Omega$, $v\in\mathcal S$, $z\in\mathbb C\setminus\mathbb R$, and $\eta>0$, produce approximations to
$$
    (T-zI)^{-1}v
    \quad\text{and}\quad
    \innerprod{(T-zI)^{-1}v}{v}
$$
with error at most $C(v,T)\eta$, where $C(v,T)<\infty$ is independent of $z$ and $\eta$. Then
$$
    (\Omega,\Lambda,\specpp)\in\Sigma_2^A,
    \qquad
    (\Omega,\Lambda,\specac)\in\Sigma_2^A,
    \qquad
    (\Omega,\Lambda,\specsc)\in\Sigma_3^A.
$$
\end{proposition}

The scalar Borel transforms $\innerprod{(T-zI)^{-1}v}{v}$ determine the scalar spectral measures of a dense set of vectors. The algorithms of \cite{colbrook2019computing} use boundary behavior of these Borel transforms to recover the pure point, absolutely continuous, and singular continuous supports with the stated SCI heights.

We apply \cref{prop:resolv_comput} to $\mathcal D_{N,d}^{\mathrm{BV}}$. The following certified least-squares argument computes resolvent vectors from matrix elements.

\begin{lemma}
\label{lem:graph_matrix_elt_comp}
Let $S$ be a closed, densely defined, invertible operator on $\mathcal H$, with bounded inverse $S^{-1}$. Let $\{e_j\}_{j\in\mathbb N}\subset \operatorname{Dom}(S)$ be a sequence whose finite linear span is a core for $S$. Assume that a computable number $M$ is known with $M\ge \|S^{-1}\|$. Suppose that there is an arithmetic algorithm which, given $i,j$, computes
$$
    \innerprod{S e_i}{e_j},
    \qquad
    \innerprod{S e_i}{S e_j},
    \qquad
    \innerprod{e_i}{e_j}
$$
to arbitrary accuracy. Then, uniformly in $i$ and $\rho>0$, one can compute a finitely supported vector $u$ such that $\|u-S^{-1}e_i\|\le M\rho$. Consequently, $\innerprod{S^{-1}e_i}{e_i}$ can be computed to arbitrary accuracy.
\end{lemma}

\begin{proof}
Since the finite linear span of the vectors $e_j$ is a core for $S$, it is dense in $\operatorname{Dom}(S)$ with respect to the graph norm $\|u\|_S^2:=\|u\|^2+\|Su\|^2$. In particular, for every $\rho>0$, there exists a finitely supported vector $u$ such that $\|Su-e_i\|<\rho$. Such a vector can be found by enumerating all finitely supported vectors with rational complex coefficients,
$$
    u=\sum_{j=1}^m c_j e_j,
    \qquad c_j\in\mathbb Q+i\mathbb Q.
$$
For each such $u$, the quantity $\|Su-e_i\|^2$ is computable from the assumed matrix elements. Indeed, expanding the square expresses it as a finite combination of terms of the form
$$
    \innerprod{S e_j}{S e_k},
    \qquad
    \innerprod{S e_j}{e_i},
    \qquad
    \innerprod{e_i}{e_i}.
$$
Denote this enumeration by
$
u_1,u_2,\ldots,
$
and let
$
a_m=\|Su_m-e_i\|^2.
$
For every $m$ and $n$, the assumed matrix-element algorithms allow us to compute a rational interval
$
[L_{m,n},U_{m,n}]
$
containing $a_m$, with
$
U_{m,n}-L_{m,n}\le2^{-n}.
$
The search must be dovetailed. At stage $n$, compute such intervals for every $m\le n$, and halt if
$
U_{m,n}<\rho^2
$
for one of these candidates.

This procedure terminates. Indeed, the graph-core property first gives a finitely supported vector $v$ with $\|Sv-e_i\|<\rho/4$. Rational coefficient vectors are dense in the finite-dimensional span containing $v$, also after applying $S$, so the enumeration contains some $u_m$ satisfying
$
\|Su_m-e_i\|<\rho/2.
$
Thus $a_m<\rho^2/4$. Once $n\ge m$ is sufficiently large, the certified interval for this fixed candidate has upper endpoint smaller than $\rho^2$, and the algorithm halts with a vector $u_m$ for which
$
\|Su_m-e_i\|<\rho.
$
Denote the returned vector by $u$. Therefore
$$
    \|u-S^{-1}e_i\|
    =
    \|S^{-1}(Su-e_i)\|
    \le
    \|S^{-1}\|\,\|Su-e_i\|
    \le
    M\rho.
$$
Since $u$ is finitely supported, the quantity $\innerprod{u}{e_i}$ is computable from the matrix elements $\innerprod{e_j}{e_i}$. Moreover,
$$
    \left|
        \innerprod{S^{-1}e_i}{e_i}
        -
        \innerprod{u}{e_i}
    \right|
    \le
    \|S^{-1}e_i-u\|\,\|e_i\|
    \le
    M\rho\,\|e_i\|.
$$
The number $\|e_i\|$ is computable from $\innerprod{e_i}{e_i}$. Hence, by choosing $\rho$ sufficiently small, we compute $\innerprod{S^{-1}e_i}{e_i}$ to arbitrary accuracy.
\end{proof}

For a self-adjoint operator $T$ and $z\in\mathbb C\setminus\mathbb R$, we apply \cref{lem:graph_matrix_elt_comp} to $S=T-zI$. Since $T$ is self-adjoint, $\spec(T)\subset\mathbb R$, and therefore $\|(T-zI)^{-1}\|\le {1}/{|\operatorname{Im} z|}$. Thus we may take $M=|\operatorname{Im} z|^{-1}$. Moreover, if the finite linear span of $\{e_j\}$ is a core for $T$, then it is also a core for $T-zI$, since the graph norms of $T$ and $T-zI$ are equivalent. It therefore suffices, in an appropriate basis, to compute
$$
    \innerprod{(T-zI)e_i}{e_j},
    \qquad
    \innerprod{(T-zI)e_i}{(T-zI)e_j},
    \qquad
    \innerprod{e_i}{e_j}
$$
to arbitrary accuracy. Equivalently, it is enough to compute the corresponding matrix elements involving $T e_i$, $T e_j$, and the basis vectors themselves.

\subsection{Constructing the Basis}

The wavelet construction in \cite{hernandezweiss1996} gives the following basis.

\begin{proposition}
\label{prop:basis_exist}
For every integer $\gamma\ge 0$, there exist real-valued compactly supported functions $\phi^{(\gamma)},\psi^{(\gamma)}\in C_c^\gamma(\mathbb R)$ such that the following hold.
\begin{enumerate}
    \item The function $\phi^{(\gamma)}$ is a scaling function and $\psi^{(\gamma)}$ is the associated mother wavelet. Moreover, $\psi^{(\gamma)}$ may be chosen to have sufficiently many vanishing moments.
    \item If
    $$
        \psi_{j,k}^{(\gamma)}(x)
        =
        2^{j/2}\psi^{(\gamma)}(2^j x-k),
        \qquad
        \phi_k^{(\gamma)}(x)
        =
        \phi^{(\gamma)}(x-k),
    $$
    then
    $$
        \mathcal B^{(\gamma)}
        =
        \{\psi_{j,k}^{(\gamma)}:j\ge 0,\ k\in\mathbb Z\}
        \cup
        \{\phi_k^{(\gamma)}:k\in\mathbb Z\}
    $$
    is an orthonormal basis of $L^2(\mathbb R)$.
    \item The tensor-product family
    $$
        \mathcal B_d^{(\gamma)}
        =
        \left\{
            \bigotimes_{\ell=1}^d f_\ell:
            f_\ell\in\mathcal B^{(\gamma)}
        \right\}
    $$
    is an orthonormal basis of $L^2(\mathbb R^d)$.
\end{enumerate}
\end{proposition}

We choose the wavelets from a standard computable compactly supported construction, for instance, a sufficiently regular Daubechies family, so that the refinement coefficients, supports, point values, and the variation bounds for all derivatives used below are computable uniformly in the scale and translation indices. We will use the following standard characterization of Sobolev norms in terms of wavelet coefficients.

\begin{proposition}[Consequence of {\cite[Theorem 6.27]{hernandezweiss1996}}]
\label{prop:sobolev_est}
Let $\phi,\psi$ be a compactly supported wavelet pair as above, with regularity $\alpha$ and sufficiently many vanishing moments. Then, for every $0\le s<\alpha$, there exist constants $c,C>0$ such that, for every $f\in H^s(\mathbb R)$,
$$
    c\|f\|_{H^s(\mathbb R)}^2
    \le
    \sum_{k\in\mathbb Z}
    |\innerprod{f}{\phi_k}|^2
    +
    \sum_{j=0}^\infty\sum_{k\in\mathbb Z}
    2^{2js}
    |\innerprod{f}{\psi_{j,k}}|^2
    \le
    C\|f\|_{H^s(\mathbb R)}^2.
$$
\end{proposition}

\begin{proposition}
\label{prop:core}
Let $T\in\mathcal D_{N,d}^{\mathrm{BV}}$, and let $\gamma\ge N+1$. Then $V_d^{(\gamma)}:=\operatorname{span}\bigl(\mathcal B_d^{(\gamma)}\bigr)$ is a core for $T$.
\end{proposition}

\begin{proof}[Proof for $d=1$]
For $d>1$, the same argument applies with multi-indices and the corresponding tensor-product wavelet characterization of Sobolev spaces.

Every compactly supported $C^N$ function belongs to $\operatorname{Dom}(T)$. Let $f\in C_c^N(\mathbb R)$, and let $f_\varepsilon\in C_c^\infty(\mathbb R)$ be a standard mollification of $f$. Then $f_\varepsilon\to f$ in $L^2(\mathbb R)$ as $\varepsilon\to0^+$. Moreover, for every $0\le \alpha\le N$, $\partial^\alpha f_\varepsilon\to    \partial^\alpha f$ in $L^2(\mathbb R)$. For all sufficiently small $\varepsilon$, the functions $f_\varepsilon$ are supported in a fixed compact set $K$. Since the coefficients $a_\alpha$ are locally bounded, each $a_\alpha$ is bounded on $K$. Hence $a_\alpha \partial^\alpha f_\varepsilon\to    a_\alpha \partial^\alpha f$ in $L^2(\mathbb R)$ for every $0\le \alpha\le N$. Therefore
$$
    T f_\varepsilon
    =
    \sum_{\alpha\le N} a_\alpha \partial^\alpha f_\varepsilon
    \to
    \sum_{\alpha\le N} a_\alpha \partial^\alpha f
    \quad\text{in } L^2(\mathbb R).
$$
Since $T$ is closed, it follows that $f\in\operatorname{Dom}(T)$, with $Tf=\sum_{\alpha\le N} a_\alpha \partial^\alpha f$. In particular, since the wavelets in $\mathcal B^{(\gamma)}$ are compactly supported and belong to $C^\gamma$, and since $\gamma\ge N+1$, we have $V^{(\gamma)}:=\operatorname{span}\bigl(\mathcal B^{(\gamma)}\bigr)\subseteq \operatorname{Dom}(T)$.

Since $C_c^\infty(\mathbb R)$ is a core for every operator in $\mathcal D_{N,1}^{\mathrm{BV}}$, it suffices to show that, for every $g\in C_c^\infty(\mathbb R)$, there exists a sequence $g_n\in V^{(\gamma)}$ such that
$$
    g_n\to g
    \quad\text{and}\quad
    Tg_n\to Tg
    \quad\text{in } L^2(\mathbb R).
$$
For brevity, let $\phi=\phi^{(\gamma)}$ and $\psi=\psi^{(\gamma)}$. Fix $g\in C_c^\infty(\mathbb R)$ with $\operatorname{supp}g\subset[-L_g,L_g]$. Its wavelet expansion in $L^2(\mathbb R)$ is
$$
    g
    =
    \sum_{k\in\mathbb Z}
    \innerprod{g}{\phi_k}\phi_k
    +
    \sum_{j=0}^\infty\sum_{k\in\mathbb Z}
    \innerprod{g}{\psi_{j,k}}\psi_{j,k}.
$$
Since $\phi$ is compactly supported and $g$ is compactly supported, only finitely many coefficients $\innerprod{g}{\phi_k}$ are non-zero. Let $J\subset\mathbb Z$ denote this finite set. Similarly, for each fixed $j$, only finitely many coefficients $\innerprod{g}{\psi_{j,k}}$ are non-zero. Let $I_j\subset\mathbb Z$ denote this finite set. Define
$$
    g_n
    =
    \sum_{k\in J}
    \innerprod{g}{\phi_k}\phi_k
    +
    \sum_{j=0}^n\sum_{k\in I_j}
    \innerprod{g}{\psi_{j,k}}\psi_{j,k}.
$$
Then $g_n\in V^{(\gamma)}$. We claim that $g_n\to g$ in $H^N(\mathbb R)$. Choose $s$ such that $N<s<\gamma$. This is possible because $\gamma\ge N+1$. Since $g\in C_c^\infty(\mathbb R)$, we have $g\in H^s(\mathbb R)$. By \cref{prop:sobolev_est},
$$
    \sum_{j=0}^\infty\sum_{k\in\mathbb Z}
    2^{2js}
    |\innerprod{g}{\psi_{j,k}}|^2
    <\infty.
$$
Applying the same wavelet characterization with exponent $N$, we get
$$
\begin{aligned}
    \|g-g_n\|_{H^N(\mathbb R)}^2
    &\lesssim
    \sum_{j>n}\sum_{k\in\mathbb Z}
    2^{2jN}
    |\innerprod{g}{\psi_{j,k}}|^2.
\end{aligned}
$$
Since $N<s$, the last tail tends to $0$ as $n\to\infty$. Hence $g_n\to g$ in $H^N(\mathbb R)$. In particular,
$$
    \partial^\alpha g_n\to \partial^\alpha g
    \quad\text{in } L^2(\mathbb R),
    \qquad
    0\le \alpha\le N.
$$
The convergence also passes through $T$. Since $\phi$ and $\psi$ are compactly supported, and since the only wavelets appearing in the expansion of $g_n$ have supports within a fixed bounded neighborhood of $\operatorname{supp}g$, there is a fixed compact set $K_g\subset\mathbb R$ such that $\operatorname{supp}g_n\subset K_g$ for all $n$. The coefficients $a_\alpha$ are bounded on $K_g$. Therefore, for every $0\le\alpha\le N$, $a_\alpha\partial^\alpha g_n\to a_\alpha\partial^\alpha g$ in $L^2(\mathbb R)$. Consequently,
$$
    Tg_n
    =
    \sum_{\alpha\le N}a_\alpha\partial^\alpha g_n
    \to
    \sum_{\alpha\le N}a_\alpha\partial^\alpha g
    =
    Tg
    \quad\text{in } L^2(\mathbb R).
$$
Thus $g_n\to g$ in the graph norm of $T$. Since $C_c^\infty(\mathbb R)$ is a core for $T$, it follows that $V^{(\gamma)}$ is a core for $T$.
\end{proof}

\subsection{Computation of Matrix Elements}

The wavelet matrix elements are computable from $\mathsf{Eval}_{N,d}^{\mathrm{BV}}$. For $r>0$, let $Q_r=[-r,r]^d$, and let $\mathcal A_r$ be the algebra of pointwise-defined bounded Borel functions $f:Q_r\to\mathbb C$ with finite Hardy--Krause variation anchored at $(r,\ldots,r)$. If $f\in C^d(Q_r;\mathbb C)$, this variation is
$$
    \mathsf{TV}_{Q_r}(f)
    =
    \sum_{\emptyset\ne J\subseteq\{1,\ldots,d\}}
    \int_{[-r,r]^J}
    \left(
        \left|\partial_J\operatorname{Re}f(\widetilde x^{\,J})\right|
        +
        \left|\partial_J\operatorname{Im}f(\widetilde x^{\,J})\right|
    \right)
    \,\mathrm dx_J,
$$
where $\partial_J=\prod_{j\in J}\partial_j$ and
$$
    \widetilde x^{\,J}_j
    =
    \begin{cases}
        x_j, & j\in J,\\
        r, & j\notin J.
    \end{cases}
$$
For a non-smooth real-valued function, $\mathsf{TV}_{Q_r}$ is the sum of the Vitali variations of its restrictions to the non-empty upper faces of $Q_r$; for a complex-valued function, we sum the corresponding variations of its real and imaginary parts. This is a pointwise notion and hence refers to the designated representative. Following \cite{blumlinger1989topological}, define
$$
    \|f\|_r
    =
    \|f\|_{\infty,Q_r}
    +(3^d+1)\mathsf{TV}_{Q_r}(f),
    \qquad
    \|f\|_{\infty,Q_r}:=\sup_{x\in Q_r}|f(x)|.
$$
For $f\in C^d(Q_r;\mathbb C)$, the preceding formula gives
$$
    \mathsf{TV}_{Q_r}(f)
    \le
    \sum_{\emptyset\ne J\subseteq\{1,\ldots,d\}}
    (2r)^{|J|}
    \bigl(
        \|\partial_J\operatorname{Re}f\|_{\infty,Q_r}
        +
        \|\partial_J\operatorname{Im}f\|_{\infty,Q_r}
    \bigr).
$$
In particular,
$$
    \|f\|_r
    \le
    \|f\|_{\infty,Q_r}
    +(3^d+1)
    \sum_{\emptyset\ne J\subseteq\{1,\ldots,d\}}
    (2r)^{|J|}
    \bigl(
        \|\partial_J\operatorname{Re}f\|_{\infty,Q_r}
        +
        \|\partial_J\operatorname{Im}f\|_{\infty,Q_r}
    \bigr).
$$
It is shown in \cite{blumlinger1989topological} that $(\mathcal A_r,\|\cdot\|_r)$ is a Banach algebra. In particular, products and complex conjugates remain in $\mathcal A_r$, and $\|fg\|_r\le\|f\|_r\|g\|_r$.

Hardy--Krause variation is chosen because it gives a deterministic quadrature error bound from point evaluations via Koksma--Hlawka. We use the Koksma--Hlawka inequality in the following form.

\begin{proposition}
\label{prop:koksma}
Let $F$ have bounded Hardy--Krause variation on $[-r,r]^d$. Let $t_1,\ldots,t_m\in[0,1]^d$, and set
$$
    s_q=2r t_q-(r,\ldots,r)\in[-r,r]^d.
$$
Then
$$
    \left|
        \frac{(2r)^d}{m}
        \sum_{q=1}^m F(s_q)
        -
        \int_{[-r,r]^d}F(x)\,\mathrm{d}x
    \right|
    \le
    (2r)^d
    \mathsf{TV}_{[-r,r]^d}(F)
    D_m^\ast(\{t_1,\ldots,t_m\}),
$$
where $D_m^\ast$ denotes the star discrepancy.
\end{proposition}

Thus, to approximate $\int_{[-r,r]^d}F(x)\,\mathrm{d}x$ from point evaluations, it suffices to know a bound on $\mathsf{TV}_{[-r,r]^d}(F)$ and to use sample points with known star discrepancy. We fix the $d$-dimensional Halton sequence $\{t_q\}_{q\ge 1}$. This computable sequence is described in \cite{niederreiter1992random}. Its points are rational and computable from $q$, and there is a computable constant $C_d$ such that
$$
    D_m^\ast(\{t_1,\ldots,t_m\})
    \le
    C_d\frac{(\log m)^d}{m},
    \qquad m\ge 2;
$$
see \cite[Theorem 3.6]{niederreiter1992random}. Fix $\gamma=N+1$, and let $\mathcal B_d=\mathcal B_d^{(\gamma)}=\{e_j\}_{j\in\mathbb N}$. Choose the wavelet basis once and for all so that the supports, point values, and the required derivative and variation bounds of its elements are computable uniformly in $j$.

\begin{lemma}
\label{lem:matrix_elt_comp}
Let $T=\sum_{|\alpha|\le N}a_\alpha \partial^\alpha$ be an operator in $\mathcal D_{N,d}^{\mathrm{BV}}$, and let $\mathcal B_d=\{e_j\}_{j\in\mathbb N}$ be the basis fixed above. From $\mathsf{Eval}_{N,d}^{\mathrm{BV}}$, one can compute, uniformly in $i,j$, the matrix elements
$$
    \innerprod{T e_i}{e_j}
    \quad\text{and}\quad
    \innerprod{T e_i}{T e_j}
$$
to arbitrary accuracy. The same assertion holds in the Markov formulation, with the point evaluations and variation bounds supplied by a source code.
\end{lemma}

\begin{proof}
By \cref{prop:core}, each $e_i$ belongs to $\operatorname{Dom}(T)$. Hence $T e_i=\sum_{|\alpha|\le N}a_\alpha \partial^\alpha e_i$. Therefore
\begin{equation}
\label{eqn:6_first}
    \innerprod{T e_i}{e_j}
    =
    \sum_{|\alpha|\le N}
    \int_{\mathbb R^d}
    a_\alpha(x)\partial^\alpha e_i(x)\overline{e_j(x)}
    \,\mathrm{d}x,
\end{equation}
and
\begin{equation}
\label{eqn:6_second}
    \innerprod{T e_i}{T e_j}
    =
    \sum_{|\alpha|,|\beta|\le N}
    \int_{\mathbb R^d}
    a_\alpha(x)\overline{a_\beta(x)}
    \partial^\alpha e_i(x)
    \overline{\partial^\beta e_j(x)}
    \,\mathrm{d}x.
\end{equation}
It suffices to treat the terms in \eqref{eqn:6_second}; those in \eqref{eqn:6_first} are analogous. Fix $\alpha,\beta$ with $|\alpha|,|\beta|\le N$, and define
$$
    F_{\alpha,\beta,i,j}(x)
    =
    a_\alpha(x)\overline{a_\beta(x)}
    \partial^\alpha e_i(x)
    \overline{\partial^\beta e_j(x)}.
$$
The functions $e_i$ and $e_j$ are compactly supported, and their supports are computable from $i,j$. Thus we can compute a rational $r>0$ such that $\operatorname{supp} e_i\cup \operatorname{supp} e_j\subset [-r,r]^d$. Then
$$
    \int_{\mathbb R^d}F_{\alpha,\beta,i,j}(x)\,\mathrm{d}x
    =
    \int_{[-r,r]^d}F_{\alpha,\beta,i,j}(x)\,\mathrm{d}x.
$$
The function $F_{\alpha,\beta,i,j}$ belongs to $\mathcal A_r$. Indeed, $a_\alpha,a_\beta\in\mathcal A_r$ by the definition of $\mathsf{Eval}_{N,d}^{\mathrm{BV}}$, while $\partial^\alpha e_i$ and $\partial^\beta e_j$ belong to $\mathcal A_r$. Since the basis functions are tensor products of one-dimensional compactly supported wavelets, every derivative $\partial^\alpha e_i$ with $|\alpha|\le N$ is a finite tensor product of compactly supported one-dimensional $C^1$ functions. Such tensor products have bounded Hardy--Krause variation, with computable $\mathcal A_r$-bounds. By submultiplicativity of the $\mathcal A_r$-norm,
$$
\begin{aligned}
    \|F_{\alpha,\beta,i,j}\|_r
    &\le
    \|a_\alpha\|_r
    \|a_\beta\|_r
    \|\partial^\alpha e_i\|_r
    \|\partial^\beta e_j\|_r.
\end{aligned}
$$
The first two factors are bounded by $\mathsf{Eval}_{N,d}^{\mathrm{BV}}$, and the last two factors are computable from the fixed wavelet basis. Hence we can compute a bound $V_{\alpha,\beta,i,j,r}$ such that
$$
    \mathsf{TV}_{[-r,r]^d}(F_{\alpha,\beta,i,j})
    \le
    V_{\alpha,\beta,i,j,r}.
$$
Given a target error $\varepsilon>0$, choose $m$ so large that
$$
    (2r)^d
    V_{\alpha,\beta,i,j,r}
    C_d\frac{(\log m)^d}{m}
    <
    \frac{\varepsilon}{2}.
$$
Using the first $m$ Halton points, form
$$
    \frac{(2r)^d}{m}
    \sum_{q=1}^m F_{\alpha,\beta,i,j}(s_q).
$$
At each sample point $s_q$, the values of $a_\alpha(s_q)$ and $a_\beta(s_q)$ are computable from $\mathsf{Eval}_{N,d}^{\mathrm{BV}}$, while the wavelet factors are computable directly. Compute each value $F_{\alpha,\beta,i,j}(s_q)$ to sufficient accuracy so that the accumulated error in the finite average is smaller than $\varepsilon/2$. Then \cref{prop:koksma} shows that the resulting number approximates $\int_{[-r,r]^d}F_{\alpha,\beta,i,j}(x)\,\mathrm{d}x$ with error at most $\varepsilon$. Summing over the finitely many pairs $\alpha,\beta$ in \eqref{eqn:6_second} computes $\innerprod{T e_i}{T e_j}$ to arbitrary accuracy. The same argument applied to \eqref{eqn:6_first}, with $G_{\alpha,i,j}(x)=a_\alpha(x)\partial^\alpha e_i(x)\overline{e_j(x)}$, computes $\innerprod{T e_i}{e_j}$. Thus both matrix elements are computable uniformly in $i,j$.
\end{proof}

\begin{proof}[Proof of \cref{thm:main_pos}]
Let $\mathcal S=\mathcal B_d=\{e_j\}_{j\in\mathbb N}$. By \cref{prop:core}, the linear span of $\mathcal S$ is a core for each $T\in\mathcal D_{N,d}^{\mathrm{BV}}$. By \cref{lem:matrix_elt_comp}, the matrix elements $\innerprod{T e_i}{e_j}$ and $\innerprod{T e_i}{T e_j}$ are computable from $\mathsf{Eval}_{N,d}^{\mathrm{BV}}$, uniformly in $i,j$. Fix $z\in\mathbb C\setminus\mathbb R$, and set $S=T-zI$. Since $T$ is self-adjoint, $\spec(T)\subset\mathbb R$. Hence $S$ is invertible and $\|S^{-1}\|=\|(T-zI)^{-1}\|\le|\operatorname{Im}z|^{-1}$. The finite linear span of $\mathcal S$ is also a core for $S$, since the graph norms of $T$ and $T-zI$ are equivalent.

The matrix elements required by \cref{lem:graph_matrix_elt_comp} are computable. Indeed, with our convention that the inner product is linear in the first argument, we have
$$
\begin{aligned}
    \innerprod{S e_i}{e_j}
    &=
    \innerprod{T e_i}{e_j}
    -
    z\innerprod{e_i}{e_j}, \\
    \innerprod{S e_i}{S e_j}
    &=
    \innerprod{T e_i}{T e_j}
    -
    \overline z\,\innerprod{T e_i}{e_j}
    -
    z\,\innerprod{e_i}{T e_j}
    +
    |z|^2\innerprod{e_i}{e_j}.
\end{aligned}
$$
Since $\{e_i\}_{i\in\mathbb N}$ is an orthonormal basis, $\innerprod{e_i}{e_j}=\delta_{ij}$. Moreover, $\innerprod{e_i}{T e_j}=\overline{\innerprod{T e_j}{e_i}}$, which is computable by \cref{lem:matrix_elt_comp}. Therefore the hypotheses of \cref{lem:graph_matrix_elt_comp} are satisfied with $M=|\operatorname{Im}z|^{-1}$. Given a required accuracy $\eta>0$, take $\rho=\eta|\operatorname{Im}z|$. Then \cref{lem:graph_matrix_elt_comp} gives an approximation to $(T-zI)^{-1}e_\ell$ with error at most $\eta$. Taking the inner product with $e_\ell$, and using $\|e_\ell\|=1$, gives an approximation to the resolvent matrix element $\innerprod{(T-zI)^{-1}e_\ell}{e_\ell}$ with error at most $\eta$.

Thus the hypotheses of \cref{prop:resolv_comput} are satisfied for $(\mathcal D_{N,d}^{\mathrm{BV}},\mathsf{Eval}_{N,d}^{\mathrm{BV}})$. Hence, $(\mathcal D_{N,d}^{\mathrm{BV}},\mathsf{Eval}_{N,d}^{\mathrm{BV}},\specpp)\in \Sigma_2^A$, $(\mathcal D_{N,d}^{\mathrm{BV}},\mathsf{Eval}_{N,d}^{\mathrm{BV}},\specac)\in \Sigma_2^A$, and $(\mathcal D_{N,d}^{\mathrm{BV}},\mathsf{Eval}_{N,d}^{\mathrm{BV}},\specsc)\in \Sigma_3^A$. If the same data are supplied by source code, Turing machines can simulate the same algorithms, since every step uses only finite rational operations together with queries to the source code for point evaluations and variation bounds. This gives the Markov version, with $\Sigma_2^A$ replaced by $\Sigma_2^M$, and $\Sigma_3^A$ replaced by $\Sigma_3^M$.
\end{proof}

\section{Consequences for Certification}
\label{sec:prov}

The Markov lower bounds also obstruct formal certification. By a theory we mean either a first-order axiom system in a fixed language containing the language of arithmetic, or a first-order theory, such as $\mathsf{ZFC}$, equipped with a fixed interpretation of arithmetic. Denote its axiom set by $\operatorname{Ax}(\mathfrak T)$.

We call $\mathfrak T$ \emph{effectively axiomatized} if there is an algorithm which, on input $n$, outputs the code of an axiom $\phi_n$ of $\mathfrak T$, and if the sequence $\{\phi_n\}_{n\in\mathbb N}$ enumerates exactly the axioms: $\operatorname{Ax}(\mathfrak T)=\{\phi_n:n\in\mathbb N\}$. We say that $\mathfrak T$ is \emph{arithmetically sound} if, under the fixed interpretation of arithmetic, it proves the axioms of Peano arithmetic and every arithmetical sentence provable from $\mathfrak T$ is true in the standard model of $\mathbb N$.

Suppose that the operators $A_e$ are effectively described, uniformly in $e\in\mathbb N$, and that a fixed tower of algorithms, possibly of height greater than one, computes $\Xi(A_e)$. The closed set $\Xi(A_e)$ is determined by the countable family of inclusion statements
$$
\Xi(A_e)\cap U\neq\emptyset,
\qquad
\Xi(A_e)\cap \overline U=\emptyset,
$$
where $U$ ranges over balls with rational radius and center in $\mathbb Q+i\mathbb Q$. Because $A_e$ and the tower are effectively specified, each inclusion statement has a uniform arithmetical formulation. This is the arithmetization used below. For fixed $e$ and $U$, one may therefore ask whether the corresponding sentence is provable, refutable, or independent over a theory such as $\mathsf{ZFC}$.

\subsection{Establishing Non-Provability from Non-Computability}

The next proposition converts computability lower bounds into non-provability.

\begin{proposition}
\label{prop:non_prov}
Let $S\subseteq \mathbb N$. Suppose that there is a computable map $e\mapsto s(e)$ into source codes for a Markov spectral problem $(\Omega,\Lambda,\Xi)$, and let $A_e$ be the operator described by $s(e)$. Suppose that, for some fixed open ball $U\subseteq\mathbb C$ with rational center and rational radius, we have
$$
    e\in S
    \implies
    \Xi(A_e)\cap U\ne\emptyset,
\qquad
    e\notin S
    \implies
    \Xi(A_e)\cap\overline U=\emptyset.
$$
Let $\mathfrak T$ be an arithmetically sound, effectively axiomatized theory, and assume that the assertions above are represented, uniformly in $e$, by arithmetical sentences about the source code $s(e)$.

If $S$ is not computably enumerable, then there exists $e \in S$ such that the true statement $\Xi(A_e) \cap U \ne \emptyset$ is not provable in $\mathfrak T$. If $\N \setminus S$ is not computably enumerable, then there exists $e \notin S$ such that the true statement $\Xi(A_e) \cap \overline U = \emptyset$ is not provable in $\mathfrak T$.
\end{proposition}

\begin{proof}
We prove the first claim. Assume, to the contrary, that $\Xi(A_e)\cap U\ne\emptyset$ is provable in $\mathfrak T$ for every $e\in S$. Since $\mathfrak T$ is effectively axiomatized, its proofs can be enumerated, and the sentence expressing $\Xi(A_e)\cap U\ne\emptyset$ can be constructed effectively from $e$. Dovetail the proof enumeration with the indices: at stage $s$, inspect the first $s$ candidate proofs and the first $s$ indices, and list every $e\le s$ for which one of these proofs has the corresponding inclusion sentence as its conclusion. Arithmetical soundness prevents any false sentence from entering the list. If $e\notin S$, then $\Xi(A_e)\cap\overline U=\emptyset$, so the corresponding inclusion sentence is false. Every listed index therefore belongs to $S$, while the contrary assumption ensures that every element of $S$ eventually enters the list. This would make $S$ computably enumerable, a contradiction. The second claim follows by the same argument applied to the exclusion statements $\Xi(A_e)\cap\overline U=\emptyset$ and the set $\mathbb N\setminus S$.
\end{proof}

\subsection{A Concrete Result}

We apply \cref{prop:non_prov} to the one-dimensional Schr\"odinger families constructed above.

\begin{theorem}
\label{thm:main_non_prov}
Let $\mathfrak T$ be an arithmetically sound, effectively axiomatized theory. Let $I=(\alpha,\beta)\subseteq\mathbb R$ be a non-empty open interval with $\alpha,\beta\in\mathbb Q$, and let $\diamond\in\{\mathrm{pp},\mathrm{ac},\mathrm{sc}\}$. Then there exist self-adjoint one-dimensional Schr\"odinger operators $ H_\diamond^+,\ H_\diamond^- \in \mathcal O_M^\infty $ with real-valued $C^\infty$ potentials, given by source codes in $\mathsf{Code}_M^\infty$, such that the statement $ \spec_\diamond(H_\diamond^+)\cap I\ne\emptyset $ is true but not provable in $\mathfrak T$, and the statement $ \spec_\diamond(H_\diamond^-)\cap\overline I=\emptyset $ is true but not provable in $\mathfrak T$.
\end{theorem}

\begin{proof}
The dichotomies in \cref{sec:pp,sec:ac,sec:sc} use fixed test regions. An affine normalization moves each region to an arbitrary rational interval. Let $H=-\frac{\mathrm d^2}{\mathrm dx^2}+V(x)$ be a one-dimensional Schr\"odinger operator. For $a>0$ and $b\in\mathbb R$, set $q=\sqrt a$, and define
$$
    H^{(a,b)}
    =
    -\frac{\mathrm d^2}{\mathrm dx^2}
    +
    aV(qx)
    +
    b.
$$
Let $U_q:L^2(\mathbb R)\to L^2(\mathbb R)$ be the unitary map $ (U_q f)(x)=q^{-1/2}f(x/q). $ Then $H^{(a,b)}=a\,U_q^{-1}HU_q+b$. Hence, for each spectral type $\diamond\in\{\mathrm{pp},\mathrm{ac},\mathrm{sc}\}$, $ \spec_\diamond(H^{(a,b)}) = a\,\spec_\diamond(H)+b. $ Unitary equivalence preserves spectral type, and the affine map $E\mapsto aE+b$ sends the pure point, absolutely continuous, and singular continuous parts of the spectral measure to the corresponding parts for $H^{(a,b)}$. Moreover, if $V$ is given by a smooth source code, then so is $x\mapsto aV(qx)+b$, uniformly in the source code for $V$ and in computable $a>0$ and $b$. Its derivatives satisfy
$$
    \frac{\mathrm d^k}{\mathrm dx^k}
    \bigl(aV(qx)+b\bigr)
    =
    a q^k V^{(k)}(qx)
    +
    b\1_{k=0}.
$$
Thus, derivative values can be obtained by querying the original source code for $V^{(k)}$ at $qx$. If the source code is formulated using rational inputs, the value at the computable point $qx$ is obtained by rational approximation, using the local derivative bounds supplied by the source code. Similarly, if the original code gives a bound $C_{k,R'}$ for $V^{(k)}$ on $[-R',R']$, and $R'$ is any rational number with $R'>qR$, then
$$
    \sup_{|x|\le R}
    \left|
        \frac{\mathrm d^k}{\mathrm dx^k}
        \bigl(aV(qx)+b\bigr)
    \right|
    \le
    a q^k C_{k,R'}+|b|\1_{k=0}.
$$
This gives the required transformed smooth source code.

Now fix $\diamond\in\{\mathrm{pp},\mathrm{ac},\mathrm{sc}\}$. The lower-bound construction for this spectral type gives a computable Markov family $\{H_e^\diamond\}_{e\in\mathbb N}\subseteq\mathcal O_M^\infty$, a set $S_\diamond\subseteq\mathbb N$, and a fixed open ball $U_\diamond=D_{r_\diamond}(c_\diamond)\subseteq\mathbb C$ with $r_\diamond,c_\diamond\in\mathbb Q$, such that
$$
    e\in S_\diamond
    \implies
    \spec_\diamond(H_e^\diamond)\cap U_\diamond\ne\emptyset,
\qquad
    e\notin S_\diamond
    \implies
    \spec_\diamond(H_e^\diamond)\cap\overline{U_\diamond}=\emptyset.
$$
Here $S_{\mathrm{pp}}=S_{\mathrm{ac}}=\mathbb N\setminus\Tot$ and $S_{\mathrm{sc}}=\Cof$. For $\diamond=\mathrm{pp}$, this is the dichotomy established in \cref{sec:pp}; for $\diamond=\mathrm{ac}$, this is the dichotomy established in \cref{sec:ac}; and for $\diamond=\mathrm{sc}$, this is the dichotomy established in \cref{subsec:dichot}.

Let
$$
    U_I
    =
    D_{(\beta-\alpha)/2}\left(\frac{\alpha+\beta}{2}\right),\qquad
    a_\diamond
    =
    \frac{(\beta-\alpha)/2}{r_\diamond}
    \in\mathbb{Q}_{>0},
    \qquad
    b_\diamond
    =
    \frac{\alpha+\beta}{2}
    -
    a_\diamond c_\diamond\in\mathbb{Q}.
$$
Then $a_\diamond U_\diamond+b_\diamond=U_I$. Define $\widetilde H_e^\diamond=(H_e^\diamond)^{(a_\diamond,b_\diamond)}$. By the preceding normalization, the map $e\mapsto \widetilde H_e^\diamond$ is again a computable Markov family in $\mathcal O_M^\infty$. Moreover,
$$
    e\in S_\diamond
    \implies
    \spec_\diamond(\widetilde H_e^\diamond)\cap U_I\ne\emptyset,
\qquad
    e\notin S_\diamond
    \implies
    \spec_\diamond(\widetilde H_e^\diamond)\cap\overline{U_I}=\emptyset.
$$
Since all operators involved are self-adjoint, their spectral-type sets are subsets of $\mathbb R$. Therefore $\spec_\diamond(\widetilde H_e^\diamond)\cap U_I\ne\emptyset$ is equivalent to $\spec_\diamond(\widetilde H_e^\diamond)\cap I\ne\emptyset$, and $\spec_\diamond(\widetilde H_e^\diamond)\cap\overline{U_I}=\emptyset$ is equivalent to $\spec_\diamond(\widetilde H_e^\diamond)\cap\overline I=\emptyset$.

Neither $S_\diamond$ nor $\mathbb N\setminus S_\diamond$ is computably enumerable. For $\diamond=\mathrm{pp},\mathrm{ac}$, this follows from the fact that $\Tot\notin\Delta_2^0$. Indeed, if either $\Tot$ or its complement were computably enumerable, then $\Tot$ would be $\Delta_2^0$. For $\diamond=\mathrm{sc}$, this follows from the fact that $\Cof\notin\Delta_3^0$, since computably enumerable and co-computably enumerable sets are both $\Delta_3^0$.

Applying \cref{prop:non_prov} to the transformed family gives an index $e^+\in S_\diamond$ such that the true statement $\spec_\diamond(\widetilde H_{e^+}^\diamond)\cap I\ne\emptyset$ is not provable in $\mathfrak T$, and an index $e^-\notin S_\diamond$ such that the true statement $\spec_\diamond(\widetilde H_{e^-}^\diamond)\cap\overline I=\emptyset$ is not provable in $\mathfrak T$. Taking $H_\diamond^+=\widetilde H_{e^+}^\diamond$ and $H_\diamond^-=\widetilde H_{e^-}^\diamond$ proves the theorem.
\end{proof}

Taking $\mathfrak T=\mathsf{ZFC}$, with arithmetic interpreted in the usual set-theoretic way, we obtain the following conditional statement: if $\mathsf{ZFC}$ is arithmetically sound, then the spectral statements above are not provable in $\mathsf{ZFC}$. Note that arithmetical soundness implies consistency, but is strictly stronger than consistency, since consistency alone does not prevent a theory from proving false arithmetical statements.

\subsection{A Hierarchy of Non-Provability}

The proof of \cref{thm:main_non_prov} uses only that neither the relevant set $S$ nor its complement is computably enumerable. The lower bounds are stronger: the pure point and absolutely continuous problems lie at the $\Sigma_2^M$ level, whereas the singular continuous problem lies at the $\Sigma_3^M$ level. They yield a corresponding hierarchy of non-provability.

We first specify the lower-level arithmetical information available to the theory. Suppose
$$
    S=\{e\in\mathbb N:(\exists m)\,\psi(e,m)\},
$$
where $\psi$ is a $\Pi_{n-1}^0$ formula. Define the instance-truth oracle for this particular predicate by
$$
    \mathsf{Truth}_\psi(e,m)
    =
    \begin{cases}
        1, & \mathbb N\models\psi(e,m),\\
        0, & \mathbb N\not\models\psi(e,m).
    \end{cases}
$$
Thus $\mathsf{Truth}_\psi$ decides every instance of the fixed predicate $\psi$; it is not being identified with an oracle for all true $\Pi_{n-1}^0$ sentences. Denote the corresponding oracle theory by $\mathfrak T+\mathsf{Truth}_\psi$: proofs are ordinary proofs from $\mathfrak T$, but with access to $\mathsf{Truth}_\psi$. Equivalently, the theorems of this theory are computably enumerable relative to $\mathsf{Truth}_\psi$. We call such an oracle theory arithmetically sound if every arithmetical sentence in its theorem set is true in the standard model of $\mathbb N$.

\begin{theorem}
\label{thm:non_prov_hier}
Let $\mathfrak T$ be an effectively axiomatized theory extending Peano arithmetic. Let $I=(\alpha,\beta)\subseteq\mathbb R$ be a non-empty open interval with $\alpha,\beta\in\mathbb Q$.

For $\diamond\in\{\mathrm{pp},\mathrm{ac}\}$, let $S_2=\mathbb N\setminus\Tot$. Choose a representation
$$
    e\in S_2
    \quad\Longleftrightarrow\quad
    (\exists m)\,\psi_2(e,m),
$$
with $\psi_2\in\Pi_1^0$. If $\mathfrak T+\mathsf{Truth}_{\psi_2}$ is arithmetically sound, then, for each $\diamond\in\{\mathrm{pp},\mathrm{ac}\}$, there exists a self-adjoint one-dimensional Schr\"odinger operator $H_\diamond\in\mathcal O_M^\infty$, with real-valued $C^\infty$ potential and given by a source code in $\mathsf{Code}_M^\infty$, such that the statement $\spec_\diamond(H_\diamond)\cap \overline I=\emptyset$ is true but not provable in $\mathfrak T+\mathsf{Truth}_{\psi_2}$.

For $\diamond=\mathrm{sc}$, let $S_3=\Cof$. Choose a representation
$$
    e\in S_3
    \quad\Longleftrightarrow\quad
    (\exists m)\,\psi_3(e,m),
$$
with $\psi_3\in\Pi_2^0$. If $\mathfrak T+\mathsf{Truth}_{\psi_3}$ is arithmetically sound, then there exists a self-adjoint one-dimensional Schr\"odinger operator $H_{\mathrm{sc}}\in\mathcal O_M^\infty$, with real-valued $C^\infty$ potential and given by a source code in $\mathsf{Code}_M^\infty$, such that the statement $\specsc(H_{\mathrm{sc}})\cap \overline I=\emptyset$ is true but not provable in $\mathfrak T+\mathsf{Truth}_{\psi_3}$.
\end{theorem}

This gives a predicate-relative hierarchy of non-provability. For pure point and absolutely continuous spectrum, the conclusion persists after adjoining the truth values of all instances of the fixed $\Pi_1^0$ predicate $\psi_2$ used in the theorem. For singular continuous spectrum, it persists after adjoining the truth values of all instances of the fixed $\Pi_2^0$ predicate $\psi_3$. These are predicate-specific instance-truth oracles; the theorem does not assert that either oracle decides every true sentence at the corresponding level of the arithmetical hierarchy.

\begin{proof}[Proof of \cref{thm:non_prov_hier}]
We treat the three spectral types uniformly. Let $S$ be either $\mathbb N\setminus\Tot$ or $\Cof$, and suppose $S=\{e\in\mathbb N:(\exists m)\,\psi(e,m)\}$, where $\psi\in\Pi_{n-1}^0$. In the pure point and absolutely continuous cases $n=2$, while in the singular continuous case $n=3$.

The relevant lower-bound construction supplies a computable Markov family of Schr\"odinger operators $H_e=-\frac{\mathrm d^2}{\mathrm dx^2}+V_e$ and a fixed rational open interval $J=(\gamma,\delta)\subseteq\mathbb R$ such that
$$
    e\in S
    \implies
    \spec_\diamond(H_e)\cap J\ne\emptyset,
\qquad
    e\notin S
    \implies
    \spec_\diamond(H_e)\cap \overline J=\emptyset.
$$
For $\diamond=\mathrm{pp}$ and $\diamond=\mathrm{ac}$, this follows from the dichotomies established in \cref{sec:pp} and \cref{sec:ac}. For $\diamond=\mathrm{sc}$, this follows from \cref{subsec:dichot}, taking $J=D_{1/16}(3/2)\cap\mathbb R=(23/16,25/16)$.

Choose the affine parameters
$$
    a=\frac{\beta-\alpha}{\delta-\gamma}\in\mathbb{Q}_{>0},
    \qquad
    b=\alpha-a\gamma\in\mathbb{Q}.
$$
Then $aJ+b=I$ and $a\overline J+b=\overline I$. Let $q=\sqrt a$, and define
$$
    \widetilde H_e
    =
    -\frac{\mathrm d^2}{\mathrm dx^2}
    +
    a V_e(qx)
    +
    b.
$$
As in the proof of \cref{thm:main_non_prov}, this transformation is effective on smooth source codes, uniformly in $e$, and $\spec_\diamond(\widetilde H_e)=a\,\spec_\diamond(H_e)+b$. Hence
$$
    e\in S
    \implies
    \spec_\diamond(\widetilde H_e)\cap I\ne\emptyset,
\qquad
    e\notin S
    \implies
    \spec_\diamond(\widetilde H_e)\cap \overline I=\emptyset.
$$
Assume, to the contrary, that every true exclusion statement $\spec_\diamond(\widetilde H_e)\cap\overline I=\emptyset$ with $e\notin S$ is provable in $\mathfrak T+\mathsf{Truth}_\psi$. Then $S$ would be computable relative to $\mathsf{Truth}_\psi$.

Given $e$, run two searches in parallel. The first search asks $\mathsf{Truth}_\psi$ whether $\psi(e,m)$ holds, for $m=1,2,\ldots$. If the oracle ever answers yes, then $e\in S$. The second search enumerates proofs in the theory $\mathfrak T+\mathsf{Truth}_\psi$, looking for a proof of $\spec_\diamond(\widetilde H_e)\cap \overline I=\emptyset$. By the contradiction assumption, this second search halts whenever $e\notin S$. By arithmetical soundness, the second search cannot halt when $e\in S$, because in that case $\spec_\diamond(\widetilde H_e)\cap I\ne\emptyset$, and $I\subseteq \overline I$. Thus, the two parallel searches decide membership in $S$ relative to the oracle $\mathsf{Truth}_\psi$.

Since $\psi\in\Pi_{n-1}^0$, the oracle $\mathsf{Truth}_\psi$ is computable from $0^{(n-1)}$. Every set computable relative to this oracle therefore lies in $\Delta_n^0$, so the preceding decision procedure would imply $S\in\Delta_n^0$. This contradicts \cref{prop:complete_class}: $\mathbb N\setminus\Tot$ is $\Sigma_2^0$-complete and not in $\Delta_2^0$, whereas $\Cof$ is $\Sigma_3^0$-complete and not in $\Delta_3^0$. Hence some true exclusion statement of the asserted form is not provable in the corresponding oracle theory.
\end{proof}

\section*{Statements and Declarations}

\paragraph{Data availability.} No datasets were generated or analyzed during the current study.

\paragraph{Competing interests.} The authors declare that they have no competing interests.

\bibliographystyle{abbrv}
\bibliography{spectral_type.bib}
\end{document}